\documentclass[11pt,reqno]{article} \usepackage{authblk}

\usepackage{graphicx}
\usepackage[top=1in, bottom=1in, left=1in, right=1in] {geometry}
\usepackage{amscd,amsmath,amstext,amsfonts,amsbsy,
amssymb,amsthm,mathrsfs,float,latexsym}
\usepackage{multicol,graphicx,array,multirow,color,lineno}
\usepackage{sidecap,url,fancybox,fancyhdr}
\usepackage[colorlinks=true,linkcolor=blue,citecolor=magenta]{hyperref}

\usepackage{subcaption}

\usepackage{enumitem}
\newcommand{\eps}{\varepsilon}

\newcommand{\di}{\nabla \cdot }
\newcommand{\Tm}{T_{\max}}

\newcommand{\R}{\mathbb R}

\renewcommand{\epsilon}{\varepsilon}

\newcommand{\LO}[1]{L^{#1}(\Omega)}

\newcommand{\intO}{\int_{\Omega}}
\newcommand{\intQTm}{\iint_{ Q_{\Tm}}}
\newcommand{\intQt}{\iint_{Q_{t}}}

\newcommand{\intQT}{\iint_{Q_{T}}}

\usepackage{comment}

\newtheorem{theorem}{{\bf Theorem}}[section]
\theoremstyle{definition} \newtheorem{definition}{\bf Definition}[section]
\theoremstyle{assumption} \newtheorem{assumption}{\bf Assumption}[section]

\theoremstyle{plain}
\newtheorem{lemma}{Lemma}[section]

\newtheorem{remark}{Remark}[section]

\title{\bf \Large Rigorous fast signal diffusion limit and  convergence rates with the initial layer effect in a competitive chemotaxis system}

\author{Cordula Reisch$^{a,b}$\footnote{Email: c.reisch@tu-bs.de}, Bao-Ngoc Tran$^{b,c}$\footnote{Email: bao-ngoc.tran@uni-graz.at (Corresponding author)}, Juan Yang$^{d,e}$\footnote{Email: yangjuan12@sjtu.edu.cn}}

\date{}  

\begin{document}

\maketitle

\vspace*{-1cm}

\begin{center}
{\small
$^{a}$Institute for Partial Differential Equations, Technische Universit\"at Braunschweig, \\ 38106 Braunschweig,  Germany \vspace{0.15cm} \\
$^{b}$Department of Mathematics and Scientific Computing, University of Graz, \\ Heinrichstrasse 36, 8010 Graz, Austria \vspace{0.15cm}\\ 
$^{c}$Department of Mathematics, Faculty of Science, Nong Lam University,    Ho Chi Minh City, Vietnam \vspace{0.15cm}\\
$^{d}$School of Mathematical Sciences, Shanghai Jiao Tong University, Shanghai 200240, China \vspace{0.15cm}\\
$^{e}$School of Mathematics and Statistics, Lanzhou University, Lanzhou, 730000, China 
}
\end{center}

\begin{abstract} We study a chemotaxis system that includes two competitive prey and one predator species in a two-dimensional domain, where the movement of prey (resp. predators) is driven by chemicals secreted by predators (resp. prey), called mutually repulsive (resp. mutually attractive) chemotactic effect. The kinetics for all species are chosen according to the competitive Lotka–Volterra equations for prey and to a Holling type functional response for the predator.  Under the biologically relevant scenario that the chemicals diffuse much faster than the individual diffusion of all species and a suitable re-scaling, equations for chemical concentrations are parabolic with slow evolution depending on the relaxation time $0<\varepsilon\ll 1$. The first main result shows the global existence of a unique classical solution to the system for each~$\varepsilon$.
Second, we study rigorously the so-called fast signal diffusion limit, passing from the system including parabolic equations with the slow evolution of the chemical concentrations to elliptic equations for the chemical concentrations, i.e. the limit as $\varepsilon \to 0$. This explains why elliptic equations can be proposed for chemical concentration instead of parabolic ones with slow evolution. Third, the $L^\infty$-in-time  convergence rates for the fast signal diffusion limit are estimated, where the effect of the initial layer is carefully treated. Finally, the differences between the systems with and without the slow evolution, and between the systems with one or two prey, as well as their dynamics, are discussed numerically.        \\



\end{abstract}

\tableofcontents
 
\section{Introduction}\label{sec:1}

\subsection{Food chain chemotaxis model}\label{sec:1_1}

  Chemotaxis 
  is well-known and plays a crucial role in many biological systems in which the species' movement is biased along concentration gradients of chemical stimuli. In the last decade, chemotaxis systems 
have gained much attention in mathematical analysis, such as, in two dimensions,   global existence results for a 
parabolic-elliptic system 
can be found in \cite{zheng2017global}, or for a
parabolic-parabolic system 
in \cite{black2017global}. For a higher-dimensional case, the global existence was established in \cite{gnanasekaran2024global,zheng2017global} by imposing restrictive conditions on parameters. Various results regarding global existence, large-time behaviour, or blowing-up of solutions have been studied   in   \cite{tao2015boundedness,
pan2020boundedness,tu2021phenomenon,
yang2022global,gnanasekaran2024global} and references therein.  
Recently, 
a chemotaxis system including two competitive prey and one predator species has been studied in  \cite{amorim2023global} with a result on global existence and in \cite{burger2020numerical} with numerical analysis, which will be studied in this paper in the context of fast signalling.  
Specifically, we consider in $Q_\infty=\Omega \times (0,\infty)$ the system
\begin{align}
\label{System.U}
\left\{ \begin{array}{llll}
\partial_\tau \widetilde u_1  - \widetilde d_1\Delta \widetilde u_1  - \widetilde \chi_1 \di (\widetilde u_1 \nabla \widetilde v_3)  &\hspace*{-0.2cm}=\hspace*{-0.2cm}& \widetilde f_1(\widetilde u_1,\widetilde u_2,\widetilde u_3) , \vspace*{0.05cm} \\
\partial_\tau \widetilde u_2  - \widetilde d_2\Delta \widetilde u_2  - \widetilde \chi_2 \di (\widetilde u_2 \nabla \widetilde v_3) &\hspace*{-0.2cm}=\hspace*{-0.2cm}& \widetilde f_2(\widetilde u_1,\widetilde u_2,\widetilde u_3)  , \vspace*{0.05cm} \\
\partial_\tau \widetilde u_3  - \widetilde d_3\Delta \widetilde u_3  +  \sum_{i=1}^2 \widetilde \chi_{3i} \di (\widetilde u_3 \nabla \widetilde v_i )  &\hspace*{-0.2cm}=\hspace*{-0.2cm}& \widetilde f_3(\widetilde u_1,\widetilde u_2,\widetilde u_3) ,
\end{array}
\right.
\end{align}
where $\Omega\subset \mathbb{R}^2$ is a bounded domain with a sufficiently smooth boundary $\Gamma$, and at a time $\tau\ge 0$ and a position $x\in \Omega$, the functions $\widetilde u_i=\widetilde u_i(x,\tau)$, $i=1,2,$ stand for the densities of
two competitive prey species, and $\widetilde u_3=\widetilde u_3(x,\tau)$ for the density of a predator species. Besides the self-diffusion with the coefficients $\widetilde d_1,\widetilde d_2,\widetilde d_3>0$, the movement of preys (resp. predators) is also driven by chemicals secreted by predators (resp. prey) with the coefficients $\widetilde \chi_{1}, \widetilde \chi_{2}>0$ (resp. $\widetilde \chi_{31}, \widetilde \chi_{32}>0$), called chemotactic movement. The function $\widetilde v_i$ denotes the concentration of the chemical secreted by the species $\widetilde u_i$, $i=1,2,3$. The preys tend to move in the direction such that the predator concentration $\widetilde v_3$ is decreasing, described by  $\widetilde \chi_i\di(\widetilde u_i\nabla \widetilde v_3)$, $i=1,2,$   and called mutually repulsive chemotactic effects. In contrast,  the predators tend to move in a direction such that the prey concentrations are increasing, described by  $-\di (\widetilde u_3 ( \widetilde \chi_{31} \nabla \widetilde v_1 + \widetilde \chi_{32}  \nabla \widetilde v_2))$   and called mutually attractive chemotactic effects. We assume that the chemical concentrations are governed by the following equations\footnote{The case in which all chemical concentration equations are parabolic will be discussed in Section ~\ref{Sec.Extention-Outlook}.}     
\begin{align}
\label{System.V}
\left\{\begin{array}{lll}
\hspace*{-0.3cm} & - \widetilde\lambda_i \Delta \widetilde v_i + \widetilde\mu_i \widetilde  v_i \hspace*{-0.3cm} & = \widetilde u_i,   \\ 
\partial_\tau \widetilde v_3 \hspace*{-0.3cm} & - \widetilde \lambda_3 \Delta \widetilde v_3 + \widetilde \mu_3 \widetilde  v_3 \hspace*{-0.3cm} & = \widetilde u_3,   
\end{array}\right. \quad i=1,2,
\end{align}
where $\widetilde \lambda_i$ are the diffusion coefficients and $\widetilde \mu_i$ the decay rates of the $i$-th  chemical, for $i=1,2,3$. The kinetics for all species are chosen according to the competitive Lotka–Volterra equations for the preys and to a Holling-type functional response for the predator, as follows
\begin{align*}
\widetilde f_1(\widetilde u_1,\widetilde u_2,\widetilde u_3) =&\, \widetilde \alpha_1 \widetilde u_1(1-\widetilde u_1 - \beta_1  \widetilde u_2) -  \dfrac{\widetilde m_1 \widetilde u_1}{\eta_1+\widetilde u_1} \widetilde u_3,  \\
\widetilde f_2(\widetilde u_1,\widetilde u_2,\widetilde u_3) =&\, \widetilde \alpha_2 \widetilde u_2(1-\widetilde u_2 - \beta_2 \widetilde u_1) -  \dfrac{\widetilde m_2 \widetilde u_2}{\eta_2+\widetilde u_2} \widetilde u_3,  \\
\widetilde f_3(\widetilde u_1,\widetilde u_2,\widetilde u_3) =&\, \left( \gamma_1 \dfrac{\widetilde m_1 \widetilde u_1}{\eta_1+\widetilde u_1} + \gamma_2 \dfrac{\widetilde m_2 \widetilde u_2}{\eta_2+\widetilde u_2} - \widetilde k \right) \widetilde u_3 - \widetilde l \, \widetilde u_3^2,
\end{align*}
where $\widetilde \alpha_1,\widetilde \alpha_2$ are biotic potentials;
 $ \beta_1, \beta_2$  are coefficients of inter-specific competition between two prey species;
 $\widetilde m_1,\widetilde m_2$ are predation
coefficients;
 $\eta_1,\eta_2$ are half-saturation constants;
 $\gamma_1,\gamma_2$ are conversion factors;
 $\widetilde k$ and $\widetilde l$ are the natural death rates of the predator and the intra-specific competition among predators, respectively. 
 The system is subjected to the no-flux boundary condition
\begin{align}
\label{Condition.boundary}
\nabla \widetilde u_i \cdot \nu = \nabla \widetilde  v_i \cdot \nu = 0  & \quad \text{on} \quad \Gamma_\infty := \Gamma \times (0,\infty),  \quad i=1,2,3,
\end{align}
and the initial condition
\begin{align}
\label{Condition.initial}
\widetilde u_i(0)=u_{i0}, \quad \widetilde v_3(0)=v_{30}  & \quad \text{on} \quad  \Omega,  \quad i=1,2,3,
\end{align}
where  $\nu$ is the unit outer normal vector to $\Gamma$, and     $u_{10},u_{20},u_{30},v_{30}$ are given functions.    
 
\subsection{Fast signal diffusion limit} \label{sec:1_2}

In chemotaxis systems, it has usually been argued to simplify \textit{at the formal level} a parabolic-parabolic chemotaxis system to the respective parabolic-elliptic one, see, e.g. in \cite{corrias2004global,kiselev2022chemotaxis}. In the technical aspect, this simplification accompanies substantial benefits for mathematical analysis with a more enhanced insight into parabolic-elliptic systems. For example, the detection of the solution explosion was found early \cite{jager1992explosions} while more significant effort is needed for the corresponding fully parabolic system \cite{mizoguchi2014nondegeneracy}. Moreover, the possibility of various tools has enabled the uncovering of further qualitative properties of the system \cite{kiselev2012biomixing,carrillo2019nonlinear}. The limits passing from the parabolic-parabolic system to the parabolic-elliptic one can be termed \textit{fast signal diffusion limit} that falls into the topic of singular limits in PDEs, see e.g. \cite{mizukami2019fast,wang2019fast,ogawa2023maximal}, which will be the main focus of this paper.  
We consider the biological scenario where the chemicals diffuse much faster than the individual diffusion of all species, or more precisely,
\begin{align}
\varepsilon:= \frac{\max\{\widetilde d_1,\widetilde d_2,\widetilde d_3\}}{\min\{\widetilde \lambda_1,\widetilde \lambda_2,\widetilde \lambda_3\}} \ll 1.
\label{Assumption.QSS}
\end{align}
The re-scaling
\begin{align*}
t:= \max \{\widetilde d_1,\widetilde d_2,\widetilde d_3\} \tau, \quad (u_i^\varepsilon(t),v_i^\varepsilon(t)):=(\widetilde u_i(\tau),\widetilde v_i(\tau)), \quad i=1,2,3,
\end{align*}
recasts the system \eqref{System.U}--\eqref{Condition.initial} as 
 \begin{align}
\label{System.UV}
\left\{ \begin{array}{llll}
\hspace*{0.15cm} \partial_t u_1^\varepsilon  - d_1 \Delta u_1^\varepsilon  - \chi_1 \di ( u_1^\varepsilon \nabla v_3^\varepsilon) &\hspace*{-0.2cm}=\hspace*{-0.2cm}& f_1( u_1^\varepsilon, u_2^\varepsilon, u_3^\varepsilon) ,  \\
\hspace*{0.15cm} \partial_t u_2^\varepsilon - d_2\Delta u_2^\varepsilon  - \chi_2 \di ( u_2^\varepsilon \nabla  v_3^\varepsilon) &\hspace*{-0.2cm}=\hspace*{-0.2cm}& f_2( u_1^\varepsilon, u_2^\varepsilon, u_3^\varepsilon)  , \\
\hspace*{0.15cm} \partial_t u_3^\varepsilon -  d_3\Delta u_3^\varepsilon +  \sum_{i=1}^2 \chi_{3i} \di ( u_3^\varepsilon \nabla v_i^\varepsilon)  &\hspace*{-0.2cm}=\hspace*{-0.2cm}&  f_3( u_1^\varepsilon, u_2^\varepsilon, u_3^\varepsilon) ,  \\
\hspace*{0.95cm} -\, \lambda_i \Delta v_i^\varepsilon + \mu_i v_i^\varepsilon \hspace*{0.1cm} = \zeta_i u_i^\varepsilon, \,\, i=1,2,   \\
\varepsilon \partial_t v_3^\varepsilon - \lambda_3 \Delta v_3^\varepsilon + \mu_3 v_3^\varepsilon = \zeta_3 u_3^\varepsilon,  
\end{array}
\right.
\end{align}
equipped with the boundary - initial conditions 
\begin{align}
\label{Condition.Boun-Init.Epsilon}
(\nabla u_i^\varepsilon \cdot \nu, \nabla v_i^\varepsilon \cdot \nu)|_{\Gamma_\infty} = 0, \quad 
(u_i^\varepsilon(0),v_3^\varepsilon(0))|_{\Omega}=(u_{i0},v_{30}), \quad i=1,2,3,
\end{align} 
where   
\begin{align}
\begin{gathered}
(d_i,\chi_i,\chi_{3j},f_i) :=   \frac{1}{\max \{\widetilde d_1,\widetilde d_2,\widetilde d_3\}} (\widetilde d_i,\widetilde \chi_i,\widetilde \chi_{3j},\widetilde f_i), \\ 
(\lambda_j, \mu_j,\zeta_j):=   (\widetilde \lambda_j, \widetilde \mu_j,1) , \quad
(\lambda_3, \mu_3,\zeta_3):=   \frac{1}{\min\{\widetilde \lambda_1,\widetilde \lambda_2,\widetilde \lambda_3\}} (\widetilde \lambda_3,\widetilde \mu_3,1),
\end{gathered}
\label{ScaledParameters}
\end{align}
for $i=1,2,3$, $j=1,2$. 
Here,  the $\varepsilon$-superscript is used to emphasize the dependence of solutions on $\varepsilon$. Due to the assumption \eqref{Assumption.QSS}, it is relevant to consider the limit as $\varepsilon\to0$.    
Formally, we expect 
\begin{align}
(u_i^\varepsilon,v_i^\varepsilon) \to (u_i,v_i) \quad \text{and} \quad \varepsilon \partial_t v_3^\varepsilon \to 0,
\label{FormalExpectation}
\end{align}
and therefore, the system \eqref{System.UV} is reduced to 
\begin{align}
\label{System.UV.Limit}
\left\{ \begin{array}{llll}
\hspace*{0.15cm} \partial_t u_1  - d_1 \Delta u_1  - \chi_1 \di ( u_1 \nabla v_3 ) &\hspace*{-0.2cm}=\hspace*{-0.2cm}& f_1( u_1 , u_2 , u_3 ) ,   \\
\hspace*{0.15cm} \partial_t u_2 - d_2\Delta u_2  - \chi_2 \di ( u_2 \nabla  v_3 ) &\hspace*{-0.2cm}=\hspace*{-0.2cm}& f_2( u_1 , u_2 , u_3 )  ,   \\
\hspace*{0.15cm} \partial_t u_3 -  d_3\Delta u_3 +  \sum_{i=1}^2 \chi_{3i} \di ( u_3 \nabla v_i )  &\hspace*{-0.2cm}=\hspace*{-0.2cm}&  f_3( u_1 , u_2 , u_3 ) ,  \\
  -\, \lambda_i \Delta v_i  + \mu_i v_i \hspace*{0.1cm} = \zeta_i u_i , \,\, i=1,2,3,  
\end{array}
\right.
\end{align}
equipped with the boundary - initial conditions
\begin{align}
\label{Condition.Boun-Init.Limit}
(\nabla u_i \cdot \nu, \nabla v_i \cdot \nu)|_{\Gamma_\infty} = 0, \quad 
 u_i (0)|_{\Omega}= u_{i0}, \quad i=1,2,3 .
\end{align}
The limiting system \eqref{System.UV.Limit}-\eqref{Condition.Boun-Init.Limit} has been studied in \cite{burger2020numerical,
amorim2023global} with results on global existence and finite volume scheme, where, under the biological situation \eqref{Assumption.QSS},    the equations for chemical concentrations are formally proposed to be elliptic.

\medskip

Although it was commonly considered early on, the mathematical analysis, which rigorously justifies the simplification, has only been studied in the last few years. In \cite{mizukami2019fast}\footnote{See its arXiv version at arXiv:1711.04328, originally announced in November 2017.}, the author proposed a rigorous analysis that shows the convergence of the solution to the following parabolic-parabolic Keller-Segel system 
\begin{align}
\label{Ma:2018:eps}
\left\{ \begin{array}{llll}
\partial_t u_\lambda =  \Delta u_\lambda - \chi \nabla \cdot ( u_\lambda \nabla v_\lambda)  & \text{in } \Omega \times(0,\infty)  \\
\lambda \partial_t v_\lambda = \Delta v_\lambda - v_\lambda + u_\lambda & \text{in } \Omega \times(0,\infty),   \\ 	
(u_\lambda,v_\lambda)|_{t=0}=(u_{0},v_{0}) &  \text{on }  \Omega, 
\end{array} \right.
\end{align}
supplemented by the homogeneous Neumann boundary conditions, to that of the corresponding parabolic-elliptic version  
\begin{align}
\left\{ \begin{array}{llll}
\partial_t u  =  \Delta u - \chi \nabla \cdot (u  \nabla v )  & \text{in } \Omega \times(0,\infty)  \\
 \Delta v - v + u = 0 & \text{in } \Omega \times(0,\infty),  \\ 	
    u|_{t=0}=u_{0} &  \text{on }  \Omega, 
\end{array} \right.
\label{Ma:2018:0}
\end{align}
where, with regular initial state $u_0,v_0$ being of small size, it had been shown as $\lambda\to 0$ that  
\begin{align}
    \left\{\begin{array}{llll}
    u_\lambda \to u \text{ in } C_{\mathsf{loc}}(\overline{\Omega}\times[0,\infty)) , \\
    v_\lambda \to v \text{ in } C_{\mathsf{loc}}(\overline{\Omega}\times(0,\infty)) \cap L^2_{\mathsf{loc}}((0,\infty);W^{1,2}(\Omega)), 
    \end{array}\right.
    \label{Ma:2018:Conv}
\end{align} 
and the limit $(u,v)$ solves \eqref{Ma:2018:0} classically. In \cite{mizukami2018fast}, this result was extended to the chemotactic flux $u_\lambda S(v_\lambda) \nabla v_\lambda$ of strong sensitivity in the sense that $S\in C^{1+\vartheta}((0,\infty))$ for some $\vartheta\in(0,1)$ and $0\le S(v) \le \chi(a+v)^{-k}$ for  $a\ge 0$, $k>1$. 
While with the usual chemotactic flux $\chi u_\lambda \nabla v_\lambda$, the simplification of \eqref{Ma:2018:eps} with non-degenerate diffusion of porous medium type had been investigated in \cite{freitag2020fast}. An  fast signal diffusion  limit for another indirect signal chemotaxis system, describing the
movement of the mountain pine beetle in forest habitats, can be found in \cite{li2023convergence}. On the other hand, results in the setting of the whole domain $\Omega = \mathbb{R}^N$ can be found in    \cite{kurokiba2020singular,ogawa2023maximal}. 
In the other context that Keller-Segel systems are coupled with a fluid, the authors of \cite{wang2019fast} simplified the Keller-Segel-Navier-Stokes system 
\begin{align}
\label{Wang:2018:eps}
\left\{ \begin{array}{llll}
\partial_t n_\varepsilon + u_\varepsilon \cdot \nabla n_\varepsilon =  \Delta n_\varepsilon - \nabla \cdot ( n_\varepsilon S(x,n_\varepsilon,c_\varepsilon) \cdot \nabla c_\varepsilon) + f(x,n_\varepsilon,c_\varepsilon)  \\
\varepsilon \partial_t c_\varepsilon + u_\varepsilon \cdot \nabla c_\varepsilon = \Delta c_\varepsilon - c_\varepsilon + n_\varepsilon  ,   \\ 	
\partial_t u_\varepsilon + \kappa (u_\varepsilon \cdot \nabla) u_\varepsilon = \Delta u_\varepsilon + \nabla P_\varepsilon + n_\varepsilon  \nabla \phi, \; \kappa \in \mathbb R, \hfill \nabla \cdot u_\varepsilon = 0  , \\
(n_\varepsilon,c_\varepsilon,u_\varepsilon)|_{t=0} = (n_0,c_0,u_0),
\end{array} \right.
\end{align}
which is subjected to the no-flux boundary conditions for   $n_\varepsilon,c_\varepsilon$ and the homogeneous Dirichlet boundary condition for the fluid velocity $u_\varepsilon$, to the relative 
\begin{align}
\label{Wang:2018:eps}
\left\{ \begin{array}{llll}
\partial_t n + u \cdot \nabla n =  \Delta n - \nabla \cdot ( n S(x,n,c) \cdot \nabla c) + f(x,n,c)  \\
u \cdot \nabla c = \Delta c - c + n  ,   \\ 	
\partial_t u + \kappa (u \cdot \nabla) u = \Delta u + \nabla P + n   \nabla \phi, \hfill \nabla \cdot u = 0  , \\
(n,u)|_{t=0} = (n_0,u_0),
\end{array} \right.
\end{align}
according to the limit as $\varepsilon\to 0$. A conditional result, including an explicit criterion allowing for rigorous simplification, had been provided, which mainly entails the uniform-in-$\varepsilon$ regularity 
\begin{align*}
    \sup_{\varepsilon>0} \Big( \|\nabla c_\varepsilon\|_{L^p((0,T);L^q(\Omega))} + \|u_\varepsilon\|_{L^\infty((0,T);L^r(\Omega))} \Big) <\infty,
\end{align*}
for some $p,q,r$ such that  
\begin{align*}
    2<p\le \infty, \quad q>N, \quad r>\max\{2;N\} \; \text{ such that } \; \frac{1}{p} + \frac{N}{2q} < \frac{1}{2}. 
\end{align*}
Other results for Keller-Segel-(Navier-)Stokes system can be found in \cite{li2021convergence,
li2023stability,wu2025fast}. 

\medskip

{\bf Initial layer effect and critical manifold}: It is useful to point out that the initial conditions for the systems \eqref{System.UV}-\eqref{Condition.Boun-Init.Epsilon} and \eqref{System.UV.Limit}-\eqref{Condition.Boun-Init.Limit} are different. Therefore, the initial values $v_{3}(0)$ of the solution to \eqref{System.UV.Limit}-\eqref{Condition.Boun-Init.Limit} is not necessarily equal to the initial datum $v_{30}$ given in \eqref{System.UV}-\eqref{Condition.Boun-Init.Epsilon}. Since the limiting solution satisfies that $\lambda_3 \Delta v_3 - \mu_3 v_3 + u_3 = 0$, or in the notion of  dynamical systems, it stays in the so-called critical manifold 
\begin{align}
 \mathcal C:= \{ (\vartheta,\varrho)\in L^2(\Omega)\times H^{2}(\Omega):  \lambda_3 \Delta \varrho - \mu_3 \varrho + \vartheta = 0\},
\label{CriMani}
\end{align}
the mentioned difference yields a distance from the initial data of the $\varepsilon$-dependent system \eqref{System.UV}-\eqref{Condition.Boun-Init.Epsilon} to this manifold, called an initial layer.

\medskip

To the best of our knowledge, fast signal diffusion limits in competitive prey-predator systems have not been studied, especially the fast signal diffusion limit from \eqref{System.UV}-\eqref{Condition.Boun-Init.Epsilon} to \eqref{System.UV.Limit}-\eqref{Condition.Boun-Init.Limit}. Moreover, the effect of this layer on the accuracy (or more mathematically, the convergence rates) of the approximation for the parabolic-parabolic system \eqref{System.UV}-\eqref{Condition.Boun-Init.Epsilon} using its parabolic-elliptic relative \eqref{System.UV.Limit}-\eqref{Condition.Boun-Init.Limit}   has not been analysed and is not well understood in general fast signal diffusion limits.

\subsection{Main goal and organisation}

\textbf{Main goal and challenges}: Our main goals are to provide a rigorous analysis for the fast signal diffusion limit from \eqref{System.UV}-\eqref{Condition.Boun-Init.Epsilon} to \eqref{System.UV.Limit}-\eqref{Condition.Boun-Init.Limit} as well as for the initial layer effect in both analysis and numerical simulations. 

\medskip

Typically, fast signal diffusion limit problems include parabolic equations with slow evolutions of the form $\varepsilon \partial_t u_\varepsilon - L u_\varepsilon = f(u_\varepsilon)$, where $L$ is an elliptic operator and $f$ is an external source possibly being nonlinear. Usual analysis tools, such as the maximal regularity, can give the boundedness of the product $\varepsilon \partial_t u_\varepsilon$, but not the time derivative $\partial_t u_\varepsilon$. Hence, it seems challenging to utilise the Aubin-Lions lemma to get the convergence of $u_\varepsilon$ in a strong sense. Moreover, to have the smoothing effect via estimates for the heat semigroup, an essential-in-time boundedness of $f(u_\varepsilon)$ is needed, which is usually difficult before having the necessary a priori estimates. 

\medskip

In general, studying chemotaxis systems faces the issue of low regularity, and since many techniques for their global solvability are inapplicable to obtain the necessary compactness of the solutions that depend on the relaxation parameter, it is challenging to provide an analysis of fast signal diffusion limits and their convergence rates with the initial layer effect. 

\medskip

{\bf Organisation}: 
We first state our main results in Section~\ref{sec:2}. The global existence of a unique classical solution to \eqref{System.UV}-\eqref{Condition.Boun-Init.Epsilon} is proved in Section~\ref{sec-strong-solu}. We present a rigorous analysis for the fast signal diffusion limit passing from \eqref{System.UV}-\eqref{Condition.Boun-Init.Epsilon} to \eqref{System.UV.Limit}-\eqref{Condition.Boun-Init.Limit} in Section~\ref{sec:4}. Then, uniform-in-time convergence rates are studied in Section~\ref{sec:5}. In Section~\ref{sec:6}, differences between the systems with and without the slow evolution and between the systems with one or two preys are discussed by numerical simulations. The final section contains further discussions.

\medskip

{\bf notations}: We write $Q_T = \Omega\times(0,T)$ for $T>0$, and $L^p(\Omega)$ and $L^p(Q_T)$, $1\le p \le \infty$, for the  usual Lebesgue spaces. We use the same symbol, $C$, without distinction to denote positive constants that may change line by line, or even in the same line, which do not depend on $\eps>0$, but can depend on other fixed parameters. We also write $C_{T}$ to underline the dependency of $C$ on $T$. With $a \in X_+$, we mean that $a\in X$ and $a\ge 0$.

\section{Main results and key ideas}\label{sec:2}

Before stating our main results, we note that although the following assumption can be relaxed in some lemmas, it will be assumed throughout this paper to avoid confusion and ensure consistency. 

\begin{assumption}
\label{Assumption.Initial1} Assume that 
$u_{i0}\in C_+(\overline{\Omega})$, $v_{30} \in C_+^{2}(\overline{\Omega})$ with the compatibility conditions $ \nabla v_{3i} \cdot \nu=0$ on $\Gamma$ for $i=1,2,3$. 
\end{assumption}

For studying the aforementioned fast signal diffusion limit, the global existence of a unique classical solution to Problem \eqref{System.UV}-\eqref{Condition.Boun-Init.Epsilon} for each $\varepsilon>0$, in the sense of Definition~\ref{Definition.Solution},   is first required. By standard arguments of fixed point theorems, the existence of a local classical, non-negative solution $(u_i^\varepsilon,v_i^\varepsilon)$ can be accomplished up to the maximal time  $0<T_{\max}\le \infty$, 
for example, see \cite{winkler2010boundedness}. To extend the maximal time to be global, as presented in Theorem \ref{Theo.GlobalClassicalSol},  we assume $\Tm<\infty$ and prove that $(u_i^\varepsilon,v_i^\varepsilon)$ belongs to $L^\infty(Q_{T_{\max}})$, explained as follows. 

\begin{itemize}
    \item \textit{Improved regularity}: The homogeneous Neumann boundary conditions  ensure that $(u_1^\varepsilon,u_2^\varepsilon,u_3^\varepsilon) \in L^\infty(0,T;L^1(\Omega)) \cap L^{2}(Q_T)$. However, we can show $(u_1^\varepsilon,u_2^\varepsilon,u_3^\varepsilon) \in L^{2+\theta}(Q_T)$ for some $\theta>0$ using the energy function
\begin{align}
E_p^\varepsilon(t)= \sum_{i=1}^3 \int_\Omega  (u_i^\varepsilon(t))^{p} , \quad 0< t \le T\le \Tm, \quad 1< p <\infty. 
\label{Energy.FunctionDef}
\end{align}
Indeed, by utilising the $L^p$-maximal regularity with independent-of-$p$ constants, see Lemmas~\ref{Lem.MaximalRegularity}-\ref{Lem.ParaMaximalRegularity}, we prove the following a priori estimate   
\begin{align}
\begin{aligned}
& E_p^\varepsilon(t) +  (p-1) \sum_{i=1}^3  \intQt   (u_i^\varepsilon)^{p-2} |\nabla u_i^\varepsilon|^2 +  \sum_{i=1}^3 \intQt  (u_i^\varepsilon)^{p+1} \\
& \le C^\varepsilon_{p,(u_{i0}),v_{30}} + C_p \int_0^t E_p^\varepsilon(s)  + (p-1) C_{p}^\varepsilon \sum_{i=1}^3 \intQt (u_i^\varepsilon)^{p+1}, 
\end{aligned}
\label{Key.Theo1.1}
\end{align}
 where the constant $C_{p}^\varepsilon$ satisfies that $\lim_{p\searrow1} (p-1)C_{p}^\varepsilon=0$, see  Lemma~\ref{Lem.EnergyEstimate}.   This step is done by a sufficiently close choice of $p$ to $1$ from the right.
    \item \textit{Feedback arguments}: Feedback from the prey species to the predator one can be observed from the structure of \eqref{System.UV},  for which, $u_3^\varepsilon \in L^{2+\theta_0}(Q_T)$ if  $u_1^\varepsilon,u_2^\varepsilon \in L^{2+\theta_0}(Q_T)$ for some  $\theta_0>0$, while, in the feedback from the predator species to the prey ones, $u_1^\varepsilon,u_2^\varepsilon \in L^{2+(3/2)\theta_*}(Q_T)$ if  $u_3^\varepsilon \in L^{2+\theta_*}(Q_T)$ for some $\theta_*>0$, see Lemmas~\ref{Lem.Feedback12to3}-\ref{Lem.Feedback3to12}.
    \item \textit{Smoothing effect}: The above feedback allows us to perform a bootstrap argument to show solution regularity up to $L^q(Q_T)$ for any $1\le q<\infty$, see Lemma~\ref{Lem.LpEstimate}. Then,  the smoothing effect of the heat semigroup claims the global existence corresponding to an $L^\infty(Q_T)$-estimate for the solution.  
\end{itemize}

\begin{theorem}[Global existence of classical solution] \label{Theo.GlobalClassicalSol} 
For each $\varepsilon>0$, there exists a unique globally classical solution $(u_i^\varepsilon,v_i^\varepsilon)_{i=1,2,3}$ to  \eqref{System.UV}-\eqref{Condition.Boun-Init.Epsilon} in the sense of Definition~\ref{Definition.Solution}.
\end{theorem}

One of the main ingredients of a rigorous analysis for \textit{fast signal diffusion limits} is the compactness of the $\varepsilon$-depending solution, or more specifically, its \textit{uniform-in-$\varepsilon$} (shortly, \textit{uniform}) bounds. Our analysis is constructed according to the following framework.

\begin{itemize}
    \item \textit{Uniformly improved regularity}: An application of the heat regularisation (Lemma~\ref{Lem.HeatRegularisation}) does not give a uniform estimate for $v_3^\varepsilon$, since its equation includes the slow evolution $\varepsilon \partial_t v_3^\varepsilon$. Thus, we need an improved version of \eqref{Key.Theo1.1}. By the Gagliardo–Nirenberg  inequality,   
\begin{align*}
\intQT (u_i^\varepsilon)^{2p} 
& \le C_p \left( \sup_{0\le t\le T} \int_\Omega (u_i^\varepsilon(t))^p \right)  \intQT \big|\nabla (u_i^\varepsilon)^{\frac{p}{2}}\big|^2 
\end{align*}
holds for any $0<T<\infty$. 
Then, the energy function \eqref{Energy.FunctionDef} can be estimated as follows
\begin{equation*}
\begin{gathered} 
\left( \sup_{0\le t\le T} \int_\Omega (u_i^\varepsilon(t))^p \right)  +  (p-1) \intQT \big| \nabla (u_i^\varepsilon)^{\frac{p}{2}}\big|^2   +  \intQT (u_i^\varepsilon)^{p+1}  \\
\le C_{p,T} + C_{p,T} \int_\Omega u_{i0}^p   +   C_{p,T} (p-1)    \left( \sup_{0\le t\le T} \int_\Omega (u_i^\varepsilon(t))^p \right) , 
\end{gathered}
\end{equation*}
where $C_{p,T}$ does not depend on $\varepsilon$, and it satisfies the limit  
$\lim_{p\searrow 1} (p-1) C_{p,T} = 0$. This implies the uniform boundedness of $u_1^\varepsilon,u_2^\varepsilon,u_3^\varepsilon$  in $L^\infty(0,T;L^{1+\delta}(\Omega)) \cap L^{2+\delta}(Q_T)$ for some $\delta>0$, see Lemma~\ref{Lem.uEps.LInftyL1+}, which is crucial for proving that $v_1^\varepsilon,v_2^\varepsilon$ is uniformly bounded in $L^\infty(0,T;L^{1+\delta}(\Omega))$, and $v_3^\varepsilon$ in $L^\infty(0,T;L^q(\Omega))$ for any $1\le q<\infty$, see Lemma~\ref{Lem.vEps.LInftyLp}.
    \item \textit{Feedback argument via parabolic maximal regularity with slow evolution}: With the parabolic equation  $ \varepsilon \partial_t w^\varepsilon - \lambda \Delta w^\varepsilon + \mu w^\varepsilon = h^\varepsilon$, we show that $\|\Delta w^\varepsilon\|_{L^q(Q_T)}$ can be controlled by $\|h^\varepsilon\|_{L^q(Q_T)}$ for any $1\le q<\infty$, see Lemma~\ref{Lem.MRSlowEvolution}. Then, thanks to the latter application of the Gagliardo–Nirenberg inequality, we show the feedback argument that the uniform boundedness of $u_1^\varepsilon,u_2^\varepsilon,u_3^\varepsilon$ in $L^{2+\delta_0}(Q_T)$ for some $\delta_0>0$ can be improved up to $L^{2+2\delta_0}(Q_T)$, see Lemma~\ref{Lem.Eps.Feedback}.
    \item \textit{Smoothing effect}: A bootstrap argument and the smoothing effect of the heat semigroup can be performed similarly to
the proof of Theorem~\ref{Theo.GlobalClassicalSol}, where the uniform boundedness of $(u_1^\varepsilon,u_2^\varepsilon,u_3^\varepsilon)$ in $L^\infty(Q_T) \cap L^2(0,T;H^1(\Omega))$ is obtained.
    \item \textit{Weak-to-strong convergence}: Due to the lack of time derivatives in the equations for $v_1^\varepsilon, v_2^\varepsilon$, and the vanishing of the parabolicity in the equation for $v_3^\varepsilon$ (i.e., $\varepsilon \partial_t v_3^\varepsilon \to 0$ in a suitable sense), the establishment of strong convergence of $(v_1^\varepsilon,v_2^\varepsilon,v_3^\varepsilon)$ is non-standard. However, we  can use the  energy equation method, see \cite{ball2004global,henneke2016fast}, and the uniform convexity of $L^2(0,T;H^1(\Omega))$ to prove that if $v_1^\varepsilon,v_2^\varepsilon,v_3^\varepsilon$ weakly converge in  $L^2(0,T;H^1(\Omega))$ then the convergence becomes strong, see Lemma~\ref{Lem:Weak-to-strong}. 
    \item \textit{Passing to the limit}: This is based on the Aubin–Lions lemma for $\{u_i^\varepsilon\}$, and the weak-to-strong convergence for $\{v_i^\varepsilon\}$. 
\end{itemize}

\begin{theorem}[Fast signal diffusion limit]  \label{Theo.FastSignalDiffLimit} Let $(u_i^\varepsilon,v_i^\varepsilon)_{i=1,2,3}$ be the global classical solution to Problem \eqref{System.UV}-\eqref{Condition.Boun-Init.Epsilon} for each $\varepsilon>0$. Then, for any $1\le q<\infty$,
\begin{gather}
\sup_{\varepsilon>0} \left( \sum_{i=1}^3 \| u_i^\varepsilon \|_{L^\infty(Q_T)} + \sum_{i=1}^3 \| u_i^\varepsilon \|_{L^2(0,T;H^1(\Omega))}   \right)   \le C_T,    
\label{Theo.FastSignalDiffLimit.State1}
\\
 \displaystyle \sup_{\varepsilon>0} \left( \sum_{i=1}^2 \| v_i^\varepsilon \|_{L^\infty(0,T;W^{2,\infty}(\Omega))} + \| v_3^\varepsilon \|_{L^\infty(Q_T)}  +   \| v_3^\varepsilon \|_{L^q(0,T;W^{2,q}(\Omega))}  \right)  \le C_T,
\label{Theo.FastSignalDiffLimit.State2}
\end{gather}
and up to the whole sequence as $\varepsilon \to 0$,  
\begin{align*}
\begin{array}{clcll}
(u_1^\varepsilon,u_2^\varepsilon,u_3^\varepsilon) &\hspace*{-0.2cm}\to\hspace*{-0.2cm}& (u_1,u_2,u_3) & \text{strongly in } L^q(Q_T)^3,   \\
(v_1^\varepsilon,v_2^\varepsilon,v_3^\varepsilon) &\hspace*{-0.2cm}\to\hspace*{-0.2cm}& (v_1,v_2,v_3) & \text{strongly in } L^q(Q_T)^3,  \\
(\nabla v_1^\varepsilon,\nabla v_2^\varepsilon,\nabla v_3^\varepsilon) &\hspace*{-0.2cm}\to\hspace*{-0.2cm}& (\nabla v_1,\nabla v_2,\nabla v_3) & \text{strongly in } L^q(Q_T)^3,
\end{array}
\end{align*}
where 
 $(u_i,v_i)_{1\le i\le 3}$ is the unique classical solution to Problem \eqref{System.UV.Limit}-\eqref{Condition.Boun-Init.Limit}, see Definition~\ref{Definition.WeakSolution}. Moreover,  the limiting solution has the following regularity, for $i=1,2,3$, 
\begin{gather}
u_i \in L^\infty(0,T;W^{1,\infty}(\Omega)) \cap W^{2,1}_q(Q_T), \quad 
v_i \, \in L^\infty(0,T;W^{2,\infty}(\Omega)) .
\label{Theo.FastSignalDiffLimit.State3}
\end{gather}  
\end{theorem} 

\begin{remark}[Initial value of the limiting solution]
\label{Remark.Theo.FastSignalDiffLimit}
We note from Definition~\ref{Definition.WeakSolution} that $(u_1,u_2,u_3)$ has the same initial condition as \eqref{Condition.Boun-Init.Epsilon}, namely,
$ u_i(0)=u_{i0}$ on $\Omega$ for $ i=1,2,3.$ 
However, due to \eqref{Expression.v}, the initial value of the component $v_3$ is given by
\begin{align}
v_3(0) &  = \int_0^\infty e^{s(\lambda_3 \Delta - \mu_3 I)} u_{30} \, ds .
\label{Expression.v30InTermOfu30}
\end{align}   
According to the identity \eqref{EigenvalueIdentity},   
\begin{align*}
\int_0^\infty e^{s(\lambda_3 \Delta - \mu_3 I)} (\lambda_3 \Delta - \mu_3 I) \, ds \equiv - I,
\end{align*}
which gives combined with \eqref{Expression.v30InTermOfu30} that
\begin{align}
v_{30}-v_3(0) = - \int_0^\infty e^{s(\lambda_3 \Delta - \mu_3 I)} ( \lambda_3 \Delta v_{30} - \mu_3 v_{30} + u_{30}) \, ds.  
\label{Expression.InitialLayer}
\end{align} 
For arbitrary initial data $u_{30},v_{30}$, we generally have  $v_3(0)\not=v_{30}$.  This difference reveals the effect of the so-called initial layer, which will be indicated in Theorem~\ref{Theo.ConvergenceRate}.
\end{remark}

The next interest is to estimate $L^\infty$-in-time convergence rates for the   fast signal diffusion  limit in Theorem~\ref{Theo.FastSignalDiffLimit}. More precisely, the rate $(\widehat{u}_i^{\,\varepsilon},\widehat{v}_i^{\,\varepsilon}) :=(u^{\eps}_i - u_i, v^{\eps}_i - v_i) $ will be estimated in $L^\infty(0,T;W^{r,q}(\Omega))$ with suitable $r,q$. For this purpose, by subtracting side-by-side corresponding equations in \eqref{System.UV} and \eqref{System.UV.Limit}, $(\widehat{u}^{\,\varepsilon}_i, \widehat{v}^{\,\varepsilon}_i)$  satisfies the \textit{rate system}   
\begin{align}\label{System.Rate.UV}
	\left\{ \begin{array}{rlllllll}
		 \partial_t \widehat{u}^{\,\varepsilon}_1 &\hspace*{-0.25cm}-\hspace*{-0.25cm}& d_1\Delta \widehat{u}^{\,\varepsilon}_1  - \chi_1 \di (\widehat{u}^{\,\varepsilon}_1 \nabla v_3^\varepsilon) - \chi_1 \di (u_1 \nabla \widehat{v}^{\,\varepsilon}_3) &=&  \widehat{f}^{\,\varepsilon}_1, \vspace*{0.05cm} \\
		 \partial_t \widehat{u}^{\,\varepsilon}_2  &\hspace*{-0.25cm}-\hspace*{-0.25cm}& d_2\Delta \widehat{u}^{\,\varepsilon}_2  - \chi_2 \di (\widehat{u}^{\,\varepsilon}_2 \nabla v_3^\varepsilon)  - \chi_2 \di (u_2 \nabla \widehat{v}^{\,\varepsilon}_3) &=& \widehat{f}^{\,\varepsilon}_2,\vspace*{0.05cm}  \\
		 \partial_t \widehat{u}^{\,\varepsilon}_3 &\hspace*{-0.25cm}-\hspace*{-0.25cm}& d_3\Delta \widehat{u}^{\,\varepsilon}_3  + \sum_{i=1}^2 \left( \chi_{3i} \di (\widehat{u}^{\,\varepsilon}_3 \nabla v_i^\varepsilon)   +  \chi_{3i} \di (u_3 ( \nabla \widehat{v}^{\,\varepsilon}_i) \right) &=&  \widehat{f}^{\,\varepsilon}_3, \vspace*{0.05cm} \\
	&\hspace*{-0.25cm}-\hspace*{-0.25cm}& \lambda_i \Delta \widehat{v}^{\,\varepsilon}_i\hspace*{0.03cm} + \mu_i \widehat{v}^{\,\varepsilon}_i  = \widehat{u}^{\,\varepsilon}_i , \,\,i=1,2, \vspace*{0.05cm}  \\
		\eps \partial_t \widehat{v}^{\,\eps}_3 &\hspace*{-0.25cm}-\hspace*{-0.25cm}& \lambda_3 \Delta \widehat{v}^{\,\varepsilon}_3 \hspace*{0.03cm} + \mu_3 \widehat{v}^{\,\varepsilon}_3 =  \widehat{u}^{\,\varepsilon}_3 - \eps \partial_t v_3,
	\end{array}
	\right.
\end{align}
subjected to the homogeneous Neumann boundary condition. With $v_3(0)$ defined in \eqref{Expression.v30InTermOfu30}, the initial state of this system is given by   
\begin{align}	\label{Condition.Initial.Rate}
	  (\widehat{u}^{\,\varepsilon}_i(0), \widehat{v}^{\,\varepsilon}_3(0) )|_{\Omega}=(0, v_{30}-v_3(0)), \quad  i=1,2,3.
\end{align} 
Here, for the sake of convenience, we denote  
$\widehat{f}^{\,\varepsilon}_i:= f_i(u_1^\varepsilon,u_2^\varepsilon,u_3^\varepsilon) - f_i(u_1,u_2,u_3).$ 
To obtain $L^\infty$-convergence rates, we first utilise the  energy function  
\begin{align*}
\mathcal E_{n}[\widehat u^{\,\varepsilon}](t):=\sum_{i=1}^3 \int_\Omega (\widehat{u}_i^{\,\varepsilon})^{2n},
\end{align*}
for $n\in \mathbb{N}$, $n\ge 1$. One of the key points is to use the  uniform boundedness of   $(u_i^\varepsilon,v_i^\varepsilon)$ in  Theorem~\ref{Theo.FastSignalDiffLimit}. We prove in Lemma~\ref{Lem.Rate.EnergyE}  that  
\begin{align*}
  \frac{d}{dt} \mathcal E_{n}[\widehat u^{\,\varepsilon}] \le - \frac{2n-1}{n} \sum_{i=1}^3 d_i \int_\Omega |\nabla (\widehat{u}_i^{\,\varepsilon})^{n}|^2  
+ C_{n,T}\, \mathcal E_{n}[\widehat u^{\,\varepsilon}]  + C_{n,T}\,  \mathcal F[\widehat v^{\,\varepsilon}] , 
\end{align*}
where $\mathcal F[\widehat v^{\,\varepsilon}]:=\sum_{i=1}^3  \int_\Omega  |\nabla \widehat{v}_i^{\,\varepsilon}|^2$. In order to estimate $\mathcal E_{n}[\widehat u^{\,\varepsilon}]$, this suggests estimating $\mathcal F[\widehat v^{\,\varepsilon}]$. Indeed, by considering $n=1$, we can show    
\begin{equation*}
\begin{aligned}
\mathcal F[\widehat v^{\,\varepsilon}] \le - \frac{\varepsilon}{2\lambda_3} \frac{d}{dt} \int_\Omega (\widehat{v}_3^{\,\varepsilon})^{2} -  \frac{\mu_3}{2\lambda_3} \int_\Omega (\widehat{v}_3^{\,\varepsilon})^{2} + C  \mathcal E [\widehat u^{\,\varepsilon}] + C\varepsilon^2 \int_\Omega |\partial_t v_3|^2  , 
\end{aligned} 
\end{equation*}
and so,
\begin{align*}
\mathcal E [\widehat u^{\,\varepsilon}]+  \sum_{i=1}^3 \intQT |\nabla \widehat{u}_i^{\,\varepsilon}|^2 \le  C_T  \left( \varepsilon^2 \intQT |\partial_t v_3|^2 + \varepsilon  \int_\Omega (\widehat{v}_3^{\,\varepsilon}(0))^{2} \right),  
\end{align*}
where $\mathcal E[\widehat u^{\,\varepsilon}]:=\mathcal E_{1}[\widehat u^{\,\varepsilon}]$, see Lemma~\ref{Lem.Rate.EstimateF}. Note that a relevant  estimate for the time derivative $\partial_t v_3$ needs to be carried out, see Lemma~\ref{Lem.Rate.TimeDeriOfv3}, for which we recall from the limiting system that this derivative is missing in the equation for $v_3$. Eventually, estimates for $L^\infty$-in-time convergence rates require careful treatment of the initial layer, which will be explained in Remark~\ref{Remark.Theo.ConvergenceRate} after the statement of the third main result.

\begin{theorem}[Convergence rates and the initial layer effects] 
\label{Theo.ConvergenceRate}
Let $(u_i^\varepsilon,v_i^\varepsilon)_{i=1,2,3}$ and $(u_i,v_i)_{i=1,2,3}$ be the global classical solutions to  \eqref{System.UV}-\eqref{Condition.Boun-Init.Epsilon}, for each $\varepsilon>0$, and \eqref{System.UV.Limit}-\eqref{Condition.Boun-Init.Limit}, respectively. Denote  
\begin{align}
\varepsilon_{\mathsf{in}} := \|\lambda_3 \Delta v_{30} - \mu_3 v_{30} + u_{30}\|_{L^2(\Omega)} .
\label{Assumption.InitialLayer}
\end{align}
\begin{itemize}
\item[a)] The followings hold for $i=1,2,3$ and $j=1,2$, 
\begin{equation}
\left\{ \,
\begin{gathered}
 \|u_i^\varepsilon-u_i\|_{L^\infty(0,T;L^2(\Omega))}   \le C_T  \sqrt{\varepsilon} \left( \varepsilon_{\mathsf{in}} + \sqrt{\varepsilon} \right),  \\
 \|u_i^\varepsilon-u_i\|_{L^2(0,T;H^1(\Omega))}    \le C_T  \sqrt{\varepsilon} \left( \varepsilon_{\mathsf{in}} + \sqrt{\varepsilon} \right), \\
 \|v_j^\varepsilon-v_j\|_{L^\infty(0,T;H^2(\Omega))}   \le C_T  \sqrt{\varepsilon} \left( \varepsilon_{\mathsf{in}} + \sqrt{\varepsilon} \right),
\end{gathered}\right.
\label{Theo.ConvergenceRate.a.State1}
\end{equation}
and 
\begin{equation}
\|v_3^\varepsilon-v_3\|_{L^\infty(0,T;H^1(\Omega))} + \|v_3^\varepsilon-v_3\|_{L^2(0,T;H^2(\Omega))} \le C_T \left( \varepsilon_{\mathsf{in}} +  \varepsilon  \right) .
\label{Theo.ConvergenceRate.a.State2}
\end{equation}

\item[b)] For any $2\le q<\infty$ and $i=1,2,3$, $j=1,2$, 
\begin{equation}
\left\{\,\begin{aligned}
 \|u_i^\varepsilon-u_i\|_{L^\infty(0,T;L^q(\Omega))} \hspace*{0.3cm} & \le C_{q,T} \, \varepsilon^{\frac{1}{q}} \Big(\varepsilon_{\mathsf{in}}^{\frac{2}{q}} + \varepsilon^{\frac{1}{q}}\Big),  \\
   \|v_j^\varepsilon-v_j\|_{L^\infty(0,T;W^{2,q}(\Omega))} &\le C_{q,T} \, \varepsilon^{\frac{1}{q}} \Big(\varepsilon_{\mathsf{in}}^{\frac{2}{q}} + \varepsilon^{\frac{1}{q}}\Big).
\end{aligned}\right.
\label{Theo.ConvergenceRate.b.State1}
\end{equation}
If $(u_{30},v_{30})\in W^{2,q}(\Omega) \times W^{4,q}(\Omega)$, then  
\begin{align}
\|v_3^\varepsilon-v_3\|_{L^q(0,T;W^{2,q}(\Omega))} \le C_T \, \varepsilon^{\frac{1}{q}}  \Big(  \widehat{\varepsilon}_{\mathsf{in}} +  \varepsilon_{\mathsf{in}}^{\frac{2}{q}} + \varepsilon^{\frac{1}{q}} \Big), 
\label{Theo.ConvergenceRate.b.State2}
\end{align}
where 
\begin{align*}
\widehat{\varepsilon}_{\mathsf{in}} := \|\lambda_3 \Delta v_{30} - \mu_3 v_{30} + u_{30}\|_{H^2(\Omega)} .
\end{align*}

\item[c)] If $(u_{30},v_{30})\in W^{2,4^+}(\Omega) \times W^{4,4^+}(\Omega)$, then   
\begin{align}
\sum_{i=1}^3 \|u_i^\varepsilon-u_i\|_{L^\infty(Q_T)} \le C_T \, \varepsilon^{(\frac{1}{4})^-}  \Big(  \widehat{\varepsilon}_{\mathsf{in}} +  \varepsilon_{\mathsf{in}}^{(\frac{1}{2})^-} + \varepsilon^{(\frac{1}{4})^-} \Big). 
\label{Theo.ConvergenceRate.b.State3}
\end{align}

\end{itemize}

\end{theorem}

\begin{remark}  
\label{Remark.Theo.ConvergenceRate}
Let us comment on the initial layer. Due to the expression \eqref{Expression.InitialLayer}, if  $(u_{30},v_{30})$  belongs to the critical manifold $\mathcal C$, that is defined at \eqref{CriMani}, 
then $\varepsilon_{\mathsf{in}}=0$ and  $v_3(0)=v_{30}$ (and so, $v_{3}(0)=v_3^\varepsilon(0)$). Otherwise, $\varepsilon_{\mathsf{in}}>0$ and $v_{3}(0)\ne v_3^\varepsilon(0).$ 
If $\varepsilon_{\mathsf{in}}$ is small enough (compared to $\varepsilon$), namely, the dynamics of \eqref{System.UV}-\eqref{Condition.Boun-Init.Epsilon} starts close to the critical manifold $\mathcal C$, then the estimate \eqref{Theo.ConvergenceRate.a.State2} is meaningful.
On the other hand,  \eqref{Theo.ConvergenceRate.a.State1}, \eqref{Theo.ConvergenceRate.b.State1}-\eqref{Theo.ConvergenceRate.b.State3} reveal that the initial layer does not affect the convergence of $(\widehat{u}_1^{\,\varepsilon},\widehat{u}_2^{\,\varepsilon},\widehat{u}_3^{\,\varepsilon},\widehat{v}_1^{\,\varepsilon},\widehat{v}_2^{\,\varepsilon})$, but improves their convergence rates if $\varepsilon_{\mathsf{in}}$, $\widehat{\varepsilon}_{\mathsf{in}}$ are small enough. 
\end{remark}

Theorem~\ref{Theo.ConvergenceRate} reveals interesting effects of the initial layer on the $L^\infty$-in-time convergence rates, where the distance from the initial data to the critical manifold $\mathcal C$ is crucial besides the smallness of the relaxation time $\varepsilon$. In this paper, we also demonstrate this effect numerically, as shown in Table~\ref{Fig:RateOrders}. In the following theorem, we will estimate this distance from the trajectory to $\mathcal C$ in $L^q(Q_T)$ for $1<q<\infty$ and $L^\infty(0,T;L^2(\Omega))$, where the latter means that we can shift the initial time to any time $t\in (0,T)$ with the distance $\varepsilon_t$ and obtain similar estimates to Theorem~\ref{Theo.ConvergenceRate}.

\begin{theorem} [Distance from trajectories to the critical manifold]
\label
{Coro.ConvToCritMani} 
Let $(u_i^\varepsilon,v_i^\varepsilon)_{i=1,2,3}$, for each $\varepsilon>0$, and $(u_i,v_i)_{i=1,2,3}$ be given by Theorem~\ref{Theo.ConvergenceRate}.
\begin{itemize}
    \item[a)] If $(u_{30},v_{30})\in W^{1,q}(\Omega) \times W^{3,q}(\Omega)$, then  
\begin{align}
\|\lambda_3\Delta v_3^\varepsilon - \mu_3 v_3^\varepsilon + u_3^\varepsilon \|_{L^q(Q_T)}   \le C_T \, \varepsilon^{\frac{1}{q}}  \Big(  \widehat{\varepsilon}_{\mathsf{in}} +  \varepsilon_{\mathsf{in}}^{\frac{2}{q}} + \varepsilon^{\frac{1}{q}} \Big)  
\label{Theo.ConvToCritMani.State1}
\end{align}
for any $2\le q<\infty$. 

    \item[b)] (Shifting the initial layer) It holds 
    \begin{align}
    \varepsilon_t:= \|\lambda_3\Delta v_3^\varepsilon(t) - \mu_3 v_3^\varepsilon(t) + u_3^\varepsilon(t) \|_{L^2(\Omega)} \le C_T(\varepsilon_{\mathsf{in}}+\varepsilon)
    \label{Distance:FurtherComment}
\end{align}
for any $0<t<T$.
\end{itemize}
\end{theorem}

Finally, we present some numerical results in Section \ref{sec:6} with comparison of the solutions to the $\ varepsilon$-depending and limiting systems, including the initial layer effect as mentioned earlier, the difference between the systems with one or two prey, as well as their dynamics.  
  
\section{Global existence of classical solution}\label{sec-strong-solu} 

We will prove the global existence of a unique classical solution to \eqref{System.UV}-\eqref{Condition.Boun-Init.Epsilon} for each $\varepsilon>0$, presented in Theorem~\ref{Theo.GlobalClassicalSol}, where the concept of  classical solution is given below.

\begin{definition}\label{Definition.Solution}
The vector of functions $(u_i,v_i)_{1\le i\le 3}$ is called
a global classical solution to Problem \eqref{System.UV}-\eqref{Condition.Boun-Init.Epsilon} if, for any $T>0$, $$ (u_i,v_i)_{1\le i\le 3} \in C(\overline{\Omega}\times[0,T))^3 \cap C^{2,1}(\Omega \times (0,T))^3$$ and \eqref{System.UV}-\eqref{Condition.Boun-Init.Epsilon} are pointwise satisfied.
\end{definition}
  By standard arguments of fixed point theorems, see e.g.  \cite{winkler2010boundedness}, the existence of a local classical, non-negative solution $(u_i^\varepsilon,v_i^\varepsilon)$ can be accomplished up to the maximal time  $0<T_{\max}\le \infty$ such that 
\begin{align}
\label{Lem.GlobalExistenceCriteria}
T_{\max}=\infty \quad \text{or} \quad \left( \lim_{t \to \Tm^-} \sum_{i=1}^3 \|u_i^\varepsilon(t)\|_{L^\infty(\Omega)} = \infty \,\text{ if } \Tm<\infty \right).
\end{align}
Our goal in this section is to prove  $T_{\max}=\infty$ using the criteria \eqref{Lem.GlobalExistenceCriteria}. 

\subsection{Energy estimate}

 We first note that, by the rescaling \eqref{ScaledParameters}, the kinetics $f_i(u_1^\varepsilon,u_2^\varepsilon,u_3^\varepsilon)$, for $i=1,2,3$, are  
\begin{align*}
f_1(u_1^\varepsilon,u_2^\varepsilon,u_3^\varepsilon) =&\, \alpha_1u_1^\varepsilon(1-u_1^\varepsilon - \beta_1  u_2^\varepsilon) -  \dfrac{m_1 u_1^\varepsilon}{\eta_1+u_1^\varepsilon}u_3^\varepsilon,  \\
f_2(u_1^\varepsilon,u_2^\varepsilon,u_3^\varepsilon) =&\,\alpha_2u_2^\varepsilon(1-u_2^\varepsilon - \beta_2u_1^\varepsilon) -  \dfrac{m_2u_2^\varepsilon}{\eta_2+u_2^\varepsilon} u_3^\varepsilon,  \\
f_3(u_1^\varepsilon,u_2^\varepsilon,u_3^\varepsilon) =&\, \left(\gamma_1 \dfrac{m_1 u_1^\varepsilon}{\eta_1+u_1^\varepsilon} + \gamma_2 \dfrac{m_2 u_2^\varepsilon}{\eta_2+u_2^\varepsilon} -k \right) u_3^\varepsilon -l(u_3^\varepsilon)^2,
\end{align*}
where 
\begin{align*}
(\alpha_i,m_i,k,l) := \frac{1}{\max \{\widetilde d_1,\widetilde d_2,\widetilde d_3\}} (\widetilde\alpha_i,\widetilde m_i,\widetilde k,\widetilde l), \quad i=1,2.
\end{align*}
One can observe that   
\begin{align*}
 f_i(u_1^\varepsilon,u_2^\varepsilon,u_3^\varepsilon) &\le \alpha_i  (u_i^\varepsilon -(u_i^\varepsilon)^2), \quad i=1,2, \\ 
 f_3(u_1^\varepsilon,u_2^\varepsilon,u_3^\varepsilon)  &\le   (\gamma_1m_1+\gamma_2m_2)u_3^\varepsilon -l(u_3^\varepsilon)^2.
\end{align*}
Therefore, by integrating the equations for $u_i$, $1\le i\le 3$, over the domain $\Omega$, we obtain the estimate for the total mass 
\begin{align*}
\sum_{i=1}^3  \int_\Omega u_i^\varepsilon(t)  +  \sum_{i=1}^3  \intQTm (u_i^\varepsilon)^2  \le C_T, \quad 0\le t< \Tm.
\end{align*}
However, this regularity is not strong enough in the following sense: feedback from $(u_i^\varepsilon)$ to $(v_i^\varepsilon)$ using the equations for $(v_i^\varepsilon)$, and then from $(v_i^\varepsilon)$ to $(u_i^\varepsilon)$ using the equations for $(u_i^\varepsilon)$ is not enough to improve the regularity of $(u_i^\varepsilon)$ again. To improve the solution regularity, an a priori estimate will be obtained by utilising the energy function \eqref{Energy.FunctionDef}.

\begin{lemma}[Energy estimate] \label{Lem.EnergyEstimate} Let $\varepsilon>0$, and $T\in (0,\infty)$ such that $T \le \Tm$. Then, for $t\in (0,T)$ and $1<p<\infty$, 
\begin{equation}
\begin{aligned}
& E_p^\varepsilon(t) +  (p-1) \sum_{i=1}^3  \intQt   (u_i^\varepsilon)^{p-2} |\nabla u_i^\varepsilon|^2 +  \sum_{i=1}^3 \intQt  (u_i^\varepsilon)^{p+1} \\
& \le C^\varepsilon_{p,(u_{i0}),v_{30}} + C_p \int_0^t E_p^\varepsilon(s)  + (p-1) C_{p}^\varepsilon \sum_{i=1}^3 \intQt (u_i^\varepsilon)^{p+1},
\end{aligned}
\label{Lem.EnergyEstimate.State1}
\end{equation}
where the constants depend on fixed parameters of the problem and $p,\varepsilon$, but not on time. In particular, $C^\varepsilon_{p,(u_{i0}),v_{30}}$ also depends on $\|(u_{i0})\|_{L^p(\Omega)^3}$ and $\|v_{30}\|_{W^{2,p+1}(\Omega)}$. Moreover, $C_{p}^\varepsilon$ satisfies    
\begin{align}
\lim_{p\searrow1} (p-1)C_{p}^\varepsilon=0. \label{Lem.EnergyEstimate.Limit}
\end{align}
\end{lemma}

\begin{proof} Due to the equations for $u_i^\varepsilon$, $1\le i\le 3$, in \eqref{System.UV} and integration by parts, the following computations are straightforward
\begin{align*}
\frac{dE_p^\varepsilon}{dt} =\,&  p\sum_{i=1}^2 \int_\Omega (u_i^\varepsilon)^{p-1} \left( d_i \Delta u_i^\varepsilon + \chi_i \di( u_i^\varepsilon \nabla v_3^\varepsilon)  \right) + p\sum_{i=1}^2 \int_\Omega (u_i^\varepsilon)^{p-1} f_i(u_1^\varepsilon,u_2^\varepsilon,u_3^\varepsilon) \\
+ \,& p\int_\Omega (u_3^\varepsilon)^{p-1} \Big( d_3 \Delta u_3^\varepsilon -  \di( u_3^\varepsilon \nabla (\chi_{31}  v_1^\varepsilon + \chi_{32}  v_2^\varepsilon) ) \Big) + p\int_\Omega (u_3^\varepsilon)^{p-1} f_3(u_1^\varepsilon,u_2^\varepsilon,u_3^\varepsilon)  \\
=\,&  - p\sum_{i=1}^3 d_i(p-1) \int_\Omega   (u_i^\varepsilon)^{p-2} |\nabla u_i^\varepsilon|^2  +  p\sum_{i=1}^3 \int_\Omega (u_i^\varepsilon)^{p-1} f_i(u_1^\varepsilon,u_2^\varepsilon,u_3^\varepsilon)  \\
+\,&  (p-1) \int_\Omega  \Big(  \nabla (u_3^\varepsilon)^p \cdot \nabla (\chi_{31}  v_1^\varepsilon + \chi_{32}  v_2^\varepsilon)  - \nabla( \chi_1(u_1^\varepsilon)^p + \chi_2 (u_2^\varepsilon)^p) \cdot \nabla v_3^\varepsilon \Big)  \\
=:\,& I_1^\varepsilon + I_2^\varepsilon + I_3^\varepsilon.
\end{align*}
Since the term $I_1^\varepsilon$ is non-positive, it is only necessary to deal with the remaining terms. The second term $I_2^\varepsilon$ can be estimated as 
\begin{align*}
 \int_0^t I_2^\varepsilon & \le p\sum_{i=1}^2 \left(  \alpha_i  \intQt (u_i^\varepsilon)^{p} -  \alpha_i \intQt (u_i^\varepsilon)^{p+1} \right) \\
 & + p \left( (\gamma_1m_1+\gamma_2m_2)  \intQt (u_3^\varepsilon)^{p} -  l \intQt (u_3^\varepsilon)^{p+1} \right). 
\end{align*}
On the other hand, by using the equations for $(v_i^\varepsilon)$ in \eqref{System.UV} and the Young's inequality,
\begin{align*}
\int_0^t I_3^\varepsilon =\,&\,  (p-1) \intQt \Big((u_3^\varepsilon)^p (- \Delta (\chi_{31}v_1^\varepsilon + \chi_{32}v_2^\varepsilon) )  - (\chi_1(u_1^\varepsilon)^p + \chi_2(u_2^\varepsilon)^p) (-\Delta v_3^\varepsilon) \Big) \\
\le \,&\, (p-1) \left( \frac{p}{p+1} \intQt (u_3^\varepsilon)^{p+1} + \frac{1}{p+1} \intQt (\chi_{31}^{p+1}|\Delta v_1^\varepsilon|^{p+1} + \chi_{32}^{p+1} |\Delta v_2^\varepsilon|^{p+1} ) \right) \\
+ \,&\, (p-1) \left( \frac{p}{p+1}  \intQt   ( (u_1^\varepsilon)^{p+1} +  (u_2^\varepsilon)^{p+1}) + \frac{\chi_1^{p+1}+\chi_2^{p+1}}{p+1} \intQt |\Delta v_3^\varepsilon|^{p+1} \right) .
\end{align*}
Applying the elliptic maximal regularity in Lemma~\ref{Lem.MaximalRegularity} to the equations for $v_i^\varepsilon$ gives 
$$ \intQt |\Delta v_i^\varepsilon|^{p+1} \le (C^{\mathsf{EM}})^{p+1} \intQt (u_i^\varepsilon)^{p+1}, \quad i=1,2.$$ 
While, by rewriting the equation for $v_3^\varepsilon$ as $\partial_t v_3^\varepsilon - (\lambda_3/\varepsilon) \Delta v_3^\varepsilon + (\mu/\varepsilon) v_3^\varepsilon = (1/\varepsilon)u_3^\varepsilon$, and   applying the parabolic maximal regularity  in Lemma~\ref{Lem.ParaMaximalRegularity} (with $p_0=3$),   
\begin{align*}
\intQt |\Delta v_3^\varepsilon|^{p+1} &\le (C^{\mathsf{PM},\varepsilon})^{p+1} \left( \|v_{30}\|_{W^{2,p+1}(\Omega)}^{p+1} + \intQt \left(\frac{u_3^\varepsilon}{\varepsilon}\right)^{p+1} \right) ,   
\end{align*}
where 
\begin{align*}
C^{\mathsf{PM},\varepsilon}:= 
\left\{ \begin{array}{llll}
C_{\lambda_3/\varepsilon,\mu_3/\varepsilon,3}^{\mathsf{PM}} & \text{if }  p\le 2, \vspace*{0.05cm} \\
C_{\lambda_3/\varepsilon,\mu_3/\varepsilon,p+1}^{\mathsf{PM}} & \text{if }  p>2. 
\end{array} \right.
\end{align*}
It is helpful to note that the constant $C^{\mathsf{PM},\varepsilon}$ does not depend on $t$, and additionally not on $p$ if $1<p\le 2$. Consequently, we obtain the following estimate for $I_3^\varepsilon$ 
\begin{align*}
\int_0^t I_3^\varepsilon  \le \,&\,  (p-1) \sum_{i=1}^2 \frac{p+\chi_{3i}^{p+1}(C^{\mathsf{EM}})^{p+1}}{p+1} \intQt (u_i^\varepsilon)^{p+1}  \\
+ \,&\, (p-1) \frac{p+(\chi_{1}^{p+1}+\chi_{2}^{p+1})(C^{\mathsf{PM},\varepsilon})^{p+1}}{(p+1)\varepsilon^{p+1}} \intQt (u_3^\varepsilon)^{p+1} \\
+ \,&\, (p-1) \frac{(\chi_{1}^{p+1}+\chi_{2}^{p+1})(C^{\mathsf{PM},\varepsilon})^{p+1}}{p+1} \|v_{30}\|_{W^{2,p+1}(\Omega)}^{p+1} .
\end{align*}

By plugging all the above estimates for $I_1^\varepsilon,I_2^\varepsilon,I_3^\varepsilon$ to have the corresponding estimate for $dE^{\varepsilon}_p/dt$, and then integrating the result over time, we get
\begin{equation*}
\begin{aligned}
E_p^\varepsilon(t) &  + p(p-1) \sum_{i=1}^3 d_i \intQt   (u_i^\varepsilon)^{p-2} |\nabla u_i^\varepsilon|^2 + p \min(\alpha_1;\alpha_2;l) \sum_{i=1}^3 \intQt  (u_i^\varepsilon)^{p+1} \\
& \le  E_p(0) + (p-1) \frac{(\chi_{1}^{p+1}+\chi_{2}^{p+1})(C^{\mathsf{PM},\varepsilon})^{p+1}}{p+1} \|v_{30}\|_{W^{2,p+1}(\Omega)}^{p+1} \\
&  + p \max(\alpha_1;\alpha_2;\gamma_1m_1+\gamma_2m_2) \int_0^t E_p^\varepsilon(s)   + (p-1) C_{p}^\varepsilon \sum_{i=1}^3 \intQt (u_i^\varepsilon)^{p+1},  
\end{aligned}  
\end{equation*}
where 
\begin{align*}
C_{p}^\varepsilon :=  \max\left( \sum_{i=1}^2 \frac{p+\chi_{3i}^{p+1}(C^{\mathsf{EM}})^{p+1}}{p+1}; \, \frac{p+(\chi_{1}^{p+1}+\chi_{2}^{p+1})(C^{\mathsf{PM},\varepsilon})^{p+1}}{(p+1)\varepsilon^{p+1}} \right).
\end{align*}
Here, the term $E_p^\varepsilon(0)$ does not depend on $\varepsilon$, so removing the superscript $\varepsilon$ is more suitable. Moreover, $E_p(0)$ is finite for any $1<p<\infty$ due to Assumption~\ref{Assumption.Initial1}. The energy estimate \eqref{Lem.EnergyEstimate.State1} is obtained by dividing two sides of the latter estimate by $\min(1;d_ip;p\alpha_1;p\alpha_2;pl)$. Since   $C^{\mathsf{EM}}$ and  $C^{\mathsf{PM},\varepsilon}$ are independent of $p$ as $1<p\le 2$, the limit \eqref{Lem.EnergyEstimate.Limit} is obvious. 
\end{proof}

\subsection{Feedback argument via heat regularisation}

In this part, we point out the feedback between prey and predator species.  

\begin{lemma}[Feedback from prey to predator] \label{Lem.Feedback12to3} Let $T\in (0,\infty)$, $T \le \Tm$.  If there exists $\theta_0>0$ such that
\begin{align}
\intQT \left( (u_1^\varepsilon)^{2+\theta_0} + (u_2^\varepsilon)^{2+\theta_0} \right) \le C_T^\varepsilon,
\label{Lem.Feedback12to3.State1}
\end{align}
then
\begin{align}
\label{Lem.Feedback12to3.State2}
 \intQT (u_3^\varepsilon)^{2+\theta_0}    \le  C_{T,\theta_0}^\varepsilon  +   C_{T,\theta_0}^\varepsilon \intQT \left( (u_1^\varepsilon)^{2+\theta_0} + (u_2^\varepsilon)^{2+\theta_0} \right) .
\end{align} 
\end{lemma}

\begin{proof} This lemma can be proved similarly to Lemma~\ref{Lem.EnergyEstimate}, where we just need to compute the last term of the energy function \eqref{Energy.FunctionDef} with $p=1+\theta_0$, to see that   
\begin{equation}
\begin{aligned}
& \int_\Omega (u_3^\varepsilon)^{1+\theta_0} +  \intQt (u_3^\varepsilon)^{\theta_0-1}  |\nabla u_3^\varepsilon|^2  +   \intQt (u_3^\varepsilon)^{2+\theta_0}   \\
& \le C_{\theta_0}  +  C_{\theta_0}  \intQt \left( (u_1^\varepsilon)^{2+\theta_0} + (u_2^\varepsilon)^{2+\theta_0} \right)  + C_{\theta_0} \intQt (u_3^\varepsilon)^{1+\theta_0} .
\end{aligned}
\label{Lem.Feedback12to3.Proof2}
\end{equation}
Thanks to \eqref{Lem.Feedback12to3.State1} and the Gr\"onwall inequality, we obtain the estimate \eqref{Lem.Feedback12to3.State2}.
\end{proof}
 
\begin{lemma}[Feedback from predator to prey via heat regularisation] \label{Lem.Feedback3to12} Let $T\in (0,\infty)$, $T \le \Tm$. If there exists $\theta_*>0$ such that 
\begin{align}
\intQT  (u_3^\varepsilon)^{2+\theta_*} \le C_{T}^\varepsilon,
\label{Lem.Feedback3to12.State1}
\end{align}
then
\begin{align}
\label{Lem.Feedback3to12.State2}
\intQT \left( (u_1^\varepsilon)^{2+\frac{3}{2}\theta_*} + (u_2^\varepsilon)^{2+\frac{3}{2}\theta_*} \right) \le C_{T,\theta_*}^\varepsilon  + C_{T,\theta_*}^\varepsilon \left( \intQT (u_3^\varepsilon)^{2+\theta_*} \right)^{\frac{2+3\theta_*}{4+2\theta_*}} .
\end{align} 
\end{lemma}

\begin{proof} Direct computations show 
\begin{equation}
\begin{aligned}
& \sum_{i=1}^2 \left( \int_\Omega (u_i^\varepsilon(T))^p  + d_i p(p-1)  \intQT (u_i^\varepsilon)^{p-2} |\nabla u_i^\varepsilon|^2 +  p \alpha_i \intQT (u_i^\varepsilon)^{p+1} \right) \\
& \le \sum_{i=1}^2 \left( \int_\Omega u_{i0}^p +  p \alpha_i \intQT (u_i^\varepsilon)^p -  p(p-1) \intQT \chi_i (u_i^\varepsilon)^{p-1} \nabla u_i^\varepsilon \cdot \nabla v_3^\varepsilon \right)
\end{aligned}
\label{Lem.Feedback3to12.Proof1}
\end{equation}
for any $p>1$. By the Young inequality, we get
\begin{align*}
  - \intQT \chi_i (u_i^\varepsilon)^{p-1} \nabla u_i^\varepsilon \cdot \nabla v_3^\varepsilon  \le  \frac{d_i}{2} \intQT (u_i^\varepsilon)^{p-2} |\nabla u_i^\varepsilon|^2  + \frac{\chi_i^2}{2d_i} \intQT (u_i^\varepsilon)^p |\nabla v_3^\varepsilon|^2 ,
\end{align*} 
and
\begin{align*}
\sum_{i=1}^2 p \alpha_i \intQT (u_i^\varepsilon)^p \le C_{Q_T,p,\alpha_i} + \sum_{i=1}^2 \frac{p\alpha_i}{2} \intQT (u_i^\varepsilon)^{p+1}.
\end{align*}
We then imply from \eqref{Lem.Feedback3to12.Proof1} that
\begin{equation}
\begin{aligned} 
& \sum_{i=1}^2 \int_\Omega (u_i^\varepsilon)^p + \sum_{i=1}^2 \intQT (u_i^\varepsilon)^{p-2} |\nabla u_i^\varepsilon|^2 + \sum_{i=1}^2 \intQT (u_i^\varepsilon)^{p+1} \\
& \le C_p \sum_{i=1}^2 \int_\Omega u_{i0}^p + C_p \left( C_{Q_T,p,\alpha_i}  +  \sum_{i=1}^2 \intQT (u_i^\varepsilon)^p |\nabla v_3^\varepsilon|^2 \right) 
\end{aligned}
\label{Lem.Feedback3to12.Proof2}
\end{equation}
for $0<t<T$. Let us consider the regularity of the term $|\nabla v_3^\varepsilon|$ under the assumption~\eqref{Lem.Feedback3to12.State1}. The heat regularisation,  Lemma~\ref{Lem.HeatRegularisation}, can be applied to the equation
\begin{align*}
 \partial_t v_3^\varepsilon -  \frac{\lambda_3}{\varepsilon} \Delta v_3^\varepsilon +  \frac{\mu_3}{\varepsilon} v_3^\varepsilon = \frac{1}{\varepsilon} u_3^\varepsilon
\end{align*}
with $u_3^\varepsilon \in L^{2+\theta_*}(Q_T)$, such that 
\begin{align}
\intQT |\nabla v_3^\varepsilon|^{\frac{4(2+\theta_*)}{4-(2+\theta_*)}} \le C_{T,\theta_*}^\varepsilon , \, \text{ or equivalently, } \, \intQT |\nabla v_3^\varepsilon|^{4+\frac{8\theta_*}{2-\theta_*}} \le C_{T,\theta_*}^\varepsilon,
\label{Lem.Feedback3to12.Proof3}
\end{align}
where we used the convention
\begin{align*}
\frac{1}{2-\theta_*} := \left\{ 
\begin{array}{llllll}
<\infty \text{ arbitrarily} & \text{if } \theta_* = 2,\\
 \infty & \text{if } \theta_* > 2.
\end{array}
\right.
\end{align*}
By the H\"older's inequality, the last term of \eqref{Lem.Feedback3to12.Proof2} can be estimated as 
\begin{align*}
\intQT  (u_i^\varepsilon)^p |\nabla v_3^\varepsilon|^2 \le \left( \intQT ((u_i^\varepsilon)^{2+\theta_*})^{\frac{p}{1+(3/2)\theta_*}} \right)^{\frac{2+3\theta_*}{4+2\theta_*}} \left( \intQT |\nabla v_3^\varepsilon|^{4+\frac{8\theta_*}{2-\theta_*}} \right)^{\frac{2-\theta_*}{4+2\theta_*}} . 
\end{align*}
Therefore, by employing the assumption \eqref{Lem.Feedback3to12.State1} and the regularity \eqref{Lem.Feedback3to12.Proof3}, it is possible to choose $p=1+\frac{3}{2}\theta_*$ to see that 
\begin{align*}
\intQT  (u_i^\varepsilon)^{1+\frac{3}{2}\theta_*} |\nabla v_3^\varepsilon|^2 \le \left( \intQT (u_i^\varepsilon)^{2+\theta_*} \right)^{\frac{2+3\theta_*}{4+2\theta_*}} \left( \intQT |\nabla v_3^\varepsilon|^{4+\frac{8\theta_*}{2-\theta_*}} \right)^{\frac{2-\theta_*}{4+2\theta_*}} \le C_{T,\theta_*}^\varepsilon. 
\end{align*}
Consequently, letting $p=1+\frac{3}{2}\theta_*$ in \eqref{Lem.Feedback3to12.Proof2} gives 
\begin{equation*}
\begin{aligned} 
& \sum_{i=1}^2 \int_\Omega (u_i^\varepsilon)^{1+\frac{3}{2}\theta_*} + \sum_{i=1}^2 \intQT (u_i^\varepsilon)^{\frac{3}{2}\theta_*-1} |\nabla u_i^\varepsilon|^2 + \sum_{i=1}^2 \intQT (u_i^\varepsilon)^{2+\frac{3}{2}\theta_*} \\
& \le \sum_{i=1}^2 \int_\Omega u_{i0}^{1+\frac{3}{2}\theta_*}   +  C_{T,\theta_*}^\varepsilon + C_{T,\theta_*}^\varepsilon \sum_{i=1}^2 \left( \intQT (u_i^\varepsilon)^{2+\theta_*} \right)^{\frac{2+3\theta_*}{4+2\theta_*}}   ,
\end{aligned}
\end{equation*}
 which shows the estimate    \eqref{Lem.Feedback3to12.State2}. 
\end{proof}

\subsection{Smoothing effect, and global existence}

We first observe from  the limit \eqref{Lem.EnergyEstimate.Limit} in Lemma~\ref{Lem.EnergyEstimate}, that an $L^{2+}( Q_T)$ estimate for solutions can be obtained by choosing $p$ sufficiently close to $1$ from the right, cf. Lemma~\ref{Lem.L2+Estimate}. Then, in Lemma~\ref{Lem.LpEstimate}, we estimate solutions in $L^p( Q_T)$ for any $1<p<\infty$ by performing a bootstrap argument, where the feedback arguments in Lemmas~\ref{Lem.Feedback12to3}-\ref{Lem.Feedback3to12} are crucial. Finally, by the smoothing effect of the heat semigroup, we can show $\Tm=\infty$, i.e. the global existence, via the criteria \eqref{Lem.GlobalExistenceCriteria}.

\begin{lemma}[$L^{2+}$-estimate] 
\label{Lem.L2+Estimate} Let $T\in (0,\infty)$, $T \le \Tm$. There exists $\theta>0$ such that
\begin{align*}
\sum_{i=1}^3 \left( \intQT (u_i^\varepsilon)^{2+\theta} +  \intQT  |\nabla u_i^\varepsilon|^{\frac{4+2\theta}{3}} \right) \le C_{T,\theta}^\varepsilon.
\end{align*}
\end{lemma}

\begin{proof} Let $C_{p}^\varepsilon$ be the constant given by Lemma~\ref{Lem.EnergyEstimate}. Thanks to the limit \eqref{Lem.EnergyEstimate.Limit}, we can find $\theta>0$ such that, with $p=1+\theta$,
$$(p-1)C_{p}^\varepsilon < 1.$$
Integrating the energy estimate \eqref{Lem.EnergyEstimate.State1} over time gives
\begin{gather*}
E_{1+\theta}^\varepsilon(t) +  \sum_{i=1}^3 \left( \intQt \frac{|\nabla u_i^\varepsilon|^2}{(u_i^\varepsilon)^{1-\theta}} + \intQt  (u_i^\varepsilon)^{2+\theta} \right) \le C_{\theta,T}^\varepsilon + C_{\theta,T}^\varepsilon \int_0^t E_{1+\theta}^\varepsilon(s)   . 
\end{gather*}
Therefore, the Gr\"onwall's inequality yields that $E_{1+\theta}^\varepsilon$ is bounded on  $(0,T)$. Moreover, by applying the Young's inequality,  we have 
\begin{align*}
\intQT  \frac{|\nabla u_i^\varepsilon|^2}{(u_i^\varepsilon)^{1-\theta}} +  \intQT  (u_i^\varepsilon)^{2+\theta} \ge C_\theta \intQT  |\nabla u_i^\varepsilon|^{\frac{4+2\theta}{3}}
\end{align*}
for $1\le i\le 3$. 
Consequently, 
\begin{align*}
\sum_{i=1}^3 \left( \intQT  (u_i^\varepsilon)^{2+\theta}+ \intQT |\nabla u_i^\varepsilon|^{\frac{4+2\theta}{3}} 
\right) \le C_{T,\theta}^\varepsilon,
\end{align*}
i.e. the desired estimate is proved.
\end{proof}

\begin{lemma}[$L^p$-estimate]\label{Lem.LpEstimate} 
Let $T\in (0,\infty)$, $T \le \Tm$. For any $1< p<\infty$,
\begin{align*}
\sum_{i=1}^3 \left( \int_\Omega (u_i^\varepsilon)^p + \intQT |\nabla u_i^\varepsilon|^2 + \intQT (u_i^\varepsilon)^{p+1} \right) \le C_{p,T}^\varepsilon.
\end{align*}
\end{lemma}

\begin{proof}
We perform a bootstrap argument via the regularity feedback studied in Lemmas~\ref{Lem.Feedback12to3}-\ref{Lem.Feedback3to12}, which is presented as follows.

\medskip

\noindent \underline{\textit{Step 1:}} Let $\theta_1:=\theta$, where $\theta$ is defined by 
 Lemma~\ref{Lem.L2+Estimate}. Then, $(u_i^\varepsilon)_{1\le i\le 3} \in L^{2+\theta_1}( Q_{T})^3$. 
By applying Lemma~\ref{Lem.Feedback3to12}, the feedback from the predator species to the prey ones gives
\begin{align*}
 \intQT \left( (u_1^\varepsilon)^{2+\frac{3}{2}\theta_1} + (u_2^\varepsilon)^{2+\frac{3}{2}\theta_1} \right) \le C_{T,\theta_1}^\varepsilon + C_{T,\theta_1}^\varepsilon  \left( \intQT (u_3^\varepsilon)^{2+\theta_1} \right)^{\frac{2+3\theta_1}{4+2\theta_1}} \le C_{T,\theta_1}^\varepsilon.
\end{align*}
Then, by Lemma~\ref{Lem.Feedback12to3}, the feedback from prey to predator yields that 
\begin{align*}
\intQT (u_3^\varepsilon)^{2+\frac{3}{2}\theta_1} \le  C_{T,\theta_1}^\varepsilon  + C_{T,\theta_1}^\varepsilon \intQT \left( (u_1^\varepsilon)^{2+\frac{3}{2}\theta_1} + (u_2^\varepsilon)^{2+\frac{3}{2}\theta_1} \right) \le C_{T,\theta_1}^\varepsilon.
\end{align*} 
Note that the above constants $C_{T,\theta_1}^\varepsilon$ are finite for finite values of $\theta_1$. 

\medskip

\noindent \underline{\textit{Step 2:}} Due to the first step, we have $(u_i^\varepsilon)_{1\le i\le 3} \in L^{2+\theta_2}( Q_{T})^3 $ with $\theta_2=\frac{3}{2}\theta_1=\frac{3}{2}\theta$. Therefore, in the same way as Step 1, we see that 
\begin{align*}
\intQT \left( (u_1^\varepsilon)^{2+\frac{3}{2}\theta_2} + (u_2^\varepsilon)^{2+\frac{3}{2}\theta_2} +  (u_3^\varepsilon)^{2+\frac{3}{2}\theta_2} \right) \le C_{T,\theta_2}^\varepsilon. 
\end{align*}

\noindent \dots

\noindent \underline{\textit{Step $n$:}} Due to the $(n-1)$-th step, we have  $(u_i^\varepsilon)_{1\le i\le 3} \in L^{2+\theta_n}( Q_{T})^3 $ with 
$$\theta_n = \left( \frac{3}{2} \right)^{n-1} \theta .$$ 
Then, by combining Lemmas~\ref{Lem.Feedback12to3} and~\ref{Lem.Feedback3to12}, 
\begin{align*}
\intQT \left( (u_1^\varepsilon)^{2+\frac{3}{2}\theta_n} + (u_2^\varepsilon)^{2+\frac{3}{2}\theta_n} + (u_3^\varepsilon)^{2+\frac{3}{2}\theta_n} \right) \le C_{T,\theta_n}^\varepsilon. 
\end{align*} 

Since $\lim_{n \to \infty} \theta_n = \infty$, the above bootstrap argument claims that 
\begin{align*}
\intQT \left( (u_1^\varepsilon)^{p} + (u_2^\varepsilon)^{p} + (u_3^\varepsilon)^{p} \right) \le C_{T,p}^\varepsilon  
\end{align*} 
for any $1<p<\infty$, where $C_{T,p}^\varepsilon$ is finite for finite values of $p$. 
This can be incoporated into the energy estimate (cf. Lemma~\ref{Lem.EnergyEstimate})  again to obtain
\begin{align*}
\sum_{i=1}^3 \left( \int_\Omega (u_i^\varepsilon)^p(t) + \intQT (u_i^\varepsilon)^{p-2} |\nabla u_i^\varepsilon|^2 +  \intQT (u_i^\varepsilon)^{p+1} \right) \le C_{T,p}^\varepsilon
\end{align*}
for any $1<p<\infty$. In particular, we have $\nabla u_i^\varepsilon\in L^2( Q_{T})$ by letting $p=2$.
\end{proof}

\begin{remark} In the above lemma, since $C_{T,p}^\varepsilon$ may tend to $\infty$ as $p\to \infty$, it is not direct to claim the $L^\infty(Q_T)$-boundedness of solutions, which, however, will be obtained due to the smoothing effect of the heat semigroup. 
\end{remark}

We are ready to prove the first main result, Theorem~\ref{Theo.GlobalClassicalSol}.

\begin{proof}[Proof of Theorem~\ref{Theo.GlobalClassicalSol}]
Taking into account the criteria \eqref{Lem.GlobalExistenceCriteria}, we prove $\Tm=\infty$ by showing that 
\begin{align}
\label{Lemma.LpEstimate.Proof1}
\sum_{i=1}^3 \|u_i^\varepsilon \|_{L^\infty( Q_{\Tm})} \le C_{\Tm}^\varepsilon 
\end{align}
under the contradiction assumption $\Tm < \infty$.  To prove \eqref{Lemma.LpEstimate.Proof1}, we will make use of the smoothing effect of the Neumann heat semigroups $\{e^{t(d_i\Delta-I)}:0\le t\le \Tm\}$.

Let us take $q>4$ and $\frac{1}{q}<\sigma<\frac{1}{2}-\frac{1}{q}$. Then, thanks to Theorem 1.6.1 in \cite{henry2006geometric}, 
\begin{align}
\|\xi\|_{L^\infty(\Omega)}\le C \|(-d_i\Delta+I)^{\sigma}\xi\|_{L^q(\Omega)}, \quad \forall \xi \in D((-d_i\Delta+I)^{\sigma}) . 
\label{Lemma.LpEstimate.Proof2}
\end{align}
Therefore, by utilising the estimate \eqref{HeatSemigroup.Contraction} and Lemma~\ref{Lem.HeatCombineDiv}, 
\begin{equation*}
  \begin{split}
    & \left\| \int^t_{0}  e^{(t-s)(d_i\Delta-I)} \chi_i  \di (u_i^\varepsilon \nabla v_3^\varepsilon) +  e^{(t-s)(d_i\Delta-I)} (f_i(u_1^\varepsilon,u_2^\varepsilon,u_3^\varepsilon)+u_i^\varepsilon) \Big) \, ds \right\|_{\LO{\infty}} \\ 
     &\le    C \int^t_{0} \|(-d_i\Delta + I)^\sigma e^{(t-s)(-d_i\Delta+I)}  \di (u_i^\varepsilon \nabla v_3^\varepsilon)  \|_{\LO{q}} \, ds\\
     & + C \int^t_{0} \|(-d_i\Delta + I)^\sigma e^{(t-s)(-d_i\Delta+I)}  (f_i(u_1^\varepsilon,u_2^\varepsilon,u_3^\varepsilon)+u_i^\varepsilon) \|_{\LO{q}} \, ds \\
     &\le    C \int^t_{0} (t-s)^{-\sigma-\frac{1}{2}-\kappa} \left(  \| u_i^\varepsilon \nabla v_3^\varepsilon \|_{\LO{q}} + T^{\frac{1}{2}+\kappa} \| f_i(u_1^\varepsilon,u_2^\varepsilon,u_3^\varepsilon)+u_i^\varepsilon \|_{\LO{q}} \right) \, ds, \\
  \end{split}
\end{equation*}
for $i=1,2$ and any $\kappa>0$. By the heat regularisation, cf. Lemma ~\ref{Lem.HeatRegularisation}, applied to the equation for $v_3^\varepsilon$, where $u_3^\varepsilon \in L^q(Q_{\Tm})$ with $q>N+2$ ($N$ is the spatial dimension, here $N=2$),  we have the boundedness of $\nabla v_3^\varepsilon$ in $L^\infty( Q_{\Tm})$. Therefore, due to the regularity in Lemma~\ref{Lem.LpEstimate}, 
\begin{align*}
& \| u_i^\varepsilon \nabla v_3^\varepsilon \|_{\LO{q}} + \Tm^{\frac{1}{2}+\kappa} \| f_i(u_1^\varepsilon,u_2^\varepsilon,u_3^\varepsilon)+u_i^\varepsilon \|_{\LO{q}} \\
& \le  \| u_i^\varepsilon\|_{\LO{2q}} \|\nabla v_3^\varepsilon \|_{\LO{2q}}  + C_{\Tm}^\varepsilon \|(u_1^\varepsilon,u_2^\varepsilon,u_3^\varepsilon)\|_{\LO{2q}^3}^2 \\
& \le C_{\Tm}^\varepsilon \|\nabla v_3^\varepsilon \|_{\LO{2q}}^2  + C_{\Tm}^\varepsilon \|(u_1^\varepsilon,u_2^\varepsilon,u_3^\varepsilon)\|_{\LO{2q}^3}^2 \\
& \le C_{\Tm}^\varepsilon \|(u_1^\varepsilon,u_2^\varepsilon,u_3^\varepsilon)\|_{\LO{2q}^3}^2,  
\end{align*}
where we used the elliptic maximal regularity~\ref{Lem.MaximalRegularity}.  Since the initial data is smooth enough as in Assumption~\ref{Assumption.Initial1}, the  Duhamel's principle and the H\"older's inequality yield  
\begin{equation*}
  \begin{split}
   \|u_i^\varepsilon(t)\|_{L^\infty(\Omega)}   
   & \le C   \| u_{i0}\|_{\LO{\infty}} + C_T^\varepsilon  \|(u_1^\varepsilon,u_2^\varepsilon,u_3^\varepsilon)\|_{L^{2q}( Q_{\Tm})^3}^2  \, t^{\frac{1}{2}-\frac{1}{q}-\sigma-\kappa} \\
   & \le C   \| u_{i0}\|_{\LO{\infty}} + C_T^\varepsilon \|(u_1,u_2,u_3)\|_{L^{2q}( Q_{\Tm})^3}^2 \,  \Tm^{\frac{1}{2}-\frac{1}{q}-\sigma-\kappa} , 
  \end{split}
\end{equation*}
 for $0<t<\Tm$, where we chose $0<\kappa < 1/2-1/q-\sigma$. Thus, we imply $u_1^\varepsilon,u_2^\varepsilon \in L^\infty( Q_{\Tm})$. We can obtain  
$u_3^\varepsilon \in L^\infty( Q_{\Tm})$ in the same way. Summarily, \eqref{Lemma.LpEstimate.Proof1} is proved, and the conclusion of this theorem is followed due to proof of contradiction.
\end{proof}

\section{Rigorous analysis for fast signal diffusion limit}\label{sec:4}

This section is devoted to studying rigorously the fast signal diffusion limit, where the $\varepsilon$-depending system \eqref{System.UV}-\eqref{Condition.Boun-Init.Epsilon} is reduced to \eqref{System.UV.Limit}-\eqref{Condition.Boun-Init.Limit}. Here, we follow the framework in the Section~\ref{sec-strong-solu}. The uniformly improved regularity will be presented in Subsection~\ref{Sec:UniImprRegu}, the feedback argument in Subsection~\ref{Sec:FSL:Feedback}, the uniform boundedness in $L^\infty(Q_T)$ in Subsection~\ref{Sec:UniBoundLInf}, and finally,  both the weak-to-strong convergence and passing to the limit in Subsection~\ref{Sec:Weak2StrongPassing}.

\subsection{Uniformly improved regularity}
\label{Sec:UniImprRegu}

By integrating the equations for $u_i^\varepsilon$, $1\le i\le 3$, we directly obtain the following estimate
\begin{align*}
\sup_{\varepsilon>0}  \left( \sup_{0\le t\le T} \int_\Omega u_i^\varepsilon(t) + \intQT (u_i^\varepsilon)^{2}   \right) \le C_T.
\end{align*}
However, in Lemma~\ref{Lem.uEps.LInftyL1+} we observe that this can be improved up to $L^\infty(0,T;L^{1+}(\Omega)) \cap L^{2+}(Q_T)$ by considering the energy similarly as \eqref{Energy.FunctionDef}, of course, and taking care of the $\varepsilon$-dependence of the solution,  which also gives a useful gradient estimate. For the proof of Lemma~\ref{Lem.uEps.LInftyL1+}, we will employ estimates in Lemma~\ref{Lem.vEps.L2H2}. 

\begin{lemma} \label{Lem.vEps.L2H2} For $i=1,2,3$, 
\begin{align*}
\sup_{\varepsilon>0} \left( \intQT \big( (v_i^\varepsilon)^2 + |\nabla v_i^\varepsilon|^2 + |\Delta v_i^\varepsilon|^2 \big)  \right) \le C_T. 
\end{align*}
\end{lemma}

\begin{proof} The lemma is proved by multiplying the two sides of the equation for $v_i^\varepsilon$ by $v_i^\varepsilon - \Delta v_i^\varepsilon$ and taking into account that $u_i$ is bounded in $L^2(Q_T)$ uniformly.
\end{proof}

\begin{lemma} \label{Lem.uEps.LInftyL1+} There exists $\delta>0$ such that 
\begin{align*}
\sup_{\varepsilon>0} \left(  \sup_{0\le t\le T} \int_\Omega (u_i^\varepsilon(t))^{1+\delta} + \intQT (u_i^\varepsilon)^{2+\delta}  \right)   \le C_T
\end{align*}
and 
\begin{align*}
\sup_{\varepsilon>0} \left(  \intQT \big|\nabla (u_i^\varepsilon)^{\frac{1+\delta}{2}} \big|^2 + \intQT \big|\nabla u_i^\varepsilon\big|^{\frac{4+2\delta}{3}} \right)   \le C_T
\end{align*}
for $i=1,2,3$.
\end{lemma}

\begin{proof} We estimate only the components $u_i^\varepsilon$ for $i=1,2$, and note that the component $u_3^\varepsilon$ can be estimated in the same way. Put $r(u_i^\varepsilon):= (u_i^\varepsilon)^{p/2}$. Then, by using the Gagliardo–Nirenberg interpolation inequality,  
\begin{equation} 
\begin{aligned}
\int_\Omega (u_i^\varepsilon)^{2p} & = \|r(u_i^\varepsilon)\|_{L^{4}(\Omega)}^4 \le C\|\nabla r(u_i^\varepsilon)\|_{L^2(\Omega)}^2 \|r(u_i^\varepsilon)\|_{L^2(\Omega)}^2 \\
& \le C \left( \sup_{0\le t\le T} \|r(u_i^\varepsilon(t))\|_{L^2(\Omega)}^2 \right) \int_\Omega |\nabla r(u_i^\varepsilon)|^2 \\
& = \frac{Cp^2}{4} \left( \sup_{0\le t\le T} \int_\Omega (u_i^\varepsilon(t))^p \right)  \int_\Omega  (u_i^\varepsilon)^{p-2} |\nabla u_i^\varepsilon|^2. 
\end{aligned}
\label{Lem.uEps.LInftyL1+.Proof1a} 
\end{equation}
We imply that 
\begin{align}
\intQT (u_i^\varepsilon)^{2p} 
& \le \frac{Cp^2}{4} \left( \sup_{0\le t\le T} \int_\Omega (u_i^\varepsilon(t))^p \right)  \intQT (u_i^\varepsilon)^{p-2} |\nabla u_i^\varepsilon|^2.   
\label{Lem.uEps.LInftyL1+.Proof1}
\end{align}
Due to a similar computation as \eqref{Lem.Feedback3to12.Proof1}, we have  
\begin{equation}
\begin{aligned} 
& \sup_{0\le t\le T} \int_\Omega (u_i^\varepsilon(t))^p  +   d_ip(p-1)  \intQT (u_i^\varepsilon)^{p-2} |\nabla u_i^\varepsilon|^2 +   p \alpha_i \intQT (u_i^\varepsilon)^{p+1}  \\
& \le  \int_\Omega u_{i0}^p + p \alpha_i \intQT (u_i^\varepsilon)^p + (p-1)\chi_i \intQT (u_i^\varepsilon)^{p} \Delta v_3^\varepsilon \\
& \le \int_\Omega u_{i0}^p + p \alpha_i \intQT (u_i^\varepsilon)^p + (p-1) \chi_i \|\Delta v_3^\varepsilon\|_{L^2(Q_T)} \left( \intQT (u_i^\varepsilon)^{2p} \right)^{1/2}. 
\end{aligned}
\label{Lem.uEps.LInftyL1+.Proof2}
\end{equation}
Therefore, a combination of \eqref{Lem.uEps.LInftyL1+.Proof1} and \eqref{Lem.uEps.LInftyL1+.Proof2} gives
\begin{equation}
\begin{aligned} 
& \sup_{0\le t\le T} \int_\Omega (u_i^\varepsilon(t))^p  +   d_ip(p-1)  \intQT (u_i^\varepsilon)^{p-2} |\nabla u_i^\varepsilon|^2   + p \alpha_i \intQT (u_i^\varepsilon)^{p+1}  \\
& \le \int_\Omega u_{i0}^p + \frac{p \alpha_i}{2} \intQT (u_i^\varepsilon)^{p+1} + C_{p,\alpha_i,T} \\
& + (p-1) \frac{C_T \sqrt{C} p \chi_i}{2}    \left( \sup_{0\le t\le T} \int_\Omega (u_i^\varepsilon(t))^p \right)^{1/2} \left( \intQT (u_i^\varepsilon)^{p-2} |\nabla u_i^\varepsilon|^2 \right)^{1/2} \\
& \le \int_\Omega u_{i0}^p + \frac{p \alpha_i}{2} \intQT (u_i^\varepsilon)^{p+1} + C_{p,\alpha_i,T} \\
& + (p-1) \frac{C_T^2C p \chi_i^2}{8d_i}     \left( \sup_{0\le t\le T} \int_\Omega (u_i^\varepsilon(t))^p \right) + \frac{d_ip(p-1)}{2}  \intQT (u_i^\varepsilon)^{p-2} |\nabla u_i^\varepsilon|^2 , 
\end{aligned}
\label{Lem.uEps.LInftyL1+.Proof3}
\end{equation}
where, by Lemma~\ref{Lem.vEps.L2H2},   
$$\|\Delta v_3^\varepsilon(t)\|_{L^2(Q_T)} \le C_T, $$
and by the Young inequality
\begin{align*}
(u_i^\varepsilon)^p \le \frac{1}{2} (u_i^\varepsilon)^{p+1} + C_{p}. 
\end{align*}
It follows from \eqref{Lem.uEps.LInftyL1+.Proof3}
 that 
 \begin{equation}
\begin{aligned} 
& \left( \sup_{0\le t\le T} \int_\Omega (u_i^\varepsilon(t))^p \right)  +  \frac{d_ip(p-1)}{2} \intQT \big| \nabla u_i^\varepsilon \big|^2   +  \frac{p\alpha_i}{2} \intQT (u_i^\varepsilon)^{p+1}  \\
& \le C_{p,T} + C_{p,T} \int_\Omega u_{i0}^p   +  (p-1)  \frac{C_T^2C p \chi_i^2}{8d_i}    \left( \sup_{0\le t\le T} \int_\Omega (u_i^\varepsilon(t))^p \right) . 
\end{aligned}
\label{Lem.uEps.LInftyL1+.Proof4}
\end{equation}
One can observe the limit 
\begin{align*}
\lim_{p\searrow 1} (p-1)  \frac{C_T^2C p \chi_i^2}{8d_i}  = 0,
\end{align*}
which allows us to choose $p=1+\delta$, with an enough small constant $\delta>0$,  such that $(p-1)C_{p,T}<1$ to get   
\begin{align*}
\sup_{\varepsilon>0} \left(  \sup_{0\le t\le T} \int_\Omega (u_i^\varepsilon(t))^{1+\delta} + \intQT (u_i^\varepsilon)^{2+\delta}  +   \intQT \big|\nabla \big((u_i^\varepsilon)^{\frac{1+\delta}{2}}\big)\big|^2 \right)   \le C_T . 
\end{align*} 
Finally, by the Young inequality again, we have 
\begin{align*}
\big|\nabla u_i^\varepsilon\big|^{\frac{4+2\delta}{3}} \le C \left(  |u_i^\varepsilon|^{2+\delta} +  \big|\nabla \big((u_i^\varepsilon)^{\frac{1+\delta}{2}}\big)\big|^2 \right),  
\end{align*}
and directly obtain a uniform bound for $|\nabla u_i^\varepsilon|$ in $L^{(4+2\delta)/3}(Q_T)$. 
\end{proof}

\begin{lemma} \label{Lem.vEps.LInftyLp} Let $\delta>0$ be obtained by Lemma~\ref{Lem.uEps.LInftyL1+}. Then,   
\begin{align*}
\sup_{\varepsilon>0} \left(  \sup_{0\le t\le T} \int_\Omega (v_i^\varepsilon(t))^{1+\delta} + \sup_{0\le t\le T} \int_\Omega (v_3^\varepsilon(t))^{p}  \right)   \le C_T,
\end{align*}
for $1\le p<\infty$ and $i=1,2$.  
\end{lemma} 

\begin{proof} We first estimate $v_i^\varepsilon$ for $i=1,2$. Multiplying the equations for $v_i^\varepsilon$ by $(v_i^\varepsilon)^{\delta}$ gives 
\begin{align*}
 \mu_i \int_\Omega (v_i^\varepsilon)^{1+\delta}  & = - \lambda_i \delta \int_\Omega (v_i^\varepsilon)^{\delta-1} |\nabla v_i^\varepsilon|^2 + \int_\Omega u_i^\varepsilon (v_i^\varepsilon)^{\delta} \\
& \le \frac{1}{1+\delta} \left( \frac{2\delta}{\mu_i(1+\delta)} \right)^\delta \int_\Omega (u_i^\varepsilon)^{1+\delta} + \frac{\mu_i}{2} \int_\Omega (v_i^\varepsilon)^{1+\delta}.
\end{align*}
Therefore, by Lemma~\ref{Lem.uEps.LInftyL1+}, 
\begin{align*}
\sup_{0\le t\le T} \int_\Omega (v_i^\varepsilon)^{1+\delta} \le \frac{1}{\delta}  \left( \frac{2\delta}{\mu_i(1+\delta)} \right)^{1+\delta} \sup_{0\le t\le T} \int_\Omega (u_i^\varepsilon)^{1+\delta} \le C_T .
\end{align*}
Now, we proceed to estimate $v_3^\varepsilon$. The case $p=1$ is straightforward since $u_3^\varepsilon$ is uniformly bounded in $L^\infty(0,T;L^1(\Omega))$. Hence, it is only necessary to consider $p>1$.  It follows from the equation for $v_3^\varepsilon$ that 
\begin{align}
\varepsilon \frac{d}{dt} \int_\Omega (v_3^\varepsilon)^p + \frac{4\lambda_3(p-1)}{p} \int_\Omega |\nabla r(v_3^\varepsilon) |^2 + \mu_3 p \int_\Omega (v_3^\varepsilon)^p = p\int_\Omega u_3^\varepsilon (v_3^\varepsilon)^{p-1} ,
\label{Lem.vEps.LInftyLp.Proof1} 
\end{align}
where $r(v_3^\varepsilon):= (v_3^\varepsilon)^{p/2}$. 
To establish an $L^\infty(0,T;L^p(\Omega))$-estimate for $v_3^\varepsilon$, we will estimate the right-hand side of \eqref{Lem.vEps.LInftyLp.Proof1} first. 
 Recalling that $u_3^\varepsilon$ is bounded uniformly in $L^\infty(0,T;L^{1+\delta}(\Omega))$. Therefore, we can choose $\delta_1$ small enough such that  
\begin{align}\label{eq:delta1}
0<\delta_1 < \min\{ \delta; p-1 \} . 
\end{align}
By the Gagliardo–Nirenberg interpolation inequality, 
\begin{align*}
\|r(v_3^\varepsilon)\|_{L^{\frac{2(1+\delta_1)(p-1)}{p\delta_1}}(\Omega)}^{\frac{2(p-1)}{p}}  \le C\|\nabla r(v_3^\varepsilon)\|_{L^2(\Omega)}^{\frac{2}{1+\delta_1}-\frac{2}{p}}  \|r(v_3^\varepsilon)\|_{L^2(\Omega)}^{\frac{2\delta_1}{1+\delta_1}} . 
\end{align*}
Then, by the H\"older inequality, 
\begin{align*}
\int_\Omega u_3^\varepsilon (v_3^\varepsilon)^{p-1} & \le \|u_3^\varepsilon\|_{L^{1+\delta_1}(\Omega)} \|(v_3^\varepsilon)^{p-1}\|_{L^{\frac{1+\delta_1}{\delta_1}}(\Omega)} \\
& \le  \|u_3^\varepsilon\|_{L^\infty(0,T;L^{1+\delta_1}(\Omega))}  \|r(v_3^\varepsilon)\|_{L^{\frac{2(1+\delta_1)(p-1)}{\delta_1p}}(\Omega)}^{\frac{2(p-1)}{p}} \\
& \le C \|u_3^\varepsilon\|_{L^\infty(0,T;L^{1+\delta_1}(\Omega))}   \|\nabla r(v_3^\varepsilon)\|_{L^2(\Omega)}^{\frac{2}{1+\delta_1}-\frac{2}{p}}  \|r(v_3^\varepsilon)\|_{L^2(\Omega)}^{\frac{2\delta_1}{1+\delta_1}} \\
& = C \|u_3^\varepsilon\|_{L^\infty(0,T;L^{1+\delta_1}(\Omega))}   \|\nabla r(v_3^\varepsilon)\|_{L^2(\Omega)}^{\frac{2}{1+\delta_1}-\frac{2}{p}}  \|v_3^\varepsilon\|_{L^p(\Omega)}^{\frac{p\delta_1}{1+\delta_1}}. 
\end{align*} 
Since, by Lemma~\ref{Lem.uEps.LInftyL1+}, 
\begin{align*}
\|u_3^\varepsilon\|_{L^\infty(0,T;L^{1+\delta_1}(\Omega))}   \le C_T,
\end{align*}
we can use the Young inequality as follows  
\begin{equation}
\begin{aligned}
\int_\Omega u_3^\varepsilon (v_3^\varepsilon)^{p-1} & \le C_T \|\nabla r(v_3^\varepsilon)\|_{L^2(\Omega)}^{\frac{2}{1+\delta_1}-\frac{2}{p}}  \|v_3^\varepsilon\|_{L^p(\Omega)}^{\frac{p\delta_1}{1+\delta_1}} \\
& \le \frac{2\lambda_3(p-1)}{p} \int_\Omega |\nabla r(v_3^\varepsilon)|^{2}   + C_{p,\delta_1,\lambda_3,T} \|v_3^\varepsilon\|_{L^p(\Omega)}^{\frac{p^2\delta_1}{p\delta_1+1+\delta_1}}  \\
& = \frac{2\lambda_3(p-1)}{p} \int_\Omega |\nabla r(v_3^\varepsilon)|^{2}   + C_{p,\delta_1,\lambda_3,T} \left( \int_\Omega (v_3^\varepsilon)^p \right)^{\frac{p\delta_1}{p\delta_1+1+\delta_1}}. 
\end{aligned}
\label{Lem.vEps.LInftyLp.Proof2} 
\end{equation}
Here, we note $p>1+\delta_1$ by the definition of $\delta_1$ in \eqref{eq:delta1}.  
By the Young inequality again, 
\begin{equation}
\begin{aligned}
C_{p,\delta_1,\lambda_3,T} \left( \intQT (v_3^\varepsilon)^p \right)^{\frac{p\delta_1}{p\delta_1+1+\delta_1}} \le \frac{\mu_3p}{2} \intQT (v_3^\varepsilon)^p + C_{p,T} 
\end{aligned}
\label{Lem.vEps.LInftyLp.Proof3} 
\end{equation}
in which $C_{p,T}$ also depends on $\delta_1,\lambda_3$. 
Let us combine the estimates \eqref{Lem.vEps.LInftyLp.Proof1}-\eqref{Lem.vEps.LInftyLp.Proof3}. This gives 
\begin{align*}
\varepsilon \frac{d}{dt} \int_\Omega (v_3^\varepsilon)^p   + \frac{\mu_3p}{2} \int_\Omega (v_3^\varepsilon)^p \le C_{T,p,\delta_1,\lambda_3}, 
\end{align*}
which, after applying the Gr\"onwall inequality, yields that the component $(v_3^\varepsilon)^p$ is uniformly bounded in $L^\infty(0,T;L^p(\Omega))$.  
\end{proof}

\subsection{Feedback argument via maximal regularity with slow evolution}  
\label{Sec:FSL:Feedback}

We will establish feedback arguments between species, where the feedback from the prey species to the predator one is obtained similarly to Lemma~\ref{Lem.Feedback12to3}. However, the feedback in the opposite direction, i.e. the feedback from the predator species to the prey one, is not clear, where we need the maximal regularity for the parabolic equation with slow evolution. The difficulty here is the vanishing of the parabolicity as $\varepsilon \to 0$. However, this can be overcome based on the idea in \cite[Lemma 2.5]{tang2023rigorous}.

\begin{lemma}[Maximal regularity with slow evolution]
\label{Lem.MRSlowEvolution} 
 Let $\varepsilon>0$, $\lambda>0$, $\mu>0$, and $w^\varepsilon$ be the solution to the problem
\begin{align}\label{System.SmallEvolution}
\left\{ \begin{array}{rlllll}
\varepsilon \partial_t w^\varepsilon - \lambda \Delta w^\varepsilon + \mu w^\varepsilon &=&   h^\varepsilon & \text{in } \Omega \times(0,T),  \\
\nabla w^\varepsilon \cdot \nu  &=& 0 & \text{on } \Gamma\times(0,T),  \\
w^\varepsilon(0) &=& w_0   & \text{on } \Omega.
\end{array}
\right.
\end{align}
For $1<p<\infty$,  
\begin{align}
 \|\Delta w^\varepsilon \|_{L^p(Q_T)}  \le  \left( \frac{\varepsilon}{p\mu} \right)^{\frac{1}{p}} \|\Delta w_0\|_{L^p(\Omega)} +  \frac{C_{1,\mu/\lambda,p}^{\mathsf{PM}}}{\lambda}  \|h^\varepsilon\|_{L^p(Q_T)}  ,
\label{Estimate.MRSlowEvol}  
\end{align}
where $C_{1,\mu/\lambda,p}^{\mathsf{PM}}$ is defined by Lemma~\ref{Lem.ParaMaximalRegularity}.
\end{lemma}

\begin{proof} With $t\in (0,T)$, let us consider the scaling 
\begin{align*}
\tau := \frac{\lambda}{\varepsilon} t, \quad \overline w^\varepsilon(x,\tau):= w^\varepsilon \bigg(x,\frac{\varepsilon}{\lambda}\tau\bigg) \quad \text{for } \tau \in \bigg(0,\frac{\lambda T}{\varepsilon}\bigg).
\end{align*}
We recast Problem \eqref{System.SmallEvolution} as   
\begin{align} 
\label{System.SlowEvolScaling}
\left\{ \begin{array}{rlllll}
\partial_\tau \overline w^\varepsilon - \Delta \overline  w^\varepsilon + \dfrac{\mu}{\lambda} \overline w^\varepsilon &=&  \dfrac{1}{\lambda} \, h^\varepsilon & \text{in } \Omega \times(0,\lambda T/\varepsilon),  \\
\nabla \overline w^\varepsilon \cdot \nu  &=& 0 & \text{on } \Gamma \times(0,\lambda T/\varepsilon),  \\
\overline w^\varepsilon(0) &=& w_0   & \text{on } \Omega.
\end{array}
\right.
\end{align}
By the Duhamel's principle, the solution to \eqref{System.SlowEvolScaling} can be split into the sum of $\overline w_1^\varepsilon$ and  $\overline w_2^\varepsilon$, where $\overline w_1^\varepsilon(x,\tau)=e^{\tau(\Delta-(\mu/\lambda)I)} w_0(x)$ is the homogeneous solution to  \eqref{System.SlowEvolScaling}, and  $\overline w_2^\varepsilon$ is the solution to \eqref{System.SlowEvolScaling} with zero initial data, i.e. 
\begin{align*} 
\left\{ \begin{array}{rlllll}
\partial_\tau \overline w^\varepsilon_2 - \Delta \overline  w^\varepsilon_2  + \dfrac{\mu}{\lambda} \overline w^\varepsilon_2  &=&  \dfrac{1}{\lambda} \, h^\varepsilon  & \text{in } \Omega \times(0,\lambda T/\varepsilon),  \\
\nabla \overline w^\varepsilon_2 \cdot \nu  &=& 0 & \text{on } \Gamma \times(0,\lambda T/\varepsilon),  \\
\overline w^\varepsilon_2(0) &=& 0   & \text{on } \Omega.
\end{array}
\right.
\end{align*}
For the component $\overline w_2^\varepsilon$, we apply Lemma~\ref{Lem.ParaMaximalRegularity} to see that  
\begin{align*}
\|\Delta \overline w_2^\varepsilon\|_{L^p(\Omega\times(0,\, \lambda T/\varepsilon))} \le \frac{C_{1,\mu/\lambda,p}^{\mathsf{PM}}}{\lambda} \| h^\varepsilon \|_{L^p(\Omega\times(0,\, \lambda T/\varepsilon))}. 
\end{align*}
Moreover, by the contraction property of the Neumann heat semigroup (cf. \eqref{HeatSemigroup.Contraction}),     
\begin{align*}
 \|\Delta \overline w_1^\varepsilon\|_{L^p(\Omega\times(0,\, \lambda T/\varepsilon))}   & = \left( \int_0^{\frac{\lambda T}{\varepsilon}} \|e^{\tau(\Delta-(\mu/\lambda)I)} \Delta w_0\|_{L^p(\Omega)}^p d\tau \right)^{\frac{1}{p}} \\ 
& \le  \left( \int_0^{\frac{\lambda T}{\varepsilon}} e^{- p(\mu/\lambda)\tau}  d\tau \right)^{\frac{1}{p}} \|\Delta w_0\|_{L^p(\Omega)}  
 \le \left( \frac{\lambda}{p\mu} \right)^{\frac{1}{p}} \|\Delta w_0\|_{L^p(\Omega)} .
\end{align*}
Then, $\Delta \overline w^\varepsilon$ is estimated as 
\begin{align*}
\|\Delta \overline w^\varepsilon\|_{L^p(\Omega\times(0,\, \lambda T/\varepsilon))}  & \le \|\Delta \overline w_1^\varepsilon\|_{L^p(\Omega\times(0,\, \lambda T/\varepsilon))}  + \|\Delta \overline w_2^\varepsilon\|_{L^p(\Omega\times(0,\, \lambda T/\varepsilon))} \\
& \le \left( \frac{\lambda}{p\mu} \right)^{\frac{1}{p}} \|\Delta w_0\|_{L^p(\Omega)} +  \frac{C_{1,\mu/\lambda,p}^{\mathsf{PM}}}{\lambda} \| h^\varepsilon \|_{L^p(\Omega\times(0,\, \lambda T/\varepsilon))},  
\end{align*}
which, after a change of variables, gives 
\begin{align*}
 \left( \frac{\lambda}{\varepsilon} \right)^{\frac{1}{p}} \|\Delta w^\varepsilon \|_{L^p(Q_T)}  \le \left( \frac{\lambda}{p\mu} \right)^{\frac{1}{p}} \|\Delta w_0\|_{L^p(\Omega)} +  \frac{C_{1,\mu/\lambda,p}^{\mathsf{PM}}}{\lambda} \left( \frac{\lambda}{\varepsilon} \right)^{\frac{1}{p}} \|h^\varepsilon\|_{L^p(Q_T)} ,  
\end{align*} 
and therefore, the estimate \eqref{Estimate.MRSlowEvol} is proved. 
\end{proof}

\begin{lemma}[Feedback from prey to predator] \label{Lem.Eps.Feedback12to3} If there exists $\delta_0>0$ such that
\begin{align}
\sup_{\varepsilon>0}  \intQT  (u_i^\varepsilon)^{2+\delta_0}  \le C_T, \quad i=1,2,
\label{Lem.Eps.Feedback12to3.State1}
\end{align}
then
\begin{align}
\label{Lem.Eps.Feedback12to3.State2}
\sup_{\varepsilon>0} \left( \sup_{0\le t\le T} \int_\Omega (u_3^\varepsilon(t))^{1+\delta_0} +  \intQT  (u_3^\varepsilon)^{2+\delta_0} \right) \le C_T .
\end{align} 
\end{lemma}

\begin{proof} This lemma can be proved similarly to Lemma~\ref{Lem.Feedback12to3} based on Lemma~\ref{Lem.MaximalRegularity}. 
\end{proof}

\begin{lemma}[Feedback from predator to prey]
\label{Lem.Eps.Feedback}
If there exists $\delta_*>0$ such that
\begin{align}
\sup_{\varepsilon>0}   \intQT (u_3^\varepsilon)^{2+\delta_*}   \le C_T,  
\label{Lem.Eps.Feedback.State1}
\end{align}
then, for $i=1,2$,
\begin{align}
\sup_{\varepsilon>0} \left(  \sup_{0\le t\le T} \int_\Omega (u_i^\varepsilon(t))^{1+2\delta_*} + \intQT (u_i^\varepsilon)^{2+2\delta_*}  \right)   \le C_T.
\label{Lem.Eps.Feedback.State2}
\end{align} 
\end{lemma}

\begin{proof} Firstly, we prove that $u_1^\varepsilon,u_2^\varepsilon$ are uniformly bounded in $L^{2+2\delta_*}(Q_T)$. By applying Lemma~\ref{Lem.MRSlowEvolution} to the equation for $v_3^\varepsilon$, i.e.
\begin{align*}
\varepsilon \partial_t v_3^\varepsilon - \lambda_3 \Delta v_3^\varepsilon + \mu_3 v_3^\varepsilon = u_3^\varepsilon , \quad (\nabla v_3^\varepsilon \cdot \nu)|_{\Gamma_T} = 0, \quad v_3^\varepsilon (0)  = v_{30},
\end{align*}
 we obtain, thanks to \eqref{Lem.Eps.Feedback.State1},  
\begin{equation}
\begin{aligned}
 \sup_{\varepsilon>0} \|\Delta v_3^\varepsilon \|_{L^{2+\delta_*}(Q_T)}  
& \le C_{\mu_3,\delta_*} \|\Delta v_{30}\|_{L^{2+\delta_*}(\Omega)} +   \frac{C_{1,\mu_3/\lambda_3,2+\delta_*}^{\mathsf{PM}}}{\lambda_3} \sup_{\varepsilon>0}  \|u_3^\varepsilon\|_{L^{2+\delta_*}(Q_T)} \\
& \le  C_{\mu_3,\delta_*} \|\Delta v_{30}\|_{L^{2+\delta_*}(\Omega)} +   \frac{C_{1,\mu_3/\lambda_3,2+\delta_*}^{\mathsf{PM}}}{\lambda_3} C_T ,  
\end{aligned}
\label{Lem.Eps.Feedback.Proof1}
\end{equation}
where $C_{1,\mu_3/\lambda_3,2+\delta_*}^{\mathsf{PM}}$ is defined by Lemma~\ref{Lem.ParaMaximalRegularity}, and $C_{\mu_3,\delta_*}:= ( (2+\delta_*)\mu_3 )^{-1/(2+\delta_*)}$. Now, by employing the computation \eqref{Lem.uEps.LInftyL1+.Proof2} with $p=1+\delta_*$, 
\begin{align*} 
& \sup_{0\le t\le T} \int_\Omega (u_i^\varepsilon(t))^{1+\delta_*}  +   \frac{4d_i\delta_*}{1+\delta_*}  \intQT |\nabla (u_i^\varepsilon)^{\frac{1+\delta_*}{2}}|^2 + (1+\delta_*) \alpha_i \intQT (u_i^\varepsilon)^{2+\delta_*}  \\
& \le  \int_\Omega u_{i0}^{1+\delta_*} + (1+\delta_*) \alpha_i \intQT (u_i^\varepsilon)^{1+\delta_*} + \delta_* \chi_i \intQT (u_i^\varepsilon)^{1+\delta_*} \Delta v_3^\varepsilon \\
& \le \int_\Omega u_{i0}^{1+\delta_*} + (1+\delta_*) \alpha_i \intQT (u_i^\varepsilon)^{1+\delta_*} + \delta_* \chi_i \|u_i^\varepsilon\|_{L^{2+\delta_*}(Q_T)}^{1+\delta_*} \|\Delta v_3^\varepsilon\|_{L^{2+\delta_*}(Q_T)},
\end{align*}
for $i=1,2$, which after taking into account the estimate \eqref{Lem.Eps.Feedback.Proof1} gives   
\begin{align*}
\sup_{0\le t\le T} \int_\Omega (u_i^\varepsilon(t))^{1+\delta_*}  +  \intQT |\nabla (u_i^\varepsilon)^{\frac{1+\delta_*}{2}}|^2 \le C_T.
\end{align*}
Then, using the Gagliardo–Nirenberg interpolation inequality in the same way as \eqref{Lem.uEps.LInftyL1+.Proof1a},  
\begin{align*}
\intQT (u_i^\varepsilon)^{2+2\delta_*} 
& \le C \left( \sup_{0\le t\le T} \int_\Omega (u_i^\varepsilon(t))^{1+\delta_*} \right) \intQT |\nabla (u_i^\varepsilon)^{\frac{1+\delta_*}{2}}|^2 \le C_T,  
\end{align*}
i.e. $u_1^\varepsilon, u_2^\varepsilon$ are uniformly bounded in $L^{2+2\delta_*}(Q_T)$, and so is $\Delta v_3^\varepsilon$ by applying Lemma~\ref{Lem.MRSlowEvolution}. This allows us to take $p=1+2\delta_*$ in  \eqref{Lem.uEps.LInftyL1+.Proof2},  which gives 
\begin{align*}
\sup_{\varepsilon>0}  \left( \sup_{0\le t\le T}  \int_\Omega (u_i^\varepsilon(t))^{1+2\delta_*} \right) \le C_T,
\end{align*} 
i.e. the estimate \eqref{Lem.Eps.Feedback.State2} is completely proved.
\end{proof}

\subsection{Uniform boundedness in $L^\infty(Q_T)$}
\label{Sec:UniBoundLInf}

Taking the uniform improved regularity and the feedback arguments in the previous sections, we will obtain the uniform boundedness of the $\varepsilon$-depending solution in $L^\infty(Q_T)$. 

\begin{lemma} 
\label{Lem.UVeps.Bootstrap}
For any $1\le q<\infty$ yields
\begin{gather}
 \sup_{\varepsilon>0} \left( \sum_{i=1}^3 \| u_i^\varepsilon \|_{L^\infty(Q_T)} + \sum_{i=1}^3 \| u_i^\varepsilon \|_{L^2(0,T;H^1(\Omega))}   \right)   \le C_T,   \label{Lem.UVeps.Bootstrap.State1} \\
 \displaystyle \sup_{\varepsilon>0} \left( \sum_{i=1}^2 \| v_i^\varepsilon \|_{L^\infty(0,T;W^{2,\infty}(\Omega))} + \| v_3^\varepsilon \|_{L^\infty(Q_T)}  +   \| v_3^\varepsilon \|_{L^q(0,T;W^{2,q}(\Omega))}  \right)  \le C_T.
 \label{Lem.UVeps.Bootstrap.State2}
\end{gather}
\end{lemma}

\begin{proof} First of all, we prove \eqref{Lem.UVeps.Bootstrap.State1} by performing a bootstrap argument. By Lemma~\ref{Lem.uEps.LInftyL1+}, we have the uniform boundedness of $u_1^\varepsilon,u_2^\varepsilon,u_3^\varepsilon$ in $L^{2+\delta}(Q_T)$, $\delta>0$. Then, thanks to Lemma~\ref{Lem.Eps.Feedback}, $u_1^\varepsilon,u_2^\varepsilon$ are uniformly bounded in $L^\infty(0,T;L^{2+\delta_1}(\Omega)) \cap L^{2+\delta_1}(Q_T)$, with $\delta_1:=2\delta$, and therefore, thanks to Lemma~\ref{Lem.Eps.Feedback12to3}, so is $u_3^\varepsilon$. 
By iterating this argument,  
\begin{align*}
\sup_{\varepsilon>0} \left( \sup_{0\le t\le T}  \int_\Omega (u_i^\varepsilon(t))^{1+\delta_n} +  \intQT (u_i^\varepsilon)^{2+\delta_n}  \right)   \le C_T  
\end{align*}
with $\delta_n=2^n\delta$. Since $\lim_{n\to \infty} \delta_n=\infty$,  
\begin{align}
\sup_{\varepsilon>0} \left( \sum_{i=1}^3 \|u_i^\varepsilon \|_{L^\infty(0,T;L^q(\Omega))} + \sum_{i=1}^3 \|u_i^\varepsilon \|_{L^q(Q_T)}    \right)   \le C_T  
\label{Lem.UVeps.Bootstrap.Proof1}
\end{align}
for any $1\le q<\infty$. By the smoothing effect of the heat semigroup with the same techniques as in the proof of Theorem~\ref{Theo.GlobalClassicalSol}, one can show that 
\begin{align*}
\sup_{\varepsilon>0} \sum_{i=1}^3 \|u_i^\varepsilon \|_{L^\infty(Q_T)}  
\le C_T  .
\end{align*}
The proof of Lemma~\ref{Lem.uEps.LInftyL1+} gives that $u_1^\varepsilon,u_2^\varepsilon,u_3^\varepsilon$ are uniformly bounded in $L^2(0,T;H^1(\Omega))$, i.e., \eqref{Lem.UVeps.Bootstrap.State1} is claimed.

\medskip

Now, we prove \eqref{Lem.UVeps.Bootstrap.State2}. For any $1< q<\infty$,   we claim that $(v_1^\varepsilon,v_2^\varepsilon)$ is uniformly bounded in $L^\infty(0,T;W^{2,q}(\Omega))$ by Lemma~\ref{Lem.MaximalRegularity}, and $v_3^\varepsilon$ in $L^q(0,T;W^{2,q}(\Omega))$ by Lemma~\ref{Lem.MRSlowEvolution}. The uniform boundedness then also guarantees that $(v_1^\varepsilon,v_2^\varepsilon)$ is uniformly bounded in $L^\infty(Q_T)^2$ due to the embedding $W^{2,q}(\Omega)\hookrightarrow L^\infty(\Omega)$ for a large $q$, and so is $(\Delta v_1^\varepsilon, \Delta v_2^\varepsilon)$.  
It remains to prove 
\begin{align}
\sup_{\varepsilon>0} \| v_3^\varepsilon \|_{L^\infty(Q_T)}  \le C_T .
\label{Lem.UVeps.Bootstrap.Proof3}
\end{align}
Indeed, multiplying the two sides of the equation for $v_3^\varepsilon$ by $(v_3^\varepsilon)^{p-1}$ gives   
\begin{align*}
& \mu_3 \intQT (v_3^\varepsilon)^p \\
&= \frac{\varepsilon}{p} \int_\Omega v_{30}^p - \frac{\varepsilon}{p} \int_\Omega (v_3^\varepsilon(T))^p  - (p-1) \lambda_3 \intQT (v_3^\varepsilon)^{p-2} |\nabla v_3^\varepsilon|^2 + \intQT u_3^\varepsilon (v_3^\varepsilon)^{p-1} \\
&\le  \frac{\varepsilon}{p} \int_\Omega v_{30}^p - \frac{\varepsilon}{p} \int_\Omega (v_3^\varepsilon(T))^p - (p-1) \lambda_3 \intQT (v_3^\varepsilon)^{p-2} |\nabla v_3^\varepsilon|^2 \\
& \hspace*{4.6cm} + \frac{1}{p} \mu_3^{-(p-1)} \intQT (u_3^\varepsilon)^p + \frac{p-1}{p} \mu_3 \intQT (v_3^\varepsilon)^{p}.
\end{align*} 
This yields 
\begin{align*}
\sup_{\varepsilon>0} \|v_3^\varepsilon\|_{L^p(Q_T)}  & \le  \left(  \frac{1}{\mu_3} \|v_{30}\|_{L^p(\Omega)}^p + \frac{1}{\mu_3^{p}}  \sup_{\varepsilon>0} \|u_3^\varepsilon\|_{L^p(Q_T)}^p \right)^{\frac{1}{p}} \\
& \le \frac{1}{\mu_3^{1/p}} \|v_{30}\|_{L^p(\Omega)} + \frac{1}{\mu_3} \sup_{\varepsilon>0} \|u_3^\varepsilon\|_{L^p(Q_T)},   
\end{align*}
which subsequently shows \eqref{Lem.UVeps.Bootstrap.Proof3} after sending $p$ to infinity.
\end{proof}
 
\subsection{Weak-to-strong convergence, passing to the limit}
\label{Sec:Weak2StrongPassing}

We will pass from the system \eqref{System.UV}-\eqref{Condition.Boun-Init.Epsilon} to the reduced system \eqref{System.UV.Limit}-\eqref{Condition.Boun-Init.Limit}, which is given in Theorem~\ref{Theo.FastSignalDiffLimit}. Due to the lack of time derivative in the equations for $v_1^\varepsilon, v_2^\varepsilon$, and the vanishing of the parabolicity in the equation for $v_3^\varepsilon$ (or more precisely, $\varepsilon \partial_t v_3^\varepsilon \to 0$ in some suitable sense), the establishment of strong convergence of $(v_1^\varepsilon,v_2^\varepsilon,v_3^\varepsilon)$ is non-standard. We overcome this difficulty in the following lemma, where we mainly use the energy equation method from \cite{ball2004global,henneke2016fast}.   

\begin{lemma}[Weak-to-strong convergence] 
\label{Lem:Weak-to-strong}
Assume that $h^\varepsilon$ is uniformly bounded in $L^\infty(0,T;L^2(\Omega))$. For each $\varepsilon>0$, let $w^\varepsilon$ be the weak solution to \eqref{System.SmallEvolution}, i.e., 
\begin{align}
\varepsilon \int_0^T \langle \partial_t w^\varepsilon, \phi \rangle + \lambda \intQT \nabla w^\varepsilon \cdot \nabla \phi + \mu \intQT w^\varepsilon\phi = \intQT h^\varepsilon \phi,
\label{WeakForm:wEps}  
\end{align}
for all $\phi \in C^\infty(\Omega \times [0,T))$, and $w^\varepsilon(0)=w_0 \in L^2(\Omega)$. If,  
\begin{align}
\left\{ 
\begin{array}{rllc}
h^\varepsilon &\rightarrow & h & \text{strongly in } L^2(0,T;L^2(\Omega)), \\ 
w^\varepsilon &\rightharpoonup& w & \text{weakly in } L^2(0,T;H^1(\Omega)),
\end{array} \right.
\label{Lem:WeakToStrong:State1}
\end{align} 
and $w$ is a weak solution to the problem $- \lambda \Delta w + \mu w = h$, $(\nabla w\cdot \nu)_{\Gamma_T}=0$, i.e.,
\begin{align}
\lambda \intQT \nabla w \cdot \nabla \phi + \mu \intQT w\phi = \intQT h \phi, \quad \forall \phi \in C_c^\infty(\Omega \times (0,T)) .
\label{WeakForm:w} 
\end{align}
Then, 
\begin{align} 
w^\varepsilon \rightarrow w \text{ strongly in } L^2(0,T;H^1(\Omega)). 
\label{Lem:WeakToStrong:State2}
\end{align}
\end{lemma}

\begin{proof} By choosing $\varphi=w^\varepsilon$ in the weak formulation \eqref{WeakForm:wEps}, we have 
\begin{gather}
\frac{\varepsilon}{2} \int_\Omega (w^\varepsilon(T))^2 - \frac{\varepsilon}{2} \int_\Omega w_0^2 +  \intQT (\lambda |\nabla w^\varepsilon|^2 + \mu (w^\varepsilon)^2) = \intQT h^\varepsilon w^\varepsilon. 
\label{Lem:WeakToStrong:Proof1}
\end{gather}
On the other hand, choosing $\varphi = w$ in the weak formulation  \eqref{WeakForm:w} gives
\begin{gather}
\intQT \left( \lambda |\nabla w|^2 + \mu w^2 \right) = \intQT h w .
\label{Lem:WeakToStrong:Proof2}
\end{gather}
Since $w^\varepsilon$ is uniformly bounded in $L^\infty(0,T;L^2(\Omega))$ and $w_0\in L^2(\Omega)$, the first two terms in \eqref{Lem:WeakToStrong:Proof1} tend to zero as $\varepsilon \to 0$. Moreover, due to the weak and strong convergences of $h^\varepsilon$ and $w^\varepsilon$ in $L^2(Q_T)$, respectively, the right-hand side of \eqref{Lem:WeakToStrong:Proof1} tends to the corresponding one of \eqref{Lem:WeakToStrong:Proof2}. Therefore, we obtain 
\begin{align}
\intQT \left( \lambda |\nabla w^\varepsilon|^2 + \mu (w^\varepsilon)^2 \right)   \rightarrow \intQT \left( \lambda |\nabla w|^2 + \mu w^2 \right),  
\label{Lem:WeakToStrong:Proof3}
\end{align}
where we note that $\varphi \mapsto \sqrt{\lambda}\|\nabla\varphi\|_{L^2(\Omega)}+\sqrt{\mu}\|\varphi\|_{L^2(\Omega)}$ forms an equivalent norm with the usual $H^1$-norm.  Here, $H^1(\Omega)$ is uniformly convex since $L^2(\Omega)$ is uniformly convex (\cite[Proof of Theorem 4.10]{Hbrezis}) and $H^1(\Omega)$ is a closed subspace of $L^2(\Omega)$. Then, thanks to \cite[Theorem 2]{day1941some}, the space $L^2(0,T;H^1(\Omega))$ is uniformly convex. Therefore, taking into account Proposition 3.32 in \cite{Hbrezis}, the strong convergence of $w^\varepsilon$ to $w$ in $L^2(0,T;H^1(\Omega))$ follows from \eqref{Lem:WeakToStrong:Proof3}. 
\end{proof}

We now define weak solutions to the limiting system \eqref{System.UV.Limit}-\eqref{Condition.Boun-Init.Limit}, which will be useful in proving Theorem~\ref{Theo.FastSignalDiffLimit}.  
 
\begin{definition}
\label{Definition.WeakSolution}
Let $T>0$.  A vector of non-negative functions $(u_i,v_i)_{i=1,2,3}$ is called a weak solution to the system \eqref{System.UV.Limit}-\eqref{Condition.Boun-Init.Limit} on the interval $(0,T)$ if 
\begin{align*}
(u_i,v_i)_{i=1,2,3} \in L^2(0,T;H^1(\Omega))^6, \quad (\partial_t u_i)_{i=1,2,3} \in L^2(0,T;(H^1(\Omega))')^3,
\end{align*}
and, for all $\varphi\in C^\infty(\Omega \times[0,T))$, it satisfies 
\begin{equation}
\begin{gathered}
-\iint_{\Omega_T} u_i \partial_t \phi - \intO u_{i0}\phi(0) + \iint_{\Omega_T} \big( d_i \nabla u_i + \chi_i u_i  \nabla v_3 \big) \cdot \nabla \phi = \iint_{\Omega_T} f_i \phi ,  \quad  i=1,2,  \\
-\iint_{\Omega_T} u_3 \partial_t \phi - \intO u_{30}\phi(0) + \iint_{\Omega_T} \bigg( d_3 \nabla u_3  - \sum_{j=1}^2 \chi_{3j} u_3 \nabla v_j \bigg) \cdot \nabla \phi = \iint_{\Omega_T} f_3 \phi  , \\
\lambda_j \intQT \nabla v_j \cdot \nabla \phi + \mu_j \intQT v_j \phi = \intQT u_j \phi,   \quad j=1,2,3, 
\end{gathered}
\label{System.Limit.WeakForm}
\end{equation}
 where $f_j:=f_j(u_1,u_2,u_3)$. On the other hand, the definition of a classical solution can be defined similarly to Definition~\ref{Definition.Solution}.
\end{definition}

Under Assumption~\ref{Assumption.Initial1}, we have the following lemma, which is a direct combination of Theorems 3.1 and 4.1 in Amorim-B{\"u}rger-Ordo{\~n}ez-Villada \cite{amorim2023global}. 

\begin{lemma}  
\label{Lem.MainResultsOfABOV23}
If $(u_i,v_i)_{i=1,2,3}$ is a global weak solution to \eqref{System.UV.Limit}-\eqref{Condition.Boun-Init.Limit}, then it will be the unique globally classical solution to \eqref{System.UV.Limit}-\eqref{Condition.Boun-Init.Limit}. Moreover, for any $T>0$, 
\begin{align*}
\sum_{i=1}^3 \|u_i\|_{L^\infty(Q_T)} \le C_T .
\end{align*}
\end{lemma}

We are now ready to prove Theorem~\ref{Theo.FastSignalDiffLimit}. 

\begin{proof}[Proof of Theorem~\ref{Theo.FastSignalDiffLimit}] By Lemma~\ref{Lem.UVeps.Bootstrap}, the sequence $\{(u_1^\varepsilon,u_2^\varepsilon,u_3^\varepsilon):\varepsilon>0\}$ is bounded in $L^2(0,T;H^1(\Omega))$. Due to the equations 
\begin{align*}
\partial_t u_i^\varepsilon = d_i \Delta u_i^\varepsilon +\chi_i \nabla u_i^\varepsilon \cdot \nabla v_3^\varepsilon + \chi_i  u_i^\varepsilon \Delta v_3^\varepsilon + f_i(u_1^\varepsilon,u_2^\varepsilon,u_3^\varepsilon), \, i=1,2,
\end{align*}
and 
\begin{align*}
\partial_t u_3^\varepsilon = d_3 \Delta u_3^\varepsilon - \sum_{i=1}^2 \left( \chi_{3i} \nabla u_3^\varepsilon \cdot  \nabla v_i^\varepsilon + \chi_{3i} u_3^\varepsilon  \Delta v_i^\varepsilon \right) + f_3(u_1^\varepsilon,u_2^\varepsilon,u_3^\varepsilon),
\end{align*}
as well as the boundedness of the sequences $\{\nabla u_i^\varepsilon\}$, $\{(\nabla v_i^\varepsilon,\Delta v_i^\varepsilon)\}$ in $L^2(Q_T)^2$ and $L^q(Q_T)^3$, $1\le q<\infty$, respectively, we have the boundedness of $\{\partial_t u_i^\varepsilon\}$ in
\begin{align*}
L^2(0,T;(H^1(\Omega))') + L^{\frac{2q}{q+2}}(Q_T) \hookrightarrow L^{\frac{2q}{q+2}}(0,T;X)
\end{align*}
with $$X:=(H^1(\Omega))' + L^{\frac{2q}{q+2}}(\Omega).$$
Note that  $H^1(\Omega)$ is embedded compactly into $L^r(\Omega)$ for any large $1\le r<\infty$, and therefore into $X$. 
Then, by the Aubin-Lions lemma, $u_i^\varepsilon$ strongly converges to $u_i$ in $L^2(Q_T)$. With the boundedness of $\{(u_1^\varepsilon,u_2^\varepsilon,u_3^\varepsilon)\}$ in $L^\infty(Q_T)^3$, cf. Lemma~\ref{Lem.UVeps.Bootstrap},  
\begin{align}
(u_1^\varepsilon,u_2^\varepsilon,u_3^\varepsilon) \to (u_1,u_2,u_3) \quad \text{strongly in } L^q(Q_T),
\label{Theo.FastSignalDiffLimit.Proof1}
\end{align}
for any $1\le q<\infty$, up to a subsequence (not relabelled).

\medskip

We consider the sequence $\{(v_1^\varepsilon,v_2^\varepsilon,v_2^\varepsilon)\}$. Since it is bounded in $L^2(0,T;H^1(\Omega))^3$, cf. Lemma~\ref{Lem.UVeps.Bootstrap}, we have 
\begin{align}
(v_1^\varepsilon,v_2^\varepsilon,v_3^\varepsilon) \rightharpoonup (v_1,v_2,v_3) \quad \text{weakly in } L^2(0,T;H^1(\Omega))^3.
\label{Theo.FastSignalDiffLimit.Proof1b}
\end{align}
Now, for all $\phi \in C^\infty(\Omega \times [0,T))$ and $i=1,2$, it follows from the system \eqref{System.UV}-\eqref{Condition.Boun-Init.Epsilon} that 
\begin{align*}
-\iint_{\Omega_T} u_i^\varepsilon \partial_t \phi - \intO u_{i0}\phi(0) + \iint_{\Omega_T} \big( d_i \nabla u_i^\varepsilon - \chi_i u_i^\varepsilon \nabla v_3^\varepsilon \big) \cdot \nabla \phi &= \iint_{\Omega_T} f_i^\varepsilon \phi ,  \\
-\iint_{\Omega_T} u_3^\varepsilon \partial_t \phi - \intO u_{30}\phi(0) + \iint_{\Omega_T} \bigg( d_3 \nabla u_3^\varepsilon - \sum_{j=1}^2 \chi_{3j} u_3^\varepsilon  \nabla v_j^\varepsilon  \bigg) \cdot \nabla \phi &= \iint_{\Omega_T} f_3^\varepsilon \phi ,\\ 
\lambda_i \intQT \nabla v_i^\varepsilon \cdot \nabla \phi + \mu_i \intQT v_i^\varepsilon \phi &= \intQT u_i^\varepsilon \phi,  \\
- \varepsilon \intQT v_3^\varepsilon \partial_t \phi - \varepsilon \int_\Omega v_{30}\phi(0) + \lambda_3 \intQT \nabla v_3^\varepsilon \cdot \nabla \phi + \mu_3 \intQT v_3^\varepsilon \phi &= \intQT u_3^\varepsilon \phi,  
\end{align*}
where we write $f_i^\varepsilon$ instead of $f_i(u_1^\varepsilon,u_2^\varepsilon,u_3^\varepsilon)$ for short. Thus, by taking into account the convergence \eqref{Theo.FastSignalDiffLimit.Proof1}-\eqref{Theo.FastSignalDiffLimit.Proof1b}, we can pass the above weak formulation to see that $(u_i,v_i)_{i=1,2,3,}$ is the global weak solution to  \eqref{System.UV.Limit}-\eqref{Condition.Boun-Init.Limit}. Thanks to Lemma~\ref{Lem.MainResultsOfABOV23}, 
this solution also coincides with the classical one in the sense of Definition~\ref{Definition.WeakSolution}. Moreover, using Lemma~\ref{Lem:Weak-to-strong}, we claim that 
\begin{align}
(v_1^\varepsilon,v_2^\varepsilon,v_3^\varepsilon) \to (v_1,v_2,v_3) \quad \text{strongly in } L^2(0,T;H^1(\Omega))^3.
\label{Theo.FastSignalDiffLimit.Proof2}
\end{align}
By the estimate \eqref{Lem.UVeps.Bootstrap.State2}, there exists a subsequence (not relabelled) $\{v_i^\varepsilon\}$ such that 
\begin{align}
(v_1^\varepsilon,v_2^\varepsilon,v_3^\varepsilon) \to (v_1,v_2,v_3) \quad \text{strongly in } L^q(0,T;W^{1,q}(\Omega))^3.
\label{Theo.FastSignalDiffLimit.Proof2}
\end{align}     
Note the uniqueness of the limiting system \eqref{System.UV.Limit}-\eqref{Condition.Boun-Init.Limit} also guarantees that the convergences \eqref{Theo.FastSignalDiffLimit.Proof1}-\eqref{Theo.FastSignalDiffLimit.Proof2} hold for the whole sequence.   
  
  \medskip
  
Since $(u_1,u_2,u_3) \in L^q(Q_T)^3$ for any $1\le q <\infty$, we have $(u_1,u_2,u_3) \in L^\infty(Q_T)^3$ by using the smoothing effect of the Neumann heat semigroup, similarly to the proof of Theorem~\ref{Theo.GlobalClassicalSol}. Therefore, the claim $ (v_1,v_2,v_3) \in L^\infty(0,T;W^{2,\infty}(\Omega))^3$ can be proved similarly to Lemma~\ref{Lem.UVeps.Bootstrap}. Moreover, due to the Ehrling lemma, for $i=1,2,3$, 
\begin{align*}
\|\nabla v_i(t)\|_{L^\infty(\Omega)} \le C\|\Delta v_i(t)\|_{L^r(\Omega)} + C\|v_i(t)\|_{L^\infty(\Omega)}, 
\end{align*} 
for $2<r<\infty$. By the equations for $v_i$, we have  
$\|v_i(t)\|_{L^\infty(\Omega)} \le C \|u_i(t)\|_{L^\infty(\Omega)}$. Therefore, applying the elliptic maximal regularity in Lemma~\ref{Lem.MaximalRegularity} gives
\begin{align*}
\|\nabla v_i(t)\|_{L^\infty(\Omega)} \le C\|u_i(t)\|_{L^r(\Omega)} + C\|u_i(t)\|_{L^\infty(\Omega)}, 
\end{align*}
which consequently implies $\nabla v_i \in L^\infty(Q_T)^2$ since $u_i\in L^\infty(Q_T)$. Analogously, we can show that $\nabla u_i \in  L^\infty(Q_T)^2$, and so, 
$u_i \in W^{2,1}_q(Q_T)$   due to Lemma~\ref{Lem.HeatRegularisation}.
\end{proof}

\section{$L^\infty$-in-time convergence rates}\label{sec:5}

We study $L^\infty$-in-time convergence rates of the fast signal diffusion limit in Section~\ref{sec:4}. Since $v_3(0)\not=v_{30}$ in general, see Remark~\ref{Remark.Theo.FastSignalDiffLimit}, we need to analyse carefully the effect of the initial layer. Moreover, a suitable estimate for $\partial_t v_3$ is necessary. Although this derivative does not appear in the equation for $v_3$, it can be estimated using \eqref{Expression.v}.

\subsection{Energy estimate for the rate system}

To estimate the rate $(\widehat u^{\,\varepsilon}_1,\widehat u^{\,\varepsilon}_2,\widehat u^{\,\varepsilon}_3)$ in \eqref{System.Rate.UV}, we consider the energy function  
\begin{align*}
\mathcal E_{n}[\widehat u^{\,\varepsilon}](t):=\sum_{i=1}^3 \int_\Omega (\widehat{u}_i^{\,\varepsilon}(t))^{2n}  ,
\end{align*} 
$t\in(0,T)$, and $n\in \mathbb{N}$, $n\ge 1$. For $n=1$, we denote $\mathcal E[\widehat u^{\,\varepsilon}]:= \mathcal E_{1}[\widehat u^{\,\varepsilon}]$. 
After estimating $(\widehat u^{\,\varepsilon}_1,\widehat u^{\,\varepsilon}_2,\widehat u^{\,\varepsilon}_3)$, we can study the rate $(\widehat v^{\,\varepsilon}_1,\widehat v^{\,\varepsilon}_2,\widehat v^{\,\varepsilon}_3)$ via maximal regularity for elliptic equations in Lemma~\ref{Lem.MaximalRegularity} and for parabolic equations with slow evolution in Lemma~\ref{Lem.MRSlowEvolution}. Let us begin with an estimate for $\mathcal E_{n}[\widehat u^{\,\varepsilon}]$. 

\begin{lemma} 
\label{Lem.Rate.EnergyE}
For $t\in (0,T)$,  
\begin{align} 
\frac{d}{dt} \mathcal E_{n}[\widehat u^{\,\varepsilon}] \le - \frac{2n-1}{n} \sum_{i=1}^3 d_i \int_\Omega |\nabla (\widehat{u}_i^{\,\varepsilon})^{n}|^2  
+ C_{n,T}\, \mathcal E_{n}[\widehat u^{\,\varepsilon}]  + C_{n,T} \mathcal F[\widehat v^{\,\varepsilon}], 
\label{Lem.Rate.EnergyEn.State1}
\end{align}
where we denote 
\begin{align*}
\mathcal F[\widehat v^{\,\varepsilon}]:=\sum_{i=1}^3  \int_\Omega  |\nabla \widehat{v}_i^{\,\varepsilon}|^2. 
\end{align*}
\end{lemma}

\begin{proof} The following computations are straightforward, 
\begin{align*}
  \frac{d}{dt} \mathcal E_{n}[\widehat u^{\,\varepsilon}](t)  
&  = - \frac{2(2n-1)}{n} \sum_{i=1}^3 d_i \int_\Omega |\nabla (\widehat{u}_i^{\,\varepsilon})^{n}|^2 +  2n \sum_{i=1}^3 \int_\Omega (\widehat{u}_i^{\,\varepsilon})^{2n-1} \widehat{f}^{\,\varepsilon}_i  \\
& + 2n \sum_{i=1}^2 \int_\Omega (\widehat{u}_i^{\,\varepsilon})^{2n-1}  \Big( \chi_i \di (\widehat{u}^{\,\varepsilon}_i \nabla v_3^\varepsilon) + \chi_i \di (u_i \nabla \widehat{v}^{\,\varepsilon}_3)  \Big) \\
& - 2n \sum_{i=1}^2 \int_\Omega (\widehat{u}_3^{\,\varepsilon})^{2n-1}  \Big( \chi_{3i} \di (\widehat{u}^{\,\varepsilon}_3 \nabla v_i^\varepsilon)   +  \chi_{3i} \di (u_3 ( \nabla \widehat{v}^{\,\varepsilon}_i) \Big) \\
& =:  - \frac{2(2n-1)}{n} \sum_{i=1}^3 d_i \int_\Omega |\nabla (\widehat{u}_i^{\,\varepsilon})^{n}|^2 + J_1 + J_2 +J_3.
\end{align*}
Due to the boundedness \eqref{Theo.FastSignalDiffLimit.State1}-\eqref{Theo.FastSignalDiffLimit.State3}, it is direct to see that 
\begin{align}
|\widehat{f}^{\,\varepsilon}_i|=|f_i(u_1^\varepsilon,u_2^\varepsilon,u_3^\varepsilon) - f_i(u_1,u_2,u_3)|\le C_T \sum_{i=1}^3 |\widehat{u}_i^{\,\varepsilon}|,
\label{Lem.Rate.EnergyE.Proof1}
\end{align}
which consequently shows   
\begin{align*}
J_1 =  2n \sum_{i=1}^3 \int_\Omega (\widehat{u}_i^{\,\varepsilon})^{2n-1} \widehat{f}^{\,\varepsilon}_i \le C_T \sum_{i=1}^3 \int_\Omega (\widehat{u}_i^{\,\varepsilon})^{2n} . 
\end{align*}
For the second term, we have  
\begin{align*}
J_2  & = - 2n(2n-1) \sum_{i=1}^2 \chi_i \int_\Omega \Big( (\widehat{u}_i^{\,\varepsilon})^{2n-1} \nabla \widehat{u}_i^{\,\varepsilon}\cdot \nabla v_3^\varepsilon  +  u_i (\widehat{u}_i^{\,\varepsilon})^{2n-2} \nabla \widehat{u}_i^{\,\varepsilon} \cdot \nabla \widehat{v}^{\,\varepsilon}_3 \Big) \\
& \le \frac{2n-1}{4n} \sum_{i=1}^2 d_i \int_\Omega |\nabla (\widehat{u}_i^{\,\varepsilon})^n|^2 + 4n(2n-1) \sum_{i=1}^2 \frac{\chi_i^2}{d_i} \int_\Omega (\widehat{u}_i^{\,\varepsilon})^{2n} |\nabla v_3^\varepsilon|^2 \\
& + \frac{2n-1}{4n} \sum_{i=1}^2 d_i \int_\Omega |\nabla (\widehat{u}_i^{\,\varepsilon})^n|^2 +  4n(2n-1) \sum_{i=1}^2  \frac{\chi_i^2}{d_i} \int_\Omega u_i^2 (\widehat{u}_i^{\,\varepsilon})^{2n-2} |\nabla \widehat{v}_3^{\,\varepsilon}|^2.
\end{align*}
Thanks to \eqref{Theo.FastSignalDiffLimit.State1}-\eqref{Theo.FastSignalDiffLimit.State3}, the triangle inequality yields 
\begin{align*}
\int_\Omega (\widehat{u}_i^{\,\varepsilon})^{2n} |\nabla v_3^\varepsilon|^2 & \le 2 \int_\Omega (\widehat{u}_i^{\,\varepsilon})^{2n} |\nabla \widehat{v}_3^{\,\varepsilon}|^2 + 2 \int_\Omega (\widehat{u}_i^{\,\varepsilon})^{2n} |\nabla v_3|^2 \\
& \le 2 \left( \sup_{\varepsilon>0} \|\widehat{u}_i^{\,\varepsilon}\|_{L^\infty(Q_T)}^{2n} \right) \int_\Omega |\nabla \widehat{v}_3^{\,\varepsilon}|^2 + 2 \|\nabla v_3\|_{L^\infty(Q_T)}^2 \int_\Omega (\widehat{u}_i^{\,\varepsilon})^{2n}  .
\end{align*}
Therefore,
\begin{align*}
J_2  & \le \frac{2n-1}{2n} \sum_{i=1}^2 d_i \int_\Omega |\nabla (\widehat{u}_i^{\,\varepsilon})^n|^2 + 4n(2n-1) \sum_{i=1}^2 \frac{\chi_i^2}{d_i} \|u_i^2(\widehat{u}_i^{\,\varepsilon})^{2n-2}\|_{L^\infty(Q_T)} \int_\Omega   |\nabla \widehat{v}_3^{\,\varepsilon}|^2\\
& + 8n(2n-1) \sum_{i=1}^2 \frac{\chi_i^2}{d_i} \left( \left( \sup_{\varepsilon>0} \|\widehat{u}_i^{\,\varepsilon}\|_{L^\infty(Q_T)}^{2n} \right) \int_\Omega |\nabla \widehat{v}_3^{\,\varepsilon}|^2  + \|\nabla v_3\|_{L^\infty(Q_T)}^2 \int_\Omega (\widehat{u}_i^{\,\varepsilon})^{2n} \right) .  
\end{align*}
Similarly to estimating the term $J_2$, one can show that 
\begin{align*}
J_3 & \le \frac{2n-1}{2n} d_3 \int_\Omega |\nabla (\widehat{u}_3^{\,\varepsilon})^{n}|^2 + \frac{4n(2n-1)}{d_3} \sum_{i=1}^2  \chi_{3i}^2 \|u_3^2(\widehat{u}_3^{\,\varepsilon})^{2n-2}\|_{L^\infty(Q_T)} \int_\Omega   |\nabla \widehat{v}_i^{\,\varepsilon}|^2 \\
& + \frac{4n(2n-1)}{d_3} \sum_{i=1}^2 \chi_{3i}^2 \left( \left( \sup_{\varepsilon>0} \|\widehat{u}_3^{\,\varepsilon}\|_{L^\infty(Q_T)}^{2n} \right) \int_\Omega |\nabla \widehat{v}_i^{\,\varepsilon}|^2  + \|\nabla v_i\|_{L^\infty(Q_T)}^2 \int_\Omega (\widehat{u}_3^{\,\varepsilon})^{2n} \right) .
\end{align*}
By combining the above estimates, using  \eqref{Theo.FastSignalDiffLimit.State1}-\eqref{Theo.FastSignalDiffLimit.State3}, 
we obtain the estimate \eqref{Lem.Rate.EnergyEn.State1}. 
\end{proof}

\subsection{$L^\infty$-in-time convergence rates} 
\label{Section.ConvergenceRate} 
Lemma~\ref{Lem.Rate.EnergyE} suggests a needed estimate for $\mathcal F[\widehat v^{\,\varepsilon}]$, which will be done in Lemmas~\ref{Lem.Rate.EstimateF}, where  the term including $\partial_t v_3$ will be  estimated in Lemma~\ref{Lem.Rate.TimeDeriOfv3}.

\begin{lemma} 
\label{Lem.Rate.EstimateF} For $t\in (0,T)$,  
\begin{equation}
\begin{aligned}
\mathcal F[\widehat v^{\,\varepsilon}] \le - \frac{\varepsilon}{2\lambda_3} \frac{d}{dt} \int_\Omega (\widehat{v}_3^{\,\varepsilon})^{2} -  \frac{\mu_3}{2\lambda_3} \int_\Omega (\widehat{v}_3^{\,\varepsilon})^{2} + C  \mathcal E [\widehat u^{\,\varepsilon}] + C\varepsilon^2 \int_\Omega |\partial_t v_3|^2  , 
\end{aligned}
\label{Lem.Rate.EstimateF.State1}
\end{equation}
and so,
\begin{align}
\mathcal E [\widehat u^{\,\varepsilon}]+  \sum_{i=1}^3 \intQT |\nabla \widehat{u}_i^{\,\varepsilon}|^2 + \intQT (\widehat{v}_3^{\,\varepsilon})^{2} \le  C_T  \left( \varepsilon^2 \intQT |\partial_t v_3|^2 + \varepsilon  \int_\Omega (\widehat{v}_3^{\,\varepsilon}(0))^{2} \right). 
\label{Lem.Rate.EstimateF.State2} 
\end{align}
\end{lemma}

\begin{proof}  
According to the equations
for $\widehat{v}_i^{\,\varepsilon}$, $i=1,2$, we see that  
\begin{align*}
\sum_{i=1}^2  \int_\Omega     |\nabla \widehat{v}_i^{\,\varepsilon}|^2   & =  \sum_{i=1}^2  \int_\Omega \left( \frac{1}{\lambda_i} \widehat{u}_i^{\,\varepsilon} \widehat{v}_i^{\,\varepsilon}  - \frac{\mu_i}{\lambda_i}  (\widehat{v}_i^{\,\varepsilon})^{2} \right)  \le \sum_{i=1}^2 \frac{1}{4\lambda_i\mu_i}  \int_\Omega (\widehat{u}_i^{\,\varepsilon})^{2} . 
\end{align*}
On the other hand, by the equation for $\widehat{v}_3^{\,\varepsilon}$,
\begin{align*}
 \int_\Omega |\nabla \widehat{v}_3^{\,\varepsilon}|^2 &  = - \frac{\varepsilon}{2\lambda_3} \frac{d}{dt} \int_\Omega (\widehat{v}_3^{\,\varepsilon})^{2} -  \frac{\mu_3}{\lambda_3} \int_\Omega (\widehat{v}_3^{\,\varepsilon})^{2} + \frac{1}{\lambda_3} \int_\Omega \widehat{u}_3^{\,\varepsilon}  \widehat{v}_3^{\,\varepsilon}   - \frac{\varepsilon}{\lambda_3} \int_\Omega \widehat{v}_3^{\,\varepsilon}  \partial_t v_3  \\
 & \le - \frac{\varepsilon}{2\lambda_3} \frac{d}{dt} \int_\Omega (\widehat{v}_3^{\,\varepsilon})^{2} -  \frac{\mu_3}{2\lambda_3} \int_\Omega (\widehat{v}_3^{\,\varepsilon})^{2} + \frac{1}{\lambda_3\mu_3} \int_\Omega (\widehat{u}_3^{\,\varepsilon})^2 + \frac{\varepsilon^2}{\lambda_3\mu_3} \int_\Omega |\partial_t v_3|^2 .
\end{align*}
Then, the estimate \eqref{Lem.Rate.EstimateF.State1}
 is obtained by taking the above estimates together. Let us prove \eqref{Lem.Rate.EstimateF.State2}. By taking $n=1$ in Lemma ~\ref{Lem.Rate.EnergyE} and \eqref{Lem.Rate.EstimateF.State1}, we have 
\begin{align}
\begin{gathered}
\frac{d}{dt} \mathcal E[\widehat u^{\,\varepsilon}] \le -  \sum_{i=1}^3 d_i \int_\Omega |\nabla \widehat{u}_i^{\,\varepsilon}|^2  
+ C_{T} \mathcal E[\widehat u^{\,\varepsilon}]    + C_T \mathcal F[\widehat v^{\,\varepsilon}] , 
\\
\mathcal F[\widehat v^{\,\varepsilon}] \le - \frac{\varepsilon}{2\lambda_3} \frac{d}{dt} \int_\Omega (\widehat{v}_3^{\,\varepsilon})^{2} -  \frac{\mu_3}{2\lambda_3} \int_\Omega (\widehat{v}_3^{\,\varepsilon})^{2} + C  \mathcal E [\widehat u^{\,\varepsilon}] + C\varepsilon^2 \int_\Omega |\partial_t v_3|^2 . 
\end{gathered}
\label{Lem.Rate.EnergyEF1.Proof1}
\end{align}
Therefore, we obtain
\begin{align*}
\begin{gathered}
\frac{d}{dt}\left(\mathcal E [\widehat u^{\,\varepsilon}]+C_T\varepsilon \int_\Omega (\widehat{v}_3^{\,\varepsilon})^{2} \right) + \sum_{i=1}^3 \int_\Omega |\nabla \widehat{u}_i^{\,\varepsilon}|^2 + C_T \int_\Omega (\widehat{v}_3^{\,\varepsilon})^{2} \\
\le  C_T \varepsilon^2 \int_\Omega |\partial_t v_3|^2 + C_T \left(\mathcal E [\widehat u^{\,\varepsilon}]+C_T\varepsilon \int_\Omega (\widehat{v}_3^{\,\varepsilon})^{2} \right).
\end{gathered}
\end{align*}
Due to the Gr\"onwall inequality, for $t\in (0,T)$,
\begin{align*}
\mathcal E [\widehat u^{\,\varepsilon}]+  \sum_{i=1}^3 \intQT |\nabla \widehat{u}_i^{\,\varepsilon}|^2 + \intQT (\widehat{v}_3^{\,\varepsilon})^{2} \le  C_T \varepsilon^2 \intQT |\partial_t v_3|^2 + C_T \varepsilon \int_\Omega (\widehat{v}_3^{\,\varepsilon}(0))^{2},
\end{align*} 
where we note that $\mathcal E [\widehat u^{\,\varepsilon}](0)=0$.
\end{proof}

By Lemma~\ref{Lem.Rate.EstimateF}, we need to estimate the time derivative $\partial_t v_3$ in $L^2(Q_T)$, which, however, can be shown to belong to $L^q(Q_T)$ for any $1<q<\infty$ in the lemma below.

\begin{lemma} \label{Lem.Rate.TimeDeriOfv3}
For any $1\le q<\infty$,   
\begin{align*}
\|\partial_t v_3\|_{L^q(Q_T)} + \|\partial_t v_3\|_{L^\infty(0,T;L^2(\Omega))} \le C_T.
\end{align*} 
\end{lemma} 

\begin{proof} Thanks to \eqref{Expression.v}, the component $v_3$ can be expressed in terms of $u_3$ as   
\begin{align}
v_3(t) & = \left( \int_0^\infty e^{s(\lambda_3 \Delta - \mu_3 I)} \, ds \right) u_3(t).
\label{Lem.Rate.TimeDeriOfv3:Proof1}
\end{align}
Due to the regularity \eqref{Theo.FastSignalDiffLimit.State3} and the equation for $u_3$, we have $\partial_t u_3 \in L^q(Q_T)$ for any $1\le q<\infty$. Therefore, by using the estimate \eqref{HeatSemigroup.Contraction},
\begin{align*}
\|\partial_t v_3(t)\|_{L^q(\Omega)} & \le \left( \int_0^\infty e^{-\mu_3 s} \, ds \right) \|\partial_t u_3(t)\|_{L^q(\Omega)},
\end{align*}
which claims the conclusion by taking the $L^q(0,T)$-norms from the latter estimate. Moreover, under the regularity \eqref{Theo.FastSignalDiffLimit.State3}, one can show   $\partial_t u_3 \in L^\infty(0,T;L^2(\Omega))$ due to the smoothing effect of the heat semigroup, which yields the same boundedness for $\partial_t v_3$.
\end{proof}

We are now ready to prove Theorem~\ref{Theo.ConvergenceRate}.  
 
\begin{proof}[\underline{Proof of Theorem~\ref{Theo.ConvergenceRate}: Part a}] 
Thanks to Lemma~\ref{Lem.Rate.TimeDeriOfv3}, we have $\partial_t v_3 \in L^2(Q_T)$. Therefore, due to Lemma~\ref{Lem.Rate.EstimateF},   
\begin{align}
\mathcal E [\widehat u^{\,\varepsilon}]+  \sum_{i=1}^3 \intQT |\nabla \widehat{u}_i^{\,\varepsilon}|^2 + \intQT (\widehat{v}_3^{\,\varepsilon})^{2} \le  C_T  \left( \varepsilon^2 + \varepsilon  \int_\Omega (\widehat{v}_3^{\,\varepsilon}(0))^{2} \right) .
\label{Theo.ConvergenceRate.Proof1}
\end{align}
By Remark~\ref{Remark.Theo.FastSignalDiffLimit} and the estimate \eqref{HeatSemigroup.Contraction}, 
\begin{align}
\begin{aligned}
\|\widehat{v}_3^{\,\varepsilon}(0)\|_{L^2(\Omega)} &= \left\| - \int_0^\infty e^{s(\lambda_3 \Delta - \mu_3 I)} ( \lambda_3 \Delta v_{30} - \mu_3 v_{30} + u_{30}) \, ds \right\|_{L^2(\Omega)} \\ 
&\le \left( \int_0^\infty e^{- \mu_3s} \, ds \right) \| \lambda_3 \Delta v_{30} - \mu_3 v_{30} + u_{30}\|_{L^2(\Omega)} \\
& = \frac{1}{\mu_3} \| \lambda_3 \Delta v_{30} - \mu_3 v_{30} + u_{30} \|_{L^2(\Omega)}  = \frac{1}{\mu_3} \varepsilon_{\mathsf{in}}.
\end{aligned}
\label{Theo.ConvergenceRate.Proof1b}
\end{align}
Therefore, it follows from \eqref{Theo.ConvergenceRate.Proof1} that 
\begin{align*}
\mathcal E [\widehat u^{\,\varepsilon}]+  \sum_{i=1}^3 \intQT |\nabla \widehat{u}_i^{\,\varepsilon}|^2 + \intQT (\widehat{v}_3^{\,\varepsilon})^{2} \le  C_T  \left(   \varepsilon_{\mathsf{in}}^2 \varepsilon + \varepsilon^2 \right),
\end{align*}
which accordingly yields 
\begin{align}
\sum_{n=1}^3 \Big( \|\widehat u^{\,\varepsilon}_i\|_{L^\infty(0,T;L^2(\Omega))} + \|\widehat u^{\,\varepsilon}_i\|_{L^2(0,T;H^1(\Omega))} \Big) \le C_T \left( \varepsilon_{\mathsf{in}} \sqrt{\varepsilon} + \varepsilon \right).
\label{Theo.ConvergenceRate.a.Proof2}
\end{align}
By applying Lemma~\ref{Lem.MaximalRegularity} to the equations for $\widehat v^{\,\varepsilon}_i$, $i=1,2$,  
\begin{align}
\|\widehat v^{\,\varepsilon}_i\|_{L^\infty(0,T;H^2(\Omega))} \le C \|\widehat u^{\,\varepsilon}_i\|_{L^\infty(0,T;L^2(\Omega))} \le C_T \left(\varepsilon_{\mathsf{in}} \sqrt{\varepsilon} + \varepsilon  \right) .
\label{Theo.ConvergenceRate.a.Proof3}
\end{align}
We now proceed to estimate $\widehat v^{\,\varepsilon}_3$ in $L^\infty(0,T;H^1(\Omega))$, where it is only necessary to estimate $\nabla \widehat v^{\,\varepsilon}_3$ in $L^\infty(0,T;L^2(\Omega))$. We test the equation for $\widehat v^{\,\varepsilon}_3$ by $-\Delta \widehat v^{\,\varepsilon}_3$, which gives 
\begin{align*}
 \varepsilon \frac{d}{dt} \int_\Omega |\nabla \widehat v^{\,\varepsilon}_3|^2 + \lambda_3 \int_\Omega |\Delta  \widehat v^{\,\varepsilon}_3|^2 + 2\mu_3 \int_\Omega |\nabla \widehat v^{\,\varepsilon}_3|^2    \le \frac{1}{\lambda_3} \int_\Omega (\widehat u^{\,\varepsilon}_3-\varepsilon  \partial_t v_3)^2 .
\end{align*}
By using \eqref{Theo.ConvergenceRate.a.Proof2} and Lemma~\ref{Lem.Rate.TimeDeriOfv3}, the latter right-hand side is of order $O(\varepsilon_{\mathsf{in}}^2 \varepsilon + \varepsilon^2)$. 
Hence,  
\begin{align}
\varepsilon \frac{d}{dt} \int_\Omega |\nabla \widehat v^{\,\varepsilon}_3|^2 + \lambda_3 \int_\Omega |\Delta  \widehat v^{\,\varepsilon}_3|^2 + 2\mu_3 \int_\Omega |\nabla \widehat v^{\,\varepsilon}_3|^2    \le C_T (\varepsilon_{\mathsf{in}}^2 \varepsilon + \varepsilon^2) ,
\label{Theo.ConvergenceRate.a.Proof4}
\end{align}
so that, by the comparison principle, 
\begin{equation*}
\begin{aligned} 
\int_\Omega |\nabla \widehat v^{\,\varepsilon}_3|^2 & \le \exp\left(-\frac{2\mu_3}{\varepsilon}t\right) \int_\Omega |\nabla \widehat v^{\,\varepsilon}_3(0)|^2 + C_T(\varepsilon_{\mathsf{in}}^2+\varepsilon) \int_0^t \exp\left(-\frac{2\mu_3}{\varepsilon}(t-s)\right) \\
& \le \left\| \nabla \int_0^\infty e^{s(\lambda_3 \Delta - \mu_3 I)} ( \lambda_3 \Delta v_{30} - \mu_3 v_{30} + u_{30}) \, ds \right\|_{L^2(\Omega)}^2 + C_T(\varepsilon_{\mathsf{in}}^2 \varepsilon + \varepsilon^2) \\
& \le C \left( \int_0^\infty e^{-\omega s} s^{-\frac{1}{2}} \right)^2 \left\| \lambda_3 \Delta v_{30} - \mu_3 v_{30} + u_{30} \right\|_{L^2(\Omega)}^2 + C_T(\varepsilon_{\mathsf{in}}^2 \varepsilon + \varepsilon^2) ,    
\end{aligned}
\end{equation*}
where the expression \eqref{Expression.InitialLayer} has been exploited to estimate $\nabla \widehat v^{\,\varepsilon}_3(0)$.   Consequently,  
\begin{align*}
\| \widehat v^{\,\varepsilon}_3\|_{L^\infty(0,T;H^1(\Omega))}  
&\le C_T \sqrt{ \varepsilon_{\mathsf{in}}^2 + \varepsilon_{\mathsf{in}}^2 \varepsilon + \varepsilon^2 } = O(\varepsilon_{\mathsf{in}} + \varepsilon ).    
\end{align*} 
The estimate for $\widehat v^{\,\varepsilon}_3$ in $L^2(0,T;H^2(\Omega))$ is shown by integrating \eqref{Theo.ConvergenceRate.a.Proof4}. 

\medskip

\noindent \underline{\it Part b.}  Let us take $n\in \mathbb{N}$ such that $2(n-1)<q<2n$. We first apply Lemma~\ref{Lem.Rate.EnergyE} for the index $k\in\{n-1;n\}$, which gives after integrating in time  
\begin{align*}
\mathcal E_{k}[\widehat u^{\,\varepsilon}](t) &\le   C_{k,T} \int_0^t \mathcal E_{k}[\widehat u^{\,\varepsilon}]  + C_{k,T} \int_0^t \mathcal F[\widehat v^{\,\varepsilon}] ,
\end{align*}
where we recall from Remark~\ref{Remark.Theo.FastSignalDiffLimit} that  $\mathcal E[\widehat u^{\,\varepsilon}](0) =0$. Integrating  the estimate \eqref{Lem.Rate.EstimateF.State1} over time gives, where we note that $\partial_t v_3 \in L^2(Q_T)$ due to Lemma~\ref{Lem.Rate.TimeDeriOfv3}, 
\begin{align*} 
\int_0^t \mathcal F[\widehat v^{\,\varepsilon}] & \le \frac{\varepsilon}{2\lambda_3} \int_\Omega (\widehat{v}_3^{\,\varepsilon}(0))^{2}   + C  \int_0^t \mathcal E [\widehat u^{\,\varepsilon}] + C\varepsilon^2  \\
&\le \frac{\varepsilon}{2\lambda_3} \frac{1}{\mu_3^2} \varepsilon_{\mathsf{in}}^2 + C_T (\varepsilon_{\mathsf{in}}^2 \varepsilon + \varepsilon^2)  .
\end{align*}
Here, the right-hand side is of order $O(\varepsilon_{\mathsf{in}}^2 \varepsilon + \varepsilon^2)$.  Therefore,  
\begin{align*}
\mathcal E_{k}[\widehat u^{\,\varepsilon}](t) &\le C_{k,T} (\varepsilon^2 + \varepsilon_{\mathsf{in}}^2 \varepsilon) +   C_{k,T} \int_0^t \mathcal E_{k}[\widehat u^{\,\varepsilon}] \le C_{k,T} (\varepsilon_{\mathsf{in}}^2 \varepsilon + \varepsilon^2) , 
\end{align*}   
for $t\in (0,T)$, due to the Gr\"onwall inequality. By the definition of $\mathcal E_{k}[\widehat u^{\,\varepsilon}]$ and an interpolation between $ L^{2(n-1)}(\Omega) $ and $L^{2n}(\Omega)$, we have 
 $$\sum_{i=1}^3 \|\widehat u^{\,\varepsilon}_i\|_{L^\infty(0,T; L^{q}(\Omega))} \le C_{q,T} \, \varepsilon^{\frac{1}{q}} (\varepsilon_{\mathsf{in}}^{\frac{2}{q}} + \varepsilon^{\frac{1}{q}}). $$
Applying Lemma~\ref{Lem.MaximalRegularity},    
\begin{align}
\|\widehat v^{\,\varepsilon}_i\|_{L^\infty(0,T;W^{2,q}(\Omega))} \le C  \|\widehat u^{\,\varepsilon}_i\|_{L^\infty(0,T;L^{q}(\Omega))} \le C_{q,T} \, \varepsilon^{\frac{1}{q}} (\varepsilon_{\mathsf{in}}^{\frac{2}{q}} + \varepsilon^{\frac{1}{q}}) 
\label{Theo.ConvergenceRate.b.Proof2} 
\end{align}
for $i=1,2$, i.e., we obtain \eqref{Theo.ConvergenceRate.b.State1}.  Now, we proceed to estimate  $\widehat v^{\,\varepsilon}_3$  in $L^q(0,T;W^{2,q}(\Omega))$. By applying Lemma~\ref{Lem.MRSlowEvolution}, 
\begin{align*}
\begin{aligned}
\|\Delta \widehat v^{\,\varepsilon}_3 \|_{L^q(Q_T)}  &\le \bigg( \frac{\varepsilon}{q\mu_3} \bigg)^{\frac{1}{q}} \|  \Delta \widehat v^{\,\varepsilon}_3(0)\|_{L^q(\Omega)} +  \frac{C_{1,\mu_3/\lambda_3,q}^{\mathsf{PM}}}{\lambda_3} \|\widehat{u}^{\,\varepsilon}_3 - \eps \partial_t v_3 \|_{L^q(Q_T)} \\
&\le C_{q,\lambda_3}  \varepsilon^{\frac{1}{q}} \left\| - \Delta \int_0^\infty e^{s(\lambda_3 \Delta - \mu_3 I)} ( \lambda_3 \Delta v_{30} - \mu_3 v_{30} + u_{30}) \, ds \right\|_{L^q(\Omega)}  \\
& + \frac{C_{1,\mu_3/\lambda_3,q}^{\mathsf{PM}}}{\lambda_3} \left(  C_{q,T} \, \varepsilon^{\frac{1}{q}} (\varepsilon_{\mathsf{in}}^{\frac{2}{q}} + \varepsilon^{\frac{1}{q}})  + \eps\| \partial_t v_3\|_{L^q(Q_T)}  \right)  \\
&\le C_{q,\lambda_3,\mu_3,\omega,T} \, \varepsilon^{\frac{1}{q}}  \left(  \left\|  \lambda_3 \Delta v_{30} - \mu_3 v_{30} + u_{30} \right\|_{W^{2,q}(\Omega)} +  \varepsilon_{\mathsf{in}}^{\frac{2}{q}} + \varepsilon^{\frac{1}{q}} \right), 
\end{aligned}
\end{align*} 
where the term $\Delta \widehat v^{\,\varepsilon}_3(0)$ was estimated using the identity \eqref{Expression.InitialLayer}, the boundedness of $\partial_t v_3$ in $L^q(Q_T)$ was established in Lemma~\ref{Lem.Rate.TimeDeriOfv3}. This shows \eqref{Theo.ConvergenceRate.b.State2}.

 \medskip
 
\noindent{\underline{Part c.}} This part will be proved by using the smoothing effect of the heat semigroup. Let $\overline{q}>4$. Then, there exists $\varsigma>0$ small enough such that $1/\overline{q}<1/2-1/\overline{q}-\varsigma$ and $4+\varsigma<\overline{q}$. Then, we can find $\sigma$ such that  
\begin{align*}
\frac{1}{\overline{q}} < \sigma < \frac{1}{2} - \frac{1}{\overline{q}} - \varsigma.
\end{align*}
Using the equation for $\widehat u^{\,\varepsilon}_i$, $i=1,2$, the same techniques as \eqref{Lemma.LpEstimate.Proof2}  show that  
\begin{equation*}
  \begin{split}
      \|\widehat u^{\,\varepsilon}_i\|_{\LO{\infty}} 
    & = 
    \left\| \int^t_{0} e^{(t-s)(d_i\Delta-I)}  \left( \chi_i \di (\widehat{u}^{\,\varepsilon}_i \nabla v_3^\varepsilon) + \chi_i \di (u_i \nabla \widehat{v}^{\,\varepsilon}_3) + \widehat{f}^{\,\varepsilon}_i + \widehat u^{\,\varepsilon}_i \right) \, ds \right\|_{\LO{\infty}} \\ 
  & \le  C_T  \int^t_{0} (t-s)^{-\sigma-\frac{1}{2}-\varsigma} \left( \|\widehat{u}^{\,\varepsilon}_i \nabla v_3^\varepsilon +  u_i \nabla \widehat{v}^{\,\varepsilon}_3\|_{L^{4+\varsigma}(\Omega)} +  \|\widehat{f}^{\,\varepsilon}_i + \widehat u^{\,\varepsilon}_i\|_{L^{4+\varsigma}(\Omega)}  \right) \, ds \\
  & \le  C_{\overline{q},\varsigma,\sigma,T}   \left( \|\widehat{u}^{\,\varepsilon}_i \nabla v_3^\varepsilon +  u_i \nabla \widehat{v}^{\,\varepsilon}_3\|_{L^{4+\varsigma}(Q_T)} +  \|\widehat{f}^{\,\varepsilon}_i + \widehat u^{\,\varepsilon}_i\|_{L^{4+\varsigma}(Q_T)}  \right) t^{\frac{1}{2}-\frac{1}{\overline{q}}-\varsigma -\sigma} \\
  & \le  C_{\overline{q},\varsigma,\sigma,T} \Bigg( \|\widehat{u}^{\,\varepsilon}_i\|_{L^{\overline{q}}(Q_T)} + \| \nabla \widehat{v}^{\,\varepsilon}_3\|_{L^{\overline{q}}(Q_T)} + \sum_{j=1}^3 \|\widehat u^{\,\varepsilon}_j\|_{L^{\overline{q}}(Q_T)}  \Bigg) T^{\frac{1}{2}-\frac{1}{\overline{q}}-\varsigma -\sigma}  . 
  \end{split}
\end{equation*}
Thanks to Parts a-b of this theorem, we get  
\begin{align*}
\|\widehat u^{\,\varepsilon}_i\|_{\LO{\infty}} \le C_T    \varepsilon^{\frac{1}{\overline{q}}}  \left(  \widehat{\varepsilon}_{\mathsf{in}} +  \varepsilon_{\mathsf{in}}^{\frac{2}{\overline{q}}} + \varepsilon^{\frac{1}{\overline{q}}} \right),  
\end{align*}
i.e., \eqref{Theo.ConvergenceRate.b.State3} is proved.
\end{proof}

\subsection{Distance from trajectories to the critical manifold}

\begin{proof}[Proof of Theorem~\ref{Coro.ConvToCritMani}]
     Since $\lambda_3\Delta v_3 + \mu_3 v_3 - u_3 = 0$, Part a is a direct corollary of Theorem~\ref{Theo.ConvergenceRate}. Let us prove Part b. Multiplying the equation for $\widehat{v}_3^\varepsilon$ by $\Delta v_3^\varepsilon$ yields 
\begin{align}
\begin{aligned}
    & \frac{\varepsilon}{2} \frac{d}{dt} \int_\Omega |\Delta \widehat{v}_3^\varepsilon|^2 + \frac{\lambda_3}{2} \int_\Omega |\nabla \Delta \widehat{v}_3^\varepsilon|^2 + \mu_3 \int_\Omega |\Delta \widehat{v}_3^\varepsilon|^2 \\
    & \le \frac{1}{\lambda_3} \int_\Omega |\nabla \widehat{u}_3^\varepsilon|^2 + \frac{\varepsilon^2}{\lambda_3} \int_\Omega |\partial_t \nabla v_3|^2 + \frac{\lambda_3}{2} \int_{\partial \Omega} \Delta \widehat{v}_3^\varepsilon \nabla \Delta \widehat{v}_3^\varepsilon \cdot \nu, 
\end{aligned}
\label{Distance:FurtherComment:Proof1}
\end{align}
where, since for each $\varepsilon>0$ the classical solutions $(u_i^\varepsilon,v_i^\varepsilon)$ and $(u_i,v_i)$ of \eqref{System.UV}-\eqref{Condition.Boun-Init.Epsilon} and \eqref{System.UV.Limit}-\eqref{Condition.Boun-Init.Limit} are sufficiently smooth, we have 
\begin{align*}
    \lambda_3 \nabla \Delta \widehat{v}_3^\varepsilon \cdot \nu =  \nabla ( \varepsilon \partial_t \widehat{v}_3^\varepsilon + \mu_3 \widehat{v}_3^\varepsilon - \widehat{u}_3^\varepsilon + \varepsilon \partial_t v_3) \cdot \nu = 0
\end{align*}
on the boundary $\Gamma$. Moreover, basing on the regularity $\partial_t u_3 \in L^2(Q_T)$, the smoothing effect of the heat semigroup shows $\partial_t \nabla v_3 \in L^2(Q_T)^2$. Thanks to the estimate for $\nabla \widehat{u}_3^\varepsilon$ in \eqref{Theo.ConvergenceRate.a.State2}, we now follow from \eqref{Distance:FurtherComment:Proof1} that   
\begin{align}
    \frac{\varepsilon}{2} \frac{d}{dt} \int_\Omega |\Delta \widehat{v}_3^\varepsilon|^2 + \mu_3 \int_\Omega |\Delta \widehat{v}_3^\varepsilon|^2 \le C_T (\varepsilon_{\mathsf{in}}^2+\varepsilon^2), 
\end{align}
which, due to the comparison principle, shows that the $L^\infty(0,T;L^2(\Omega))$-norm of $\Delta \widehat{v}_3^\varepsilon$ is of the order $O(\varepsilon_{\mathsf{in}}+\varepsilon)$. This, combined with \eqref{Theo.ConvergenceRate.a.State1}-\eqref{Theo.ConvergenceRate.a.State2}, shows the estimate \eqref{Distance:FurtherComment}.
\end{proof}

\section{Numerical simulations} \label{sec:6}

The numerical simulations are based on Julia (\cite{rackauckas2017differentialequations}) using finite differences in space for discretizing the spatial operators and a split ODE solver for treating the $\varepsilon$-depending diffusion part of $v_3$  separately. As $\varepsilon$ increases the stiffness of the discretized diffusion matrix even more, an exponential integrator method is used for calculating the dominant part concerning the stiffness with higher accuracy. The elliptic equations are solved by iterative methods using Krylov-subspaces and generalized minimal residuals (Krylov-GMRES), \cite{Montoison2023}. System \eqref{System.UV}-\eqref{Condition.Boun-Init.Epsilon} is numerically solved for the fixed parameters in Table~\ref{tab:parameters}.
\begin{table}[h!]
    \centering
    \begin{tabular}{cc c|c cc c|c cc c|c cc}
        $d_i$ & 0.1 & &  & $\chi_j$ & 1.0 & & &$\chi_{31}$ & 1.0 & & & $\chi_{32}$ & 1.0 \\
         $\alpha_1$ & 0.8  &&  &$\alpha_2$  &1.0  & &       &  $\beta_1$& 0.6 & & & $\beta_1$ & 0.5\\
        $m_1$ & 0.3 & & & $m_2$ & 0.1 & & &     $\gamma_1$ & 0.5  & & & $\gamma_2$ & 0.3 \\
         $k$& 0.1 &  && $l$  & 0.1 & &&
       $\lambda_i$  & 1  & & & $\mu_i$  & 0.1 \\
    \end{tabular}
    \caption{Parameters used in numerical simulations. $i=1,2,3$ and $j=1,2$. }
    \label{tab:parameters}
\end{table}

Figure~\ref{fig:pde1} shows the dynamical behaviour for system \eqref{System.UV}-\eqref{Condition.Boun-Init.Epsilon} with $\varepsilon=10^{-5}$.
The simulations are carried out in one spatial dimension, and we may find different behaviours in two spatial dimensions. 
\begin{figure}[h]
\begin{center}
\begin{subfigure}{.45\linewidth}
\centering
\includegraphics[width=0.9\textwidth]{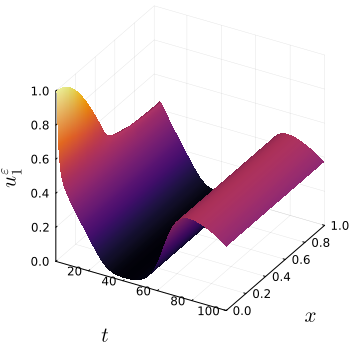}
\caption{prey $u_1^\varepsilon$}
\label{fig:sub_pde1}
\end{subfigure}%
\begin{subfigure}{.45\linewidth}
\centering
\includegraphics[width=0.9\textwidth]{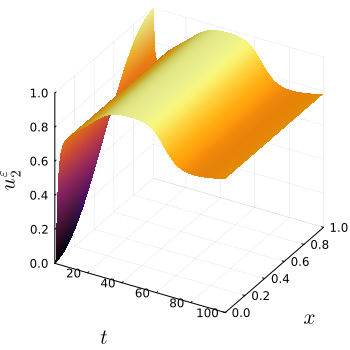}
\caption{prey $u_2^\varepsilon$}
\label{fig:sub_pde2}
\end{subfigure}%

\begin{subfigure}{.45\linewidth}
\centering
\includegraphics[width=0.9\textwidth]{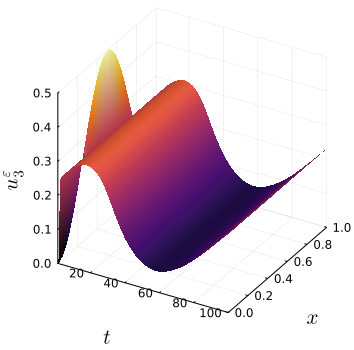}
\caption{predator $u_3^\varepsilon$}
\label{fig:sub_pde3}
\end{subfigure}%
\begin{subfigure}{.45\linewidth}
\centering
\includegraphics[width=0.9\textwidth]{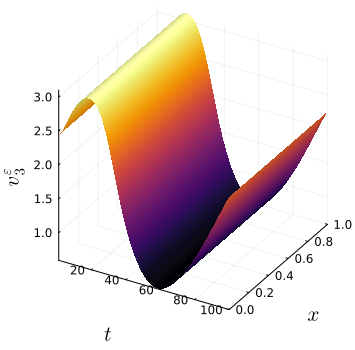}
\caption{chemical $v_3^\varepsilon$}
\label{fig:sub_pde4}
\end{subfigure}%
\caption{Dynamics for system \eqref{System.UV}-\eqref{Condition.Boun-Init.Epsilon} with $\varepsilon=10^{-5}$.}
\label{fig:pde1}
\end{center}
\end{figure}

\subsection{Comparison of the $\varepsilon-$depending and limiting systems}

In Theorem~\ref{Theo.ConvergenceRate}, Part a, we obtained respectively the convergence rate $\widehat{u}_i^\varepsilon=u_i^\varepsilon-u_i$, for $i=1,2,3,$ in $L^\infty(0,T; L^2(\Omega))$ and $\widehat{v}_3^\varepsilon=v_3^\varepsilon-v_3$ in $L^\infty(0,T; H^1(\Omega))$. Here, we investigate these rates numerically by considering the initial values on the critical manifold $\mathcal C$ (see Remark~\ref{Remark.Theo.FastSignalDiffLimit}), or more precisely, $\varepsilon_{\mathsf{in}}=0$. 

\medskip

Figure~\ref{fig:pde2} compares the solutions of the $\varepsilon-$depending system and the limiting system for $\varepsilon=10^{-5}$. 
The difference between the solutions of the species is of order $10^{-6}$, while the chemical $v_3$ shows a larger difference of $10^{-5}$. 
Therefore, the $L^\infty(Q_T)$ difference is of the order $\varepsilon$.

\medskip
 
\begin{figure}[htbp]
\begin{center}
\begin{subfigure}{.45\linewidth}
\centering
\includegraphics[width=\textwidth]{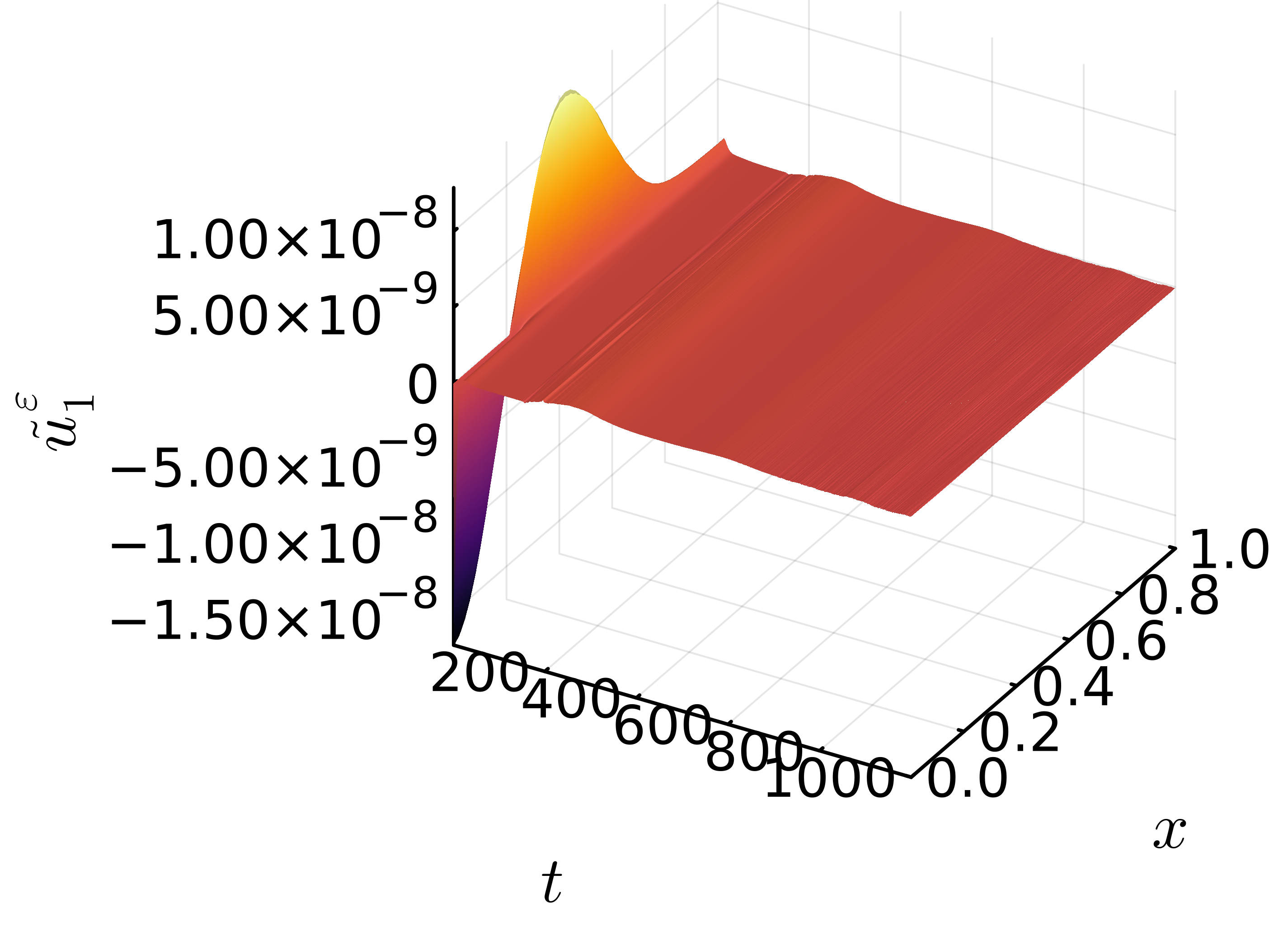}
\caption{ prey rate $\widehat{u}_1^\varepsilon=u_1^\varepsilon-u_1$}
\label{fig:sub_pde1d}
\end{subfigure}%
\vspace{0.2cm}
\begin{subfigure}{.45\linewidth}
\centering
\includegraphics[width=\textwidth]{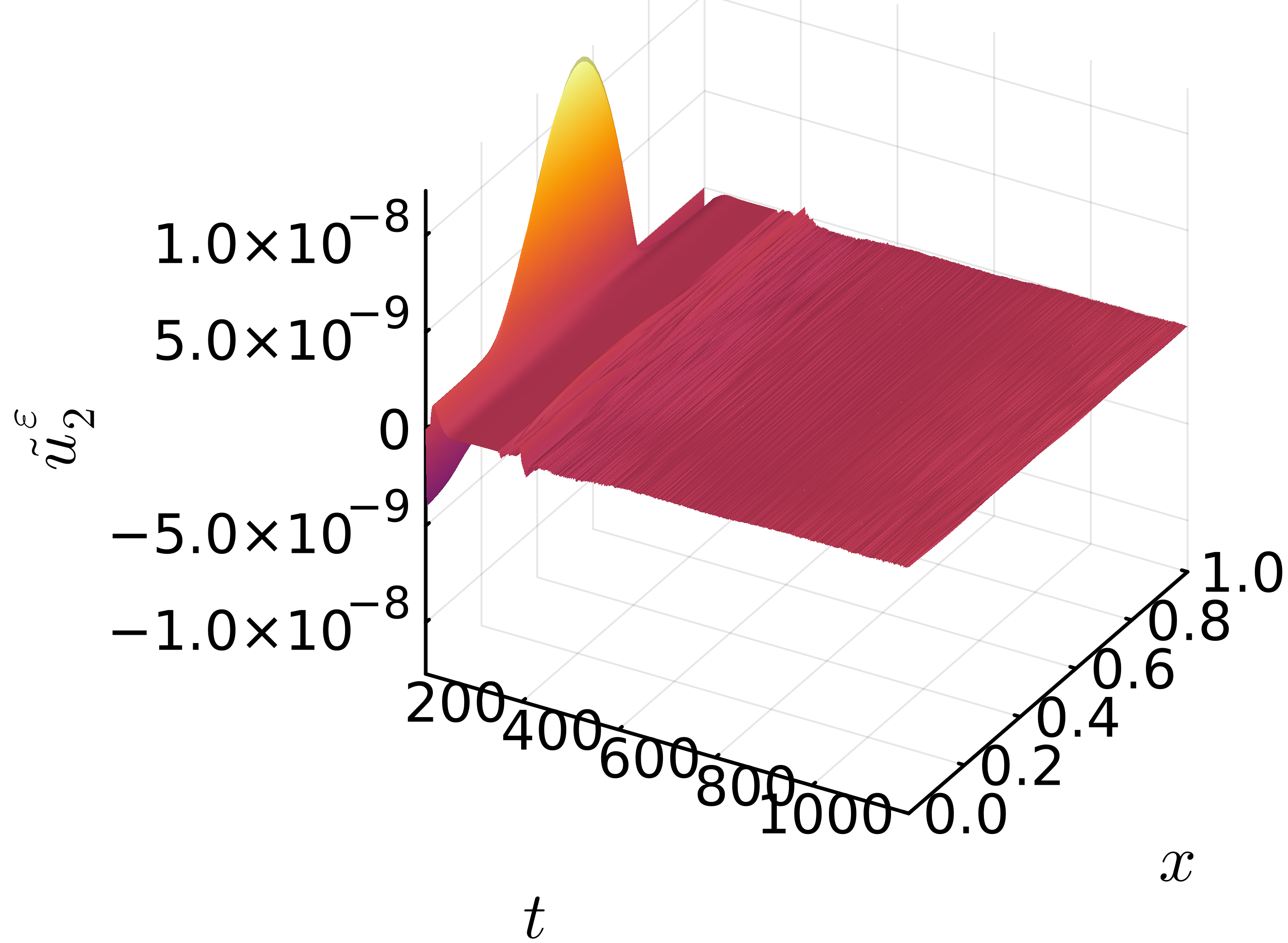}
\caption{ prey rate $\widehat{u}_2^\varepsilon=u_2^\varepsilon-u_2$}
\label{fig:sub_pde2d}
\end{subfigure}%

\begin{subfigure}{.45\linewidth}
\centering
\includegraphics[width=\textwidth]{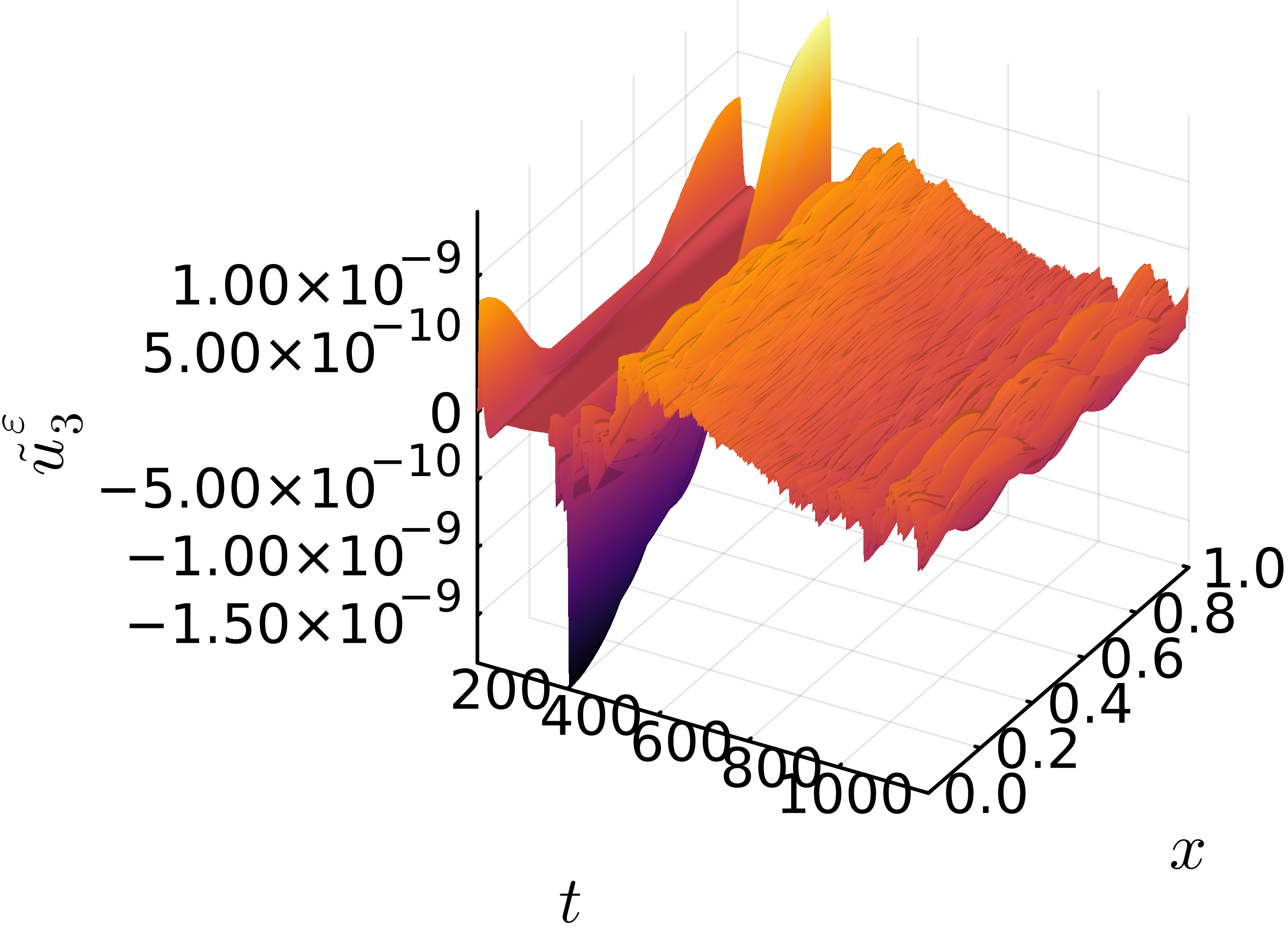}
\caption{ predator rate $\widehat{u}_3^\varepsilon=u_3^\varepsilon-u_3$}
\label{fig:sub_pde3d}
\end{subfigure}%
\begin{subfigure}{.45\linewidth}
\centering
\includegraphics[width=\textwidth]{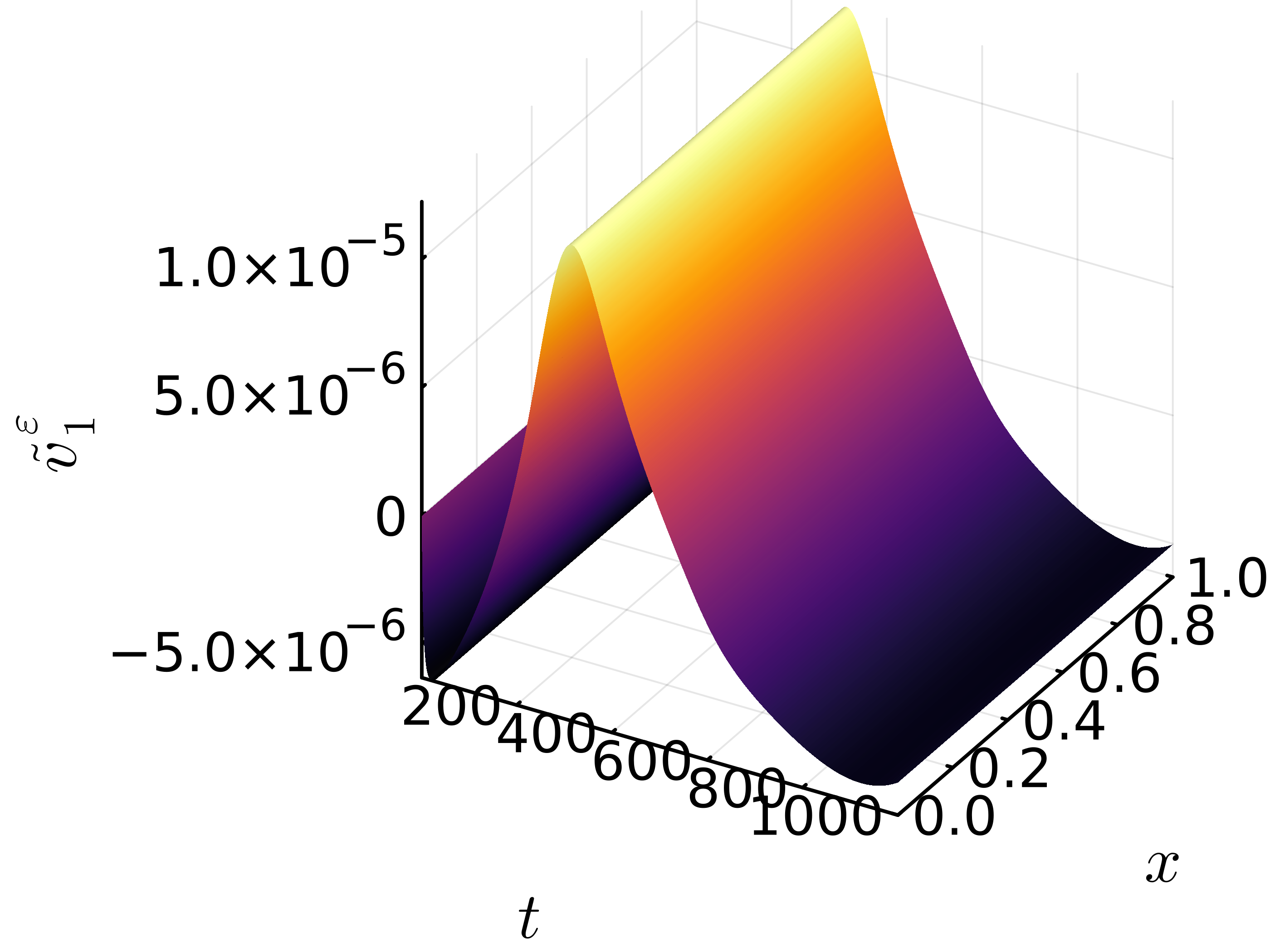}
\caption{ chemical rate $\widehat{v}_3^\varepsilon=v_3^\varepsilon-v_3$}
\label{fig:sub_pde4d}
\end{subfigure}%
\caption{Difference of the solutions of systems \eqref{System.UV.Limit}-\eqref{Condition.Boun-Init.Limit} and \eqref{System.UV}-\eqref{Condition.Boun-Init.Epsilon} with $\varepsilon=10^{-5}$.}
\label{fig:pde2}
\end{center}
\end{figure}


 Figure~\ref{fig:linf} shows the differences between the solutions of the $\varepsilon-$depending system and the limiting system for various $\varepsilon=10^{-k}$ for $k=1,\dots, 7$. Smaller values of $\varepsilon$ are not meaningful for simulations using an accuracy of $10^{-16}$. The $L^\infty$ error behaves like expected linearly with the order of magnitude of $\varepsilon$, compare Theorem~\ref{Theo.ConvergenceRate}.
The initial data starts on the critical manifold, compare Remark~\ref{Remark.Theo.ConvergenceRate}, therefore $\varepsilon_\mathsf{in} = 0$.


\begin{figure}
        \begin{subfigure}{.48\linewidth}
        \centering
              \includegraphics[height=4.8cm]{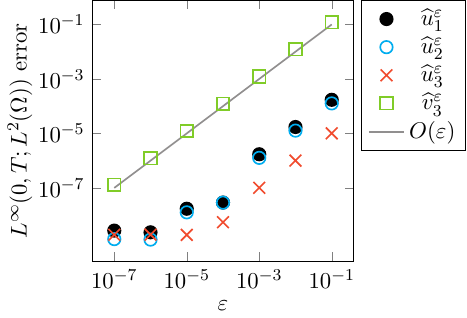}
    \caption{ $\varepsilon_\mathsf{in} =0$}
    \label{fig:linf}
        \end{subfigure}%
        \begin{subfigure}{.48\linewidth}
        \centering
        \includegraphics[height=4.8cm]{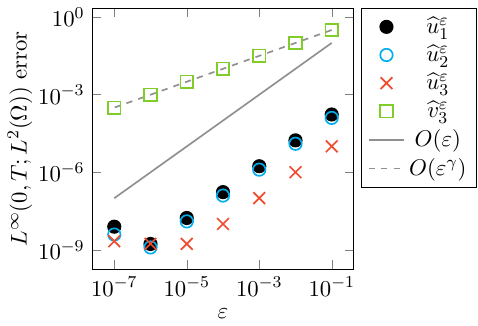}
        \caption{ $0\not =\varepsilon_\mathsf{in} = O(\varepsilon^\gamma)$ for $\gamma=1/2$}
        \label{fig:linf_inital_manifold}
        \end{subfigure}%
    \caption{$L^\infty(0,T;L^2(\Omega))$-differences of the solutions of systems \eqref{System.UV.Limit}-\eqref{Condition.Boun-Init.Limit} and \eqref{System.UV}-\eqref{Condition.Boun-Init.Epsilon}, with initial values (a) on or (b) outside the critical manifold.}
    \label{fig:linf_both}
\end{figure}

By Theorem~\ref{Theo.ConvergenceRate}, it is interesting to consider initial values $(u_{30},v_{30})$ from outside the critical manifold. More precisely, we consider here the case $0\not =\varepsilon_{\mathsf{in}} = O(\varepsilon^\gamma)$ for $0\le \gamma <1$. A different dependency of the $L^\infty(0,T;L^2(\Omega))$ error on $\varepsilon$ is observed. Figure~\ref{fig:linf_inital_manifold} shows that the rates of $u_i^\varepsilon-u_i$, $i=1,2,3$, are of the order $O(\varepsilon)$, but the rate of $v_3^\varepsilon-v_3$  has the order $O(\varepsilon^\gamma)$. Moreover, one can also observe the vanishing of the parabolicity since the smaller $\varepsilon$, the faster rate from $O(\varepsilon)$ for $v_3^\varepsilon-v_3$.

\medskip

Varying the parameter $\gamma$, the effect of the initial layer is given in Table~\ref{Fig:RateOrders}. This is a numerical validation for the analytical proofs, showing as well the high accuracy of the numerical scheme.  
\begin{table}[h!]
    \centering
    \begin{tabular}{c|c|c|c}
   rate of    & $0\le \gamma<1/2$   &  $1/2\le \gamma<1$ & $\gamma\ge 1$   \\[0.5em]
  $u_i^\varepsilon-u_i$, $i=1,2,3$   & $O(\varepsilon^{1/2+\gamma})$   &  $O(\varepsilon)$ & $O(\varepsilon)$    \\
  $v_3^\varepsilon-v_3$   & $O(\varepsilon^{\gamma})$  & $O(\varepsilon^{\gamma})$  & $O(\varepsilon)$  \\
\end{tabular}
    \caption{The effect of the initial layer on the rates as $0\not =\varepsilon_{\mathsf{in}} = O(\varepsilon^\gamma)$ for $0\le \gamma <1$.}
    \label{Fig:RateOrders}
\end{table}

\subsection{Dynamics of the spatially independent system}

Spatial differences of the solution in Figure~\ref{fig:pde1} are smoothed quickly, and the system shows the oscillatory behaviour of the underlying ODE system. Therefore, we investigate further the underlying ODE system and compare the model with two preys and one predator to a competition system with only one prey. 

\medskip

The system under investigation models the dynamics of two competitive preys and one predator population. 
Additionally to the interaction by local reaction, chemotactic movement is included. 
Even in the spatially localised setting, without regarding any diffusive or chemotactic movement, the system dynamics differ from the classical two-population predator-prey model. 

\medskip

To justify from a modelling point of view why we analyse the three population models, we give numerical arguments for the different structures of the ordinary differential equation system 
\begin{align}
\label{System.U.ODE}
\left\{ \begin{array}{llll}
\partial_t u_1    &=&   \alpha_1u_1(1-u_1 - \beta_1  u_2) -  \dfrac{m_1 u_1}{\eta_1+u_1} u_3, \vspace*{0.05cm} \\
\partial_t u_2   &=&  \alpha_2u_2(1-u_2 - \beta_2 u_1) -  \dfrac{m_2 u_2}{\eta_2+u_2} u_3 , \vspace*{0.05cm} \\
\partial_t u_3   &=&  \left( \gamma_1 \dfrac{m_1 u_1}{\eta_1+u_1} + \gamma_2 \dfrac{m_2 u_2}{\eta_2+u_2} - k \right) u_3 - lu_3^2,
\end{array}
\right.
\end{align}
where $\alpha_1,\alpha_2$ are biotic potentials;
 $\beta_1,\beta_2$  are coefficients of inter-specific competition between two prey species;
 $m_1,m_2$ are predation coefficients;
 $\eta_1,\eta_2$ are half-saturation constants;
 $\gamma_1,\gamma_2$ are conversion factors;
 $k$ and $l$ are the natural death rates of the predator and the intra-specific competition among predators, respectively.
We compare the system's behaviour with a  
predator-prey model
\begin{align}
\label{System.pp.ODE}
\left\{ \begin{array}{llll}
\partial_t u_1    &=&   \alpha_1u_1(1-u_1) -  \dfrac{m_1}{\eta_1+u_1}u_1 u_3, \vspace*{0.05cm} \\
\partial_t u_3   &=&  \left( \gamma_1 \dfrac{m_1 u_1}{\eta_1+u_1} - k \right) u_3 - lu_3^2.
\end{array}
\right.
\end{align} 

The bifurcation diagrams show the (in-)stability of stationary states and provide information on oscillatory solutions. The bifurcation diagrams were implemented with the Julia package BifurcationKit.jl, \cite{veltz:hal-02902346}.

\medskip

\begin{figure}[htbp]
\begin{center}
\begin{subfigure}{.32\linewidth}
\centering
\includegraphics[width=\textwidth]{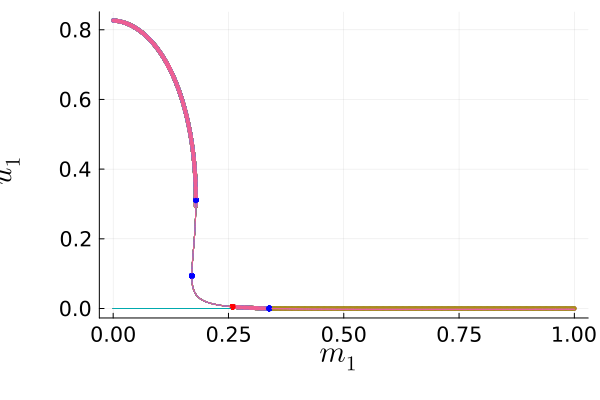}
\caption{$u_1$}
\label{fig:sub31}
\end{subfigure}%
\begin{subfigure}{.32\linewidth}
\centering
\includegraphics[width=\textwidth]{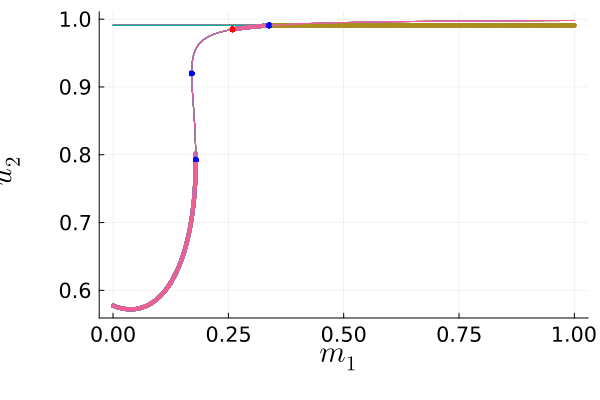}
\caption{$u_2$}
\label{fig:sub32}
\end{subfigure}%
\begin{subfigure}{.32\linewidth}
\centering
\includegraphics[width=\textwidth]{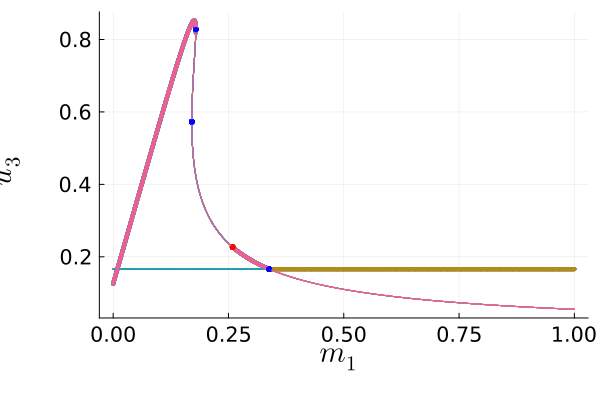}
\caption{$u_3$}
\label{fig:sub33}
\end{subfigure}%
\caption{Bifurcation diagram of the three-population-model \eqref{System.U.ODE}. The bifurcation parameter is $m_1$. Thin lines indicate unstable states, thick lines stable states. Dots indicate bifurcation points.}
\label{fig:1}
\end{center}
\end{figure}

Figure~\ref{fig:1} shows the bifurcation diagram for the full ODE model \eqref{System.U.ODE} depending on $m_1$. All other parameters are fixed.
In the parameter region where none of the stationary states is stable, the system shows oscillations, see Figure~\ref{fig:2}.

\begin{figure}[htbp]
\begin{center}
\begin{subfigure}{.48\linewidth}
\centering
\includegraphics[width=\textwidth]{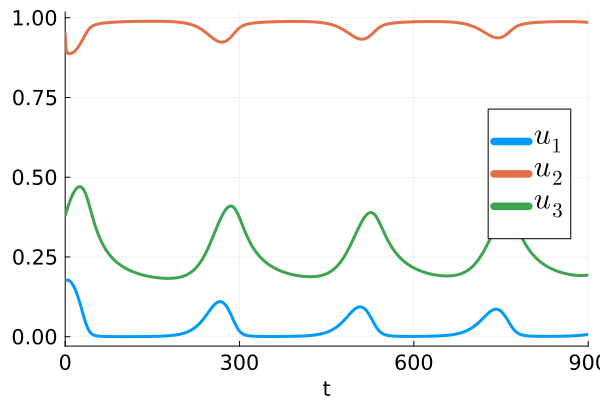}
\caption{short time interval}
\label{fig:sub21}
\end{subfigure}%
\begin{subfigure}{.48\linewidth}
\centering
\includegraphics[width=\textwidth]{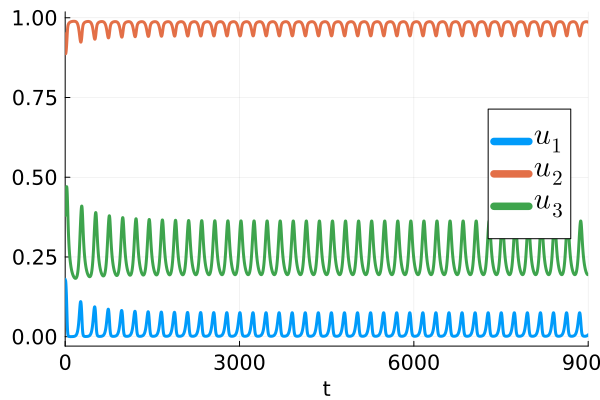}
\caption{large time interval}
\label{fig:sub22}
\end{subfigure}%
\caption{Dynamics of the three-population-model \eqref{System.U.ODE} for $\beta_1$ in the unstable region.}
\label{fig:2}
\end{center}
\end{figure}

\medskip
 
The two bifurcation diagrams are relevant for comparing with the reduced ODE system~\eqref{System.pp.ODE}.
In the one prey - one predator system \eqref{System.pp.ODE}, we see the change of the system behaviour from stable stationary states for small $m_1$ to oscillating solution for larger $m_1$, see Figure~\ref{fig:4}. For very small $m_1$, only the prey population survives, and the predator population becomes extinct.  

\begin{figure}[htbp]
\begin{center}
\begin{subfigure}{.48\linewidth}
\centering
\includegraphics[width=\textwidth]{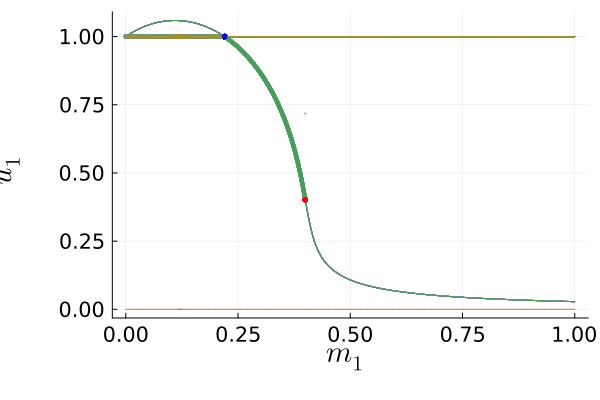}
\caption{prey $u_1$}
\label{fig:sub41}
\end{subfigure}%
\begin{subfigure}{.48\linewidth}
\centering
\includegraphics[width=\textwidth]{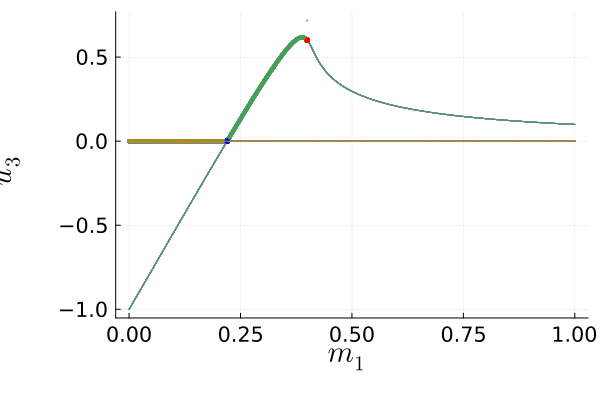}
\caption{predator $u_3$}
\label{fig:sub42}
\end{subfigure}%
\caption{Bifurcation diagram of the one prey - one predator system \eqref{System.pp.ODE} with the bifurcation parameter $m_1$. }
\label{fig:4}
\end{center}
\end{figure}

The oscillatory behaviour of the predator-prey system for parameter values $m_1$ larger (or equal) to the red Hopf bifurcation point is shown in Figure~\ref{fig:5}.
\begin{figure}[htbp]
\begin{center}
\includegraphics[width=0.5\textwidth]{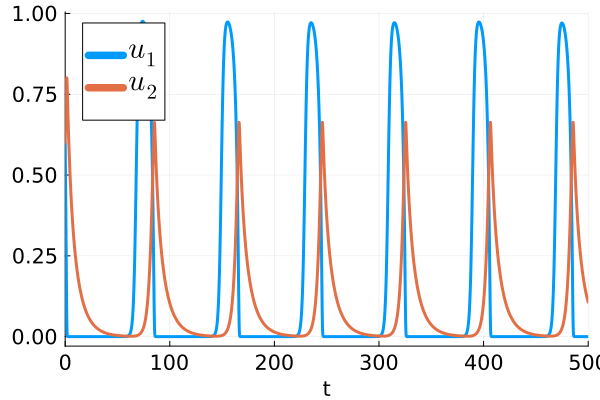}
\caption{Oscillatory dynamics of the one prey - one predator system \eqref{System.pp.ODE}.  }
\label{fig:5}
\end{center}
\end{figure}
The predator-prey model with only one prey population shows that for small parameters, the predator extinction occurs for small $m_1$, then stationary co-existence and the well-known oscillations for larger $m_1$. 

\medskip

The two-prey-one-predator model under investigation in this paper shows depending on the parameter $m_1$ a different behaviour: for small $m_1$, all three populations co-exist, for some medium parameter values $m_1$ the system oscillates, and for large $m_1$ the system becomes a one-prey-one-predator system due to extinction of $u_2$. 
The oscillations show a time-delay of population maxima of the two prey. Therefore, combining the two prey populations into one prey population is not meaningful and reduces the system's complexity.  

\medskip

Consequently, the ordinary differential equation setting of the studied model shows a richer behaviour than the two included two-population models. 
The space-depending system in Figure~\ref{fig:pde1} shows the same behaviour due to a fast levelling of any spatial difference. 
The study of a chemotaxis system with two prey populations and one predator population is therefore an extension of models for one prey population.

\section{Further comments}
\label{Sec.Extention-Outlook} 

Our results are obtained for the system with only parabolic equation for $v_3^\varepsilon$. However, the analysis clearly works with the fully parabolic system, i.e. the system in which equations for all chemical concentrations are parabolic, given as follows 
 \begin{align}
\left\{ \begin{array}{llll}
\hspace*{0.15cm} \partial_t u_1^\varepsilon  - d_1 \Delta u_1^\varepsilon  - \chi_1 \di ( u_1^\varepsilon \nabla v_3^\varepsilon) &\hspace*{-0.2cm}=\hspace*{-0.2cm}& f_1( u_1^\varepsilon, u_2^\varepsilon, u_3^\varepsilon) ,  \\
\hspace*{0.15cm} \partial_t u_2^\varepsilon - d_2\Delta u_2^\varepsilon  - \chi_2 \di ( u_2^\varepsilon \nabla  v_3^\varepsilon) &\hspace*{-0.2cm}=\hspace*{-0.2cm}& f_2( u_1^\varepsilon, u_2^\varepsilon, u_3^\varepsilon)  , \\
\hspace*{0.15cm} \partial_t u_3^\varepsilon -  d_3\Delta u_3^\varepsilon +  \sum_{i=1}^2 \chi_{3i} \di ( u_3^\varepsilon \nabla v_i^\varepsilon)  &\hspace*{-0.2cm}=\hspace*{-0.2cm}&  f_3( u_1^\varepsilon, u_2^\varepsilon, u_3^\varepsilon) ,  \\ 
\varepsilon \partial_t v_i^\varepsilon - \lambda_i \Delta v_i^\varepsilon + \mu_i v_i^\varepsilon = \zeta_i u_i^\varepsilon, && i=1,2,3,  
\end{array}
\right.
\end{align}
equipped with the boundary - initial conditions 
\begin{align}
(\nabla u_i^\varepsilon \cdot \nu, \nabla v_i^\varepsilon \cdot \nu)|_{\Gamma_\infty} = 0, \quad 
(u_i^\varepsilon(0),v_i^\varepsilon(0))|_{\Omega}=(u_{i0},v_{30}), \quad i=1,2,3.
\end{align} 
In this case, we can see from Theorems~\ref{Theo.GlobalClassicalSol}-\ref{Theo.ConvergenceRate} the following slight differences
\begin{itemize}
\item The feedback argument from prey to predator as Lemma~\ref{Lem.Feedback12to3} is not necessary for the global existence; 
\item The components $v_1^\varepsilon,v_2^\varepsilon,v_3^\varepsilon$ are uniformly bounded in $L^\infty(Q_T) \cap L^q(0,T;W^{2,q}(\Omega))$; and 
\item The $L^\infty$-in-time convergence rates $\widehat v^{\,\varepsilon}_1,\widehat v^{\,\varepsilon}_2,\widehat v^{\,\varepsilon}_3$ will be affected by the initial layer corresponding to the critical manifold $$\big\{(u_{i0},v_{i0})\in L^2(\Omega)^3\times H^2(\Omega)^3: \,  \lambda_i \Delta v_i - \mu_i v_i + \zeta_i u_i = 0, i=1,2,3 \big\}.$$
\end{itemize}

Since our analysis is not restricted to only positive chemotaxis coefficients, fast signal diffusion limits and $L^\infty$-in-time convergence rates for the following two-dimensional chemotaxis system can be studied similarly   
\begin{align*}
\left\{ \begin{array}{llll}
\partial_t u_i^\varepsilon  - d_i \Delta u_i^\varepsilon  +   \displaystyle \sum_{j=1}^k \chi_{ij} \di ( u_i^\varepsilon \nabla v_j^\varepsilon) = f_i( u_1^\varepsilon,\dots, u_k^\varepsilon) ,   \\
\varepsilon_i \partial_t v_i^\varepsilon - \lambda_i \Delta v_i^\varepsilon + \mu_i v_i^\varepsilon = \zeta_i u_i^\varepsilon,    
\end{array}
\right.  
\end{align*}
$i=1,\dots,k$,
which is subjected to the homogeneous Neumann boundary conditions and given smooth initial data. Here, $d_i>0$, $\chi_{ij}\in\R$, $\lambda_i,\mu_i,\zeta_i>0$ for $1\le i,j\le k$. The parameters  $\varepsilon_1,\dots,\varepsilon_k\in \{0;\,\varepsilon\}$, with $0<\varepsilon\ll1$, satisfy that there exists at least one different from $0$. The species kinetics are generally of the competitive or logistic types.

\vspace{0.4cm}

\noindent {\Large \bf Acknowledgement}

\medskip

\noindent The authors would like to thank Bao Quoc Tang for his valuable
comments on the paper and Simon-Christian Klein for discussing the numerical schemes. The second author is funded by the FWF project “Quasi-steady-state approximation for PDE”, number I-5213. The third author is partially supported by NSFC Grants No. 12271227 and China Scholarship Council (Contract No. 202206180025).

\vspace{0.3cm}

\noindent {\Large \bf Conflict of Interest}   The authors declare that they have no conflict of interest.

\vspace{0.3cm}

\noindent {\Large \bf Data Availability} 
 Data sharing not applicable.

\appendix

\vspace*{0.5cm}

\noindent {\bf \Large Appendices}

\vspace*{0.5cm}

\noindent This part is to recall or slightly improve well-known results related to heat semigroup, heat regularisation, and maximal regularity, where we assume  that $\Omega$ is a bounded domain in $\mathbb{R}^N$ with sufficiently smooth (such as $C^{2+\alpha}$ for some $\alpha>0$) boundary. 

\section{Neumann heat semigroup}

For $\lambda,\mu>0$ and $1<p<\infty$, the sectorial operator $-\lambda\Delta + \mu I$ defined on $\{v\in W^{2,p}(\Omega): \nabla v \cdot \nu =0 \text{ on } \Gamma \}$ has a countable sequence of eigenvalues with the smallest one is $\mu>0$. Therefore, it generates an analytic semigroup $\{e^{t(\lambda\Delta-\mu I)}\}_{t\ge 0}$ on $L^p(\Omega)$ such that  
\begin{align}
\|(\lambda\Delta-\mu I)^\beta e^{t(\lambda\Delta-\mu I)}f\|_{L^p(\Omega)} \le C e^{-\omega t} t^{-\beta} \|f\|_{L^p(\Omega)} , \quad t>0,  
\label{HeatSemigroup.Contraction}
\end{align}
for some $\omega >0$ and all $\beta\ge 0$, see \cite[Lemma 2.1]{horstmann2005boundedness}. In particular,  if $\beta=0$, we can take $\omega = \mu$ and $C=1$ (i.e. we have a contraction semigroup), see \cite[Theorem 13.4]{amann1984existence}. Note that $e^{t(\lambda\Delta-\mu I)}$ is not commutative with the divergence $\nabla \cdot$. A combination of them is estimated as follows, where, for the purpose of consistency, we will use the same $\omega$ not only in Lemma~\ref{Lem.HeatCombineDiv} but also throughout the paper.

\begin{lemma}[{\cite[Lemma 2.1]{horstmann2005boundedness}}]
\label{Lem.HeatCombineDiv}
Let  $1<p<\infty$. Then, for all $\kappa>0$, 
\begin{align*}
\|(-\lambda\Delta + \mu I)^{\beta}e^{t\lambda\Delta} \nabla \cdot v \|_{L^p(\Omega)} \le C_\kappa  e^{-\omega t} t^{-\beta-\frac{1}{2}-\kappa} \|v\|_{L^p(\Omega)}, \quad t>0,
\end{align*}
for some $\omega>0$ and all $v \in C_0^\infty(\Omega)$, $\beta\ge 0$. The operator $(-\lambda \Delta + \mu I)^{k}e^{t \lambda\Delta } \nabla \cdot$ consequently admits an only extension (with the same notation) to the whole space $L^p(\Omega)$. 
\end{lemma}

The Neumann heat semigroup $\{e^{t(\lambda\Delta-\mu I)}\}_{t\ge 0}$ on $L^p(\Omega)$ can also be used in expressing solutions to some elliptic equations. Indeed, let us consider the elliptic problem  
\begin{align}
-\lambda\Delta v + \mu  v = f  \text{ in } \Omega,  \quad \text{and}\quad 
 \nabla v \cdot \nu  = 0  \text{ on } \Gamma.
\label{EllipticEquation}
\end{align}
Since $\lambda,\mu>0$, all eigenvalues of the Neumann operator $-\lambda \Delta + \mu I$ are strictly positive, and will be $ \{b_i\}_{i\ge 1}$, $b_i:=\lambda a_i + \mu$, if the eigenvalues of the Neumann Laplacian $-\Delta$ are denoted by $\{a_i\}_{i \ge 1}$.  We, therefore, can follow  from the identity
\begin{align}
\frac{1}{b_i} = \int_0^\infty e^{-b_i s} \, ds, \quad \text{for } \lambda >0  
\label{EigenvalueIdentity}
\end{align}
that 
\begin{align}
v(x) & = (-\lambda\Delta + \mu I)^{-1} f(x)  = \left( \int_0^\infty e^{s(\lambda_3 \Delta - \mu  I)} \, ds \right) f(x) .
\label{Expression.v}
\end{align}
  
\section{Heat regularisation} 

To obtain optimal regularity of solutions to the heat equation, we recall here the so called \textit{heat regularisation}, see \cite[Theorem 9.1, Chapter IV]{ladyvzenskaja1988linear}.

\begin{lemma}[Heat regularisation]  \label{Lem.HeatRegularisation}  
Let $T>0$, $1<p<\infty$. Assume that $f\in L^p( Q_T)$ and $v_0\in W^{2-2/p,p}(\Omega)$ with the compatibility condition $\nabla v_{0} \cdot \nu=0$ on $\Gamma$. If $v$ is the weak solution to
	  \begin{equation}
	  \partial_t v - \lambda \Delta v + \mu v = f \text{ in }   Q_T, \quad 
		\nabla v \cdot \nu  = 0 \text{ on } \Gamma_T, \quad 
		v(0) = v_0  \text{ in } \Omega, 
		\label{HeatEquation}
	\end{equation}
for $\lambda>0$, $\mu\ge 0$, then 
\begin{align}
	\|v\|_{L^{q}( Q_T  )} + \|\nabla v\|_{L^{r}( Q_T  )} + \||\partial_t v|+|\Delta v|\|_{L^{p}( Q_T  )}
	 \le C \left( \|f\| _{L^p(\Omega_{T})} + \| v_0\|_{W^{2-\frac{2}{p},p}(\Omega)} \right), \nonumber 
\end{align}
where
	\begin{align*}
	q=\left\{ \begin{array}{lllllll}
	 \dfrac{(N+2)p}{N+2-2p}  & \text{if }  p<\frac{N+2}{2}, \vspace*{0.15cm}\\
	\in[1,\infty) \text{ arbitrary}  & \text{if }  p=\frac{N+2}{2}, \vspace*{0.15cm}\\
	\infty   & \text{if }  p>\frac{N+2}{2},
	\end{array}
	\right. 
	\quad 
	r=\left\{ \begin{array}{lllllll}
	\dfrac{(N+2)p}{N+2-p}  & \text{if }  p<N+2, \vspace*{0.15cm}\\
	\in[1,\infty)\text{ arbitrary}  & \text{if }  p=N+2, \vspace*{0.15cm}\\
	\infty   & \text{if }  p>N+2 ,
	\end{array}
	\right.
	\end{align*}
and $C$ depends only on $\lambda,\mu, p,N,\Omega, T$,  and remains bounded for 	finite values of $T>0$.
\end{lemma}

\begin{remark} 
\label{Remark.Lem.HeatRegularisation} 
In Lemma~\ref{Lem.HeatRegularisation}, the dependence of the constant $C$ on $T$ can be removed by applying Theorem 2.3 in \cite{giga1991abstract}.
\end{remark}

\section{$L^p$-maximal regularity with independent-of-$p$ constants}

In this part, we present $L^p$-maximal regularity, where we emphasise that maximal \textit{regularity constants do not depend on $p$}. This independence plays an important role in the analysis of fast signal diffusion limits.  

\begin{lemma}[Elliptic maximal regularity]
\label{Lem.MaximalRegularity}  Let $\lambda,\mu>0$, $1< p<\infty$, and $f\in L^p(\Omega)$. Then the solution to the problem
\eqref{EllipticEquation} 
satisfies the following estimate
\begin{align*}
\|\Delta v\|_{L^{p}(\Omega)} \le C^{\mathsf{EM}} \|f\|_{L^p(\Omega)},
\end{align*}
where $C^{\mathsf{EM}}$ depends on $\lambda,\mu,\Omega,N$, but not on $p$. 
\end{lemma}

\begin{proof} By the same argument in Remark~\ref{Remark.Theo.FastSignalDiffLimit}, we have 
$ v(x)= \int_0^\infty e^{s(\lambda\Delta-\mu I)} f(x) ds$. 
Taking into account the fact that the heat semigroup associated with the homogeneous Neumann boundary condition is a contraction semigroup on $L^p(\Omega)$, we have 
\begin{align*}
\|\Delta v\|_{L^p(\Omega)} & = \left\|  - \frac{1}{\lambda}f + \frac{\mu}{\lambda} \left( \int_0^\infty e^{s(\lambda\Delta-\mu I)} \, ds \right) f \right\|_{L^p(\Omega)} \\
& \le \frac{1}{\lambda}\|f\|_{L^p(\Omega)} + \frac{\mu}{\lambda}   \left( \int_0^\infty  e^{-\omega s} \, ds \right) \|f\|_{L^p(\Omega)}, 
\end{align*}
where the improper integral is finite.
\end{proof}

In the following lemma, we improve the heat regularisation given in Lemma~\ref{Lem.HeatRegularisation} in the sense that the dependence of $C$ on both $T$ and $p$ will be removed.

\begin{lemma} 
\label{Lem.ParaMaximalRegularity}
If $v$ is the solution to \eqref{HeatEquation} with $\lambda>0$, $\mu\ge0$, then   
\begin{align}
\|\Delta v\|_{L^p(Q_T)} \le C_{\lambda,\mu,p_0}^{\mathsf{PM}} \left( \|v_0\|_{W^{2,p}(\Omega)} +  \|f\|_{L^p(Q_T)} \right), \quad 2\le  p\le p_0, 
\label{Lem.ParaMaximalRegularity.State1} 
\end{align} 
for $p_0<\infty$, where $C_{\lambda,\mu,p_0}^{\mathsf{PM}}$ depends on $\lambda,\mu,\Omega,N$, but not on $p,T$. Moreover, the constant will be $C_{\lambda,\mu,p}^{\mathsf{PM}}$, i.e. it generally depends on $p$ if we consider $1<  p <\infty$. 
\end{lemma} 

\begin{proof} We split $v$ into the sum of $\widetilde v$ and $\widehat v$, which are the solutions in the cases $v_0=0$ and $f=0$, respectively. Thanks to Lemma \cite[Theorem 1]{lamberton1987equations}, there exists an optimal constant $C_{\lambda,\mu,p}$, which may depend on $\Omega,N$ but not on $T$, such that 
\begin{align}
\|\Delta \widetilde v\|_{L^p(Q_T)} \le C_{\lambda,\mu,p} \|f\|_{L^p(Q_T)}. 
\label{Lem.ParaMaximalRegularity.Proof1}
\end{align} 
On the other hand, it is straightforward to see  $C_{\lambda,\mu,2} \le 1/\lambda$ due to multiplying the equation for $\widetilde v$ by $-\Delta \widetilde v$. Specifically, we have 
\begin{align}
\|\Delta \widetilde v \|_{L^2(Q_T)} \le \frac{1}{\lambda} \|f\|_{L^2(Q_T)}. 
\label{Lem.ParaMaximalRegularity.Proof2}
\end{align}
Let us use the same idea in \cite[Lemma 3.2]{canizo2014improved} with  an interpolation between \eqref{Lem.ParaMaximalRegularity.Proof1} for $p=p_0$ and \eqref{Lem.ParaMaximalRegularity.Proof2}, which gives 
$C_{\lambda,\mu,p}  \le \lambda^{-s} (C_{\lambda,\mu,p_0})^{1-s}$,   $1/p=s/2+(1-s)/p_0$. We get  
\begin{align*}
C_{\lambda,\mu,p}  \le \lambda^{\frac{2}{p_0-2}(1-\frac{p_0}{p})} (C_{\lambda,\mu,p_0})^{\frac{p_0}{p_0-2}(1-\frac{2}{p})}  \le \frac{\max(1; C_{\lambda,\mu,p_0})}{\min(1;\lambda)}   .
\end{align*}
On the other hand, by the estimate \eqref{HeatSemigroup.Contraction},
\begin{align}
\|\Delta \widehat v \|_{L^p(Q_T)} = \|e^{t(\lambda \Delta - \mu I)} \Delta v_0 \|_{L^p(Q_T)} \le \|\Delta v_0 \|_{L^p(Q_T)}.
\label{Lem.ParaMaximalRegularity.Proof3}
\end{align} 
We obtain the estimate \eqref{Lem.ParaMaximalRegularity.State1} by combining \eqref{Lem.ParaMaximalRegularity.Proof1} and  \eqref{Lem.ParaMaximalRegularity.Proof3}, where 
\begin{align*}
C_{\lambda,\mu,p_0}^{\mathsf{PM}}:= 1 + \frac{\max(1; C_{\lambda,\mu,p_0})}{\min(1;\lambda)}. 
\end{align*}
The conclusion for the case $1<p<\infty$ is clear without the use of the interpolation.
\end{proof}


\end{document}